\documentclass[11pt,letter]{amsart}
\usepackage[utf8]{inputenc}
\usepackage{amsmath, calligra, mathrsfs}
\usepackage{enumerate, color, hyperref}
\usepackage{bbm}

\DeclareMathOperator{\Hom}{\mathscr{H}\text{\kern -3pt {\calligra\large om}}}

\newtheorem{thm}{Theorem}[section]
\newtheorem{prop}[thm]{Proposition}
\newtheorem{cor}[thm]{Corollary}
\newtheorem{lemma}[thm]{Lemma}

\newtheorem{hyp}[thm]{Hypothesis}

\theoremstyle{definition}
\newtheorem{rmk}[thm]{Remark}
\newtheorem{defi}[thm]{Definition}

\usepackage{amsfonts}
\usepackage{amsmath}
\usepackage{amssymb}
\usepackage{graphics}
\usepackage{amscd}
\usepackage[all]{xy}
\usepackage{enumitem}
\usepackage{todonotes}
\usepackage{stackrel}

\newcommand{\beqn}{\begin{eqnarray*}}
\newcommand{\eeqn}{\end{eqnarray*}}
\newcommand{\beqa}{\begin{eqnarray}}
\newcommand{\eeqa}{\end{eqnarray}}

\DeclareMathOperator{\N}{\mathbb{N}}
\DeclareMathOperator{\Q}{\mathbb{Q}}
\DeclareMathOperator{\R}{\mathbb{R}}

\DeclareMathOperator{\Z}{\mathbb{Z}}
\DeclareMathOperator{\V}{\mathbb{V}}

\DeclareMathOperator{\G}{\mathbb{G}}
\DeclareMathOperator{\C}{\mathbb{C}}
\DeclareMathOperator{\A}{\mathbb{A}}
\newcommand{\Homo}{{\rm Hom}}

\DeclareMathOperator{\GL}{\mathrm{GL}}
\newcommand{\cO}{\mathcal{O}}
\newcommand{\SO}{\mathrm{SO}}

\newcommand{\cP}{\mathcal{P}}
\newcommand{\dP}{\mathfrak{p}}
\newcommand{\cX}{\mathcal{X}}
\newcommand{\dX}{\mathfrak{X}}
\newcommand{\cD}{\mathcal{D}}
\newcommand{\cN}{\mathfrak{n}}
\newcommand{\cC}{\mathfrak{c}}

\newcommand{\dH}{\mathfrak{H}}
\newcommand{\cG}{\mathcal{G}}
\newcommand{\cI}{\mathcal{I}}
\newcommand{\Spec}{\mathrm{Spec}}
\newcommand{\Spf}{\mathrm{Spf}}

\newcommand{\cA}{\mathcal{A}}
\newcommand{\cM}{\mathfrak{m}}
\newcommand{\cF}{\mathcal{F}}
\newcommand{\cY}{\mathfrak{Y}}
\newcommand{\cR}{\mathcal{R}}
\newcommand{\cH}{\mathcal{H}}
\newcommand{\cW}{\mathcal{W}}
\newcommand{\dW}{\mathfrak{W}}
\newcommand{\Sym}{\mathrm{Sym}}
\newcommand{\Pic}{\mathrm{Pic}}

\newcommand{\M}{\mathrm{M}}

\makeatletter
\newcommand\figcaption{\def\@captype{figure}\caption}
\newcommand\tabcaption{\def\@captype{table}\caption}
\makeatother

\usepackage[left=1in, right=1in, top=1in, bottom=1in]{geometry}

\title[Triple product $p$-adic $L$-functions]{Triple product $p$-adic $L$-functions for Shimura curves over totally real number fields}
	
\author{Daniel Barrera-Salazar \and Santiago Molina}

\newcommand{\Addresses}{{
		\bigskip
		\footnotesize
		
		Daniel Barrera-Salazar; Universidad de Santiago de Chile\\ Alameda 3363, Estación Central, Santiago, Chil\par\nopagebreak
		\texttt{danielbarreras@hotmail.com}
				
		\medskip
		Santiago Molina; Universitat Polit\`{e}cnica de Catalunya\\Campus Nord, Calle Jordi Girona, 1-3, 08034 Barcelona, Spain\par\nopagebreak
		\texttt{santiago-molina@upc.edu}
		
}}

\date{\today}

\begin{document}
	
	\begin{abstract} Let $F$ be a totally real number field. Using a recent geometric approach developed by Andreatta and Iovita we construct several variables $p$-adic families of finite slope quaternionic automorphic forms over $F$. It is achieved by interpolating the modular sheaves defined over some auxiliary unitary Shimura curves.  
	
	Secondly, we attach $p$-adic $L$-functions to triples of ordinary $p$-adic families of quaternionic automorphic eigenforms. This is done by relating trilinear periods to some trilinear products over unitary Shimura curves which can be interpolated adapting the work of Liu-Zhang-Zhang to our families. 
	\end{abstract}
	
  \maketitle
  	
\tableofcontents

\section{Introduction}	







\subsection{Arithmetic, $p$-adic $L$-functions and $p$-adic families}

 Several recent works with important arithmetic applications  uses in a crucial way $p$-adic $L$-functions (see  for example \cite{DR17}, \cite{BDR15}, \cite{SU14}, \cite{JSW15}). 
In contrast with its complex counterpart the theory of $p$-adic $L$-functions is far from being well established. 
 Thus depending of the context, their construction is performed with the available technology. This work is devoted to the construction of the so called \emph{triple product $p$-adic $L$-functions}.

 During the nineties Kato obtained deep results on the Birch and Swinnerton-Dyer conjecture in rank $0$ for twists of elliptic curves over $\Q$ by Dirichlet characters. More recently, Bertolini-Darmon-Rotger in \cite{BDR15} and Darmon-Rotger in \cite{DR17} developed analogous methods to treat twists by certain Artin representations of dimension 2 and 4. In \cite{LLZ16}, \cite{KLZ17}, \cite{LZ} and \cite{BSV2}  these methods were extended in different directions, as for example bounding certain Selmer groups and treating finite slope settings. We could say that in these situations a prominent role is played by the \emph{unbalanced} $p$-adic $L$-function attached to a triple of $p$-adic families of modular forms. Such $p$-adic $L$-functions were constructed in \cite{harris-tilouine01}, \cite{DR14}, \cite{hsieh18} in the ordinary case and in \cite{AI17} for Coleman Families. In other hand \emph{balanced} triple $p$-adic $L$-functions had been constructed in \cite{hsieh18} and in \cite{greenberg-seveso}.

This work grew up from the aim to generalize the methods mentioned above to totally real number fields. In the present paper we furnish the $p$-adic $L$-functions that come into play in the ordinary setting. Our main results are:  
\begin{itemize}
	\item[(i)] the construction of several variables $p$-adic families of finite slope quaternionic automorphic forms for over totally real number fields;
	\item[(ii)] the construction of triple product $p$-adic $L$-functions for ordinary families obtained in (i).
\end{itemize}

To achieve (i) we use a geometrical approach (developed in \cite{AI17}) using the theory of overconvergent modular forms in the generality needed. For (ii) we naturally work in a hybrid situation: unbalanced in a single component of the weight and balanced in the rest of the components. Our strategy mix the two approaches used in \cite{AI17} and \cite{greenberg-seveso} and crucially adapts to our situation ideas and constructions performed in \cite{liu-zhang-zhang17}.

\subsection{Main results} Let $F$ be a totally real number field of degree $d= [F:\Q]$. We denote by $\Sigma_F$ the set of real embeddings and fix $\tau_0 \in \Sigma_F$.  Let $p> 2$ be a prime number and let $\mathfrak{p}_0$ the prime over $p$ associated to $\tau_0$ under a fixed embedding $\iota_p:\bar\Q\hookrightarrow\C_p$. \emph{We suppose $F$ is unramified at $p$ and the inertia degree of $\mathfrak{p}_0$ over $p$ is $1$.}

Let $B$ be a quaternion algebra over $F$ split at $\tau_0$ and any prime over $p$, and ramified at any $\tau \in \Sigma_F\setminus \{\tau_0\}$. As already mentioned, one of our goals is the construction of $p$-adic families of automorphic forms on $(B\otimes\A)^\times$ using geometrical tools. One of the main obstructions to perform this task working directly on Shimura curves of $B$ is the lack of an adequate moduli problem. To remedy this issue we work with $D= B\otimes_F E$ for some CM extension $E$ of $F$ on which each prime of $F$ over $p$ splits. Let $X$ be the unitary Shimura curve attached to $D$ of level prime to $p$ and $\mathrm{disc}(B)$ (see \S\ref{ss:unitary Shimura curves}) and  by $\pi: A\rightarrow X$ its universal abelian variety (of dimension $4d$).  For each $\underline{k}\in \mathbb{N}[\Sigma_F]$ a precise piece of the sheaf of invariant differentials of $A$ produces a modular sheaf $\omega^{\underline{k}}$ that gives rise to \emph{modular  forms for $D$}.

The works devoted to develop the theory of $p$-adic families for unitary Shimura curves (\cite{brasca13}, \cite{kassaei04}, \cite{ding17}) construct essentially $1$-dimensional $p$-adic families. In fact those works only treat the $p$-adic variation of powers of the invertible subsheaf of $\omega$ attached to $\tau_0$. Nevertheless, for applications it would be better to have more variation, thus our first task is to provide a  more general theory of $p$-adic families in this context. The complexity to perform this was reflected for example in the fact that the rank of $\omega^{\underline{k}}$ growths with $\underline{k}$. Nowadays we have enough technology to perform this task.

The \emph{weight space for $D$}  is the $d$-dimensional adic space, denoted $\mathcal{W}$, attached to the complete group algebra $\Z_p[[(\cO_{F}\otimes\Z_p)^\times]]$. The \emph{weight space for $B$} is the $d+1$-dimensional adic space attached to $\Z_p[[(\cO_{F}\otimes\Z_p)^\times\times\Z_p^{\times}]]$ and is denoted $\mathcal{W}^G$  in the text. For each $n> 0$ we consider certain open subspaces $\mathcal{W}_n\subset \mathcal{W}$  and $\mathcal{W}_n^G\subset\mathcal{W}^G$ (see \S\ref{ss:weight space} and \S\ref{ss:weight space for G}).
  
 Let $\cX$ be denote the adic analytic space attached to $X$. For $r> 0$ we denote by $\cX_{r}$ the strict neighborhood of the $\mathfrak{p}_0$-ordinary locus of $\cX$ where the universal abelian variety has a $\mathfrak{p}_0$-canonical subgroup of order $\leq r$. We have (see \S\ref{ss:overconvergent modular sheaves}, \S\ref{ss:eigenvarieties} and proposition \ref{propclassty}):

\begin{thm}\label{t:main thm 1} Let $n \leq r$. 
 \begin{itemize}
\item[(i)] there exist a sheaf of Banach modules $\cF_n$ over $\cX_{r}\times \mathcal{W}_n$ such that for each classical weight $\underline{k} \in \mathbb{N}[\Sigma_F]$ the map $(\mathrm{id}, \underline{k}): \cX_{r} \rightarrow \cX_{r}\times \mathcal{W}_n$ induces a natural embedding $\omega^{\underline{k}}|_{\cX_{r}} \subset (\mathrm{id}, \underline{k})^{\ast}(\cF_n)$ of Banach sheaves over $\cX_{r}$.
\item[(ii)] There exists an adic space $\mathcal{E}_n$ equidimensional of dimension $d+1$ and endowed with a locally free and without torsion map $w: \mathcal{E}_n\rightarrow \mathcal{W}_n^G$. Moreover, $\mathcal{E}_n$ parametrizes systems of Hecke  eigenvalues appearing in the space of automorphic forms for $B$.
\end{itemize}
\end{thm}

The construction of the sheaves $\cF_n$ is carried out using a slight modification of the machinery of \emph{formal vector bundles with marked sections} introduced in \cite{AI17}. Then we exploit the description of automorphic forms for $B$ in terms of modular forms for $D$. Thus using the sheaves $\cF_n$ (more precisely the dual of them) we produce the module of \emph{$p$-adic families of locally analytic overconvergent automorphic forms on $(B\otimes\A_F)^\times$}. Such module is projective and the usual Hecke operator $U_p$ acts compactly on it. Using the theory developed in \cite[appendix B]{AIP-halo} we obtain the eigenvariety $\mathcal{E}_n$.

Now we explain our result on triple product $p$-adic $L$-functions. We say that a triple $(\underline{k}_1,\underline{k}_2, \underline{k}_3) \in \Z[\Sigma_F]^3$ is \emph{unbalanced at $\tau_0$ with dominant weight $\underline{k}_3$} if  $k_{3, \tau_0}\geq k_{1, \tau_0}+k_{2, \tau_0}$, $k_{1, \tau}+k_{2, \tau}+k_{3, \tau}$ is even for each $\tau \in \Sigma_F$ and $(k_{1, \tau}, k_{2, \tau}, k_{3, \tau})$ is balanced for each $\tau \neq \tau_0$ (see definition \ref{d:unbalanced weights}). The \emph{interpolation region} for our $p$-adic $L$-functions is the set $S_{3}$ of triples $((\underline{k}_1, \nu_1), (\underline{k}_2, \nu_2), (\underline{k}_3, \nu_3)) \in (\Z[\Sigma_F]\times \Z)^3$ such that  i) $(\underline{k}_1,\underline{k}_2, \underline{k}_3)$ is unbalanced at $\tau_0$, ii) $k_{\tau, i}> 0$ for $i=1, 2, 3$ and $\tau \in \Sigma_F$ and iii) $\nu_3= \nu_1+ \nu_2$.

Let $\mu_1, \mu_2, \mu_3$ be three ordinary eigenfamilies of automorphic forms on $(B\otimes\A_F)^\times$ for $B$ and we denote by $\Lambda_1$, $\Lambda_2$ and $\Lambda_3$ the rings over which they are defined respectively. Let  $x, y, z$ be a triple of classical points corresponding to $((\underline{k}_1, \nu_1), (\underline{k}_2, \nu_2), (\underline{k}_3, \nu_3)) \in S_3$. We denote by $\pi_x$ the automorphic representation of $(B\otimes\A_F)^\times$ generated by the automorphic form obtained from the specialization of $\mu_1$ at $x$, and $\Pi_x$ the corresponding cuspidal automorphic representation of $\mathrm{GL}_2(\A_F)$. Moreover, we denote by $\alpha_{x}^{\dP}$ and $\beta_{x}^{\dP}$ the roots of the Hecke polynomial at $\mathfrak{p}$. In the same way we obtain $\pi_y$, $\Pi_y$, $\alpha_{y}^{\dP}$, $\beta_{y}^{\dP}$, $\pi_z$, $\Pi_z$, $\alpha_{z}^{\dP}$, $\beta_{z}^{\dP}$. We denote by $\mu_z^\circ$ the newform of $\pi_z$. 

We have (see lemma \ref{l:construction} and theorem \ref{t:interpolation} for more details):
\begin{thm}\label{t:main theorem 2} There exists $\mathcal{L}_p(\mu_1,\mu_2,\mu_3)\in \Lambda_1\hat{\otimes}\Lambda_2\hat{\otimes}\mathrm{Frac}(\Lambda_3)$ such that for each classical point $(x, y, z)$ corresponding to a triple $((\underline{k}_1, \nu_1), (\underline{k}_2, \nu_2), (\underline{k}_3, \nu_3)) \in S_3$ we have:
$$\mathcal{L}_p(\mu_1,\mu_2,\mu_3)(x, y, z)= K\cdot\left(\prod_{\dP\mid p}\frac{\mathcal{E}_\dP(x,y,z)}{\mathcal{E}_{\dP,1}(z)}\right)\cdot\frac{L\left(\frac{1-\nu_1-\nu_2-\nu_3}{2},\Pi_{x}\otimes\Pi_{y}\otimes\Pi_{z}\right)^{\frac{1}{2}}}{\langle\mu_{z}^\circ,\mu_{z}^\circ\rangle},$$
where  $K$ is a non-zero constant depending of $(x, y, z)$, $\mathcal{E}_\dP(x,y,z)=$
\[
\left\{\begin{array}{lc}
\mbox{\small$(1-\beta_{x}^{\dP}\beta_{y}^{\dP}\alpha_{z}^{\dP}\varpi_{\dP}^{-\underline{m}_{\dP}-\underline{2}})(1-\alpha_{x}^{\dP}\beta_{y}^{\dP}\beta_{z}^{\dP}\varpi_{\dP}^{-\underline{m}_{\dP}-\underline{2}})(1-\beta_{x}^{\dP}\alpha_{y}^{\dP}\beta_{z}^{\dP}\varpi_{\dP}^{-\underline{m}_{\dP}-\underline{2}})(1-\beta_{x}^{\dP}\beta_{y}^{\dP}\beta_{z}^{\dP}\varpi_{\dP}^{-\underline{m}_{\dP}-\underline{2}})$},&\dP\neq\dP_0\\
\mbox{\small$(1-\alpha_{x}^{\dP_0}\alpha_{y}^{\dP_0}\beta_{z}^{\dP_0}p^{1-m_{0}})(1-\alpha_{x}^{\dP_0}\beta_{y}^{\dP_0}\beta_{z}^{\dP_0}p^{1-m_{0}})(1-\beta_{x}^{\dP_0}\alpha_{y}^{\dP_0}\beta_{z}^{\dP_0}p^{1-m_{0}})(1-\beta_{x}^{\dP_0}\beta_{y}^{\dP_0}\beta_{z}^{\dP_0}p^{1-m_{0}})$},&\dP=\dP_0
\end{array}\right., 
\]
\[
\mathcal{E}_{\dP,1}(z):=\left\{\begin{array}{lc} 
(1- (\beta_z^{\dP})^2\varpi_{\dP}^{-\underline{k}_{3,\dP}-\underline{2}})\cdot (1- (\beta_z^{\dP})^2\varpi_{\dP}^{-\underline{k}_{3,\dP}-\underline{1}}),&\dP\neq \dP_0,\\
(1- (\beta_z^{\dP_0})^{2}p^{-k_{3,\tau_0}})\cdot (1- (\beta_z^{\dP_0})^{2}p^{1-k_{3,\tau_0}}),&\dP= \dP_0,
\end{array}\right. 
\]
$m_{0}=\frac{k_{1,\tau_0}+k_{2,\tau_0}+k_{3,\tau_0}}{2}\geq 0$, $\underline{m}_{\dP}=\frac{\underline{k}_{1,\dP}+\underline{k}_{2,\dP}+\underline{k}_{3,\dP}}{2}= \left(\frac{k_{1,\tau}+k_{2,\tau}+k_{3,\tau}}{2}\right)_{\tau \sim\dP}$ and $\tau \sim\dP$ means real embeddings $\tau$ corresponding to embeddings $F_\dP\hookrightarrow\C_p$ through $\iota_p$.
\end{thm}

For the precise shape of the constant $K$ see the end of the proof of theorem \ref{t:interpolation}. The starting point of our construction is a result of Harris-Kudla and Ichino relating the central value $L\left(\frac{1-\nu_1-\nu_2-\nu_3}{2},\Pi_{x}\otimes\Pi_{y}\otimes\Pi_{z}\right)$ in terms of certain trilinear period integrals defined in terms of automorphic forms of $\pi_1$, $\pi_2$ and $\pi_3$. These trilinear periods can be described in terms of some trilinear products (see \ref{IntProp}) which have a geometric interpretation in terms of trilinear products of sections of modular sheaves over unitary Shimura curves (see \ref{tripleonsheaves}). Adapting to our situation the $p$-adic interpolation of the integral powers of the Gauss-Manin connexion in \cite{liu-zhang-zhang17}, and inspired by ideas from \cite{AI17} and \cite{greenberg-seveso} we perform a $p$-adic interpolations of the linear periods over the unitary Shimura curves. This is enough to perform the construction of the $p$-adic $L$-functions.

\subsection{About our hypothesis} The condition $[F_{\mathfrak{p}_0}: \Q_p]= 1$ is used to avoid subtleties about the weight space. But it interesting to remark that several parts of the paper can be performed without any condition. 

The ordinarity condition allow us to follows the approach of \cite{liu-zhang-zhang17} about the interpolation of integral power of the Gauss-Manin connection. We believe that following \cite{AI17} it is possible to remove this hypothesis.

\subsubsection*{Acknowledgements.}
The authors are supported in part by DGICYT Grant MTM2015-63829-P.
This project has received funding from the European Research Council
(ERC) under the European Union's Horizon 2020 research and innovation
programme (grant agreement No. 682152).

\part{Background}

\section{Basic notations}
Let $\mathbb{A}$ be adeles of $\Q$ and $\A_f$ the finite adeles. Let $F$ be a totally real field of degree $d= [F:\Q]$, $\mathcal{O}_{F}$ its ring of integers and $\Sigma_F$ the set of real embeddings of $F$. In all this paper we fix an embedding  $\tau_0 \in \Sigma_F$. We denote by $\underline{1}\in \Z[\Sigma_F]$ the element with each coordinate equals to $1$. For $x\in F^\times$ and $\underline{k}\in \Z[\Sigma_F]$ we put $x^{\underline{k}}= \prod_{\tau\in\Sigma_F}\tau(x)^{k_{\tau}}$.

We fix a prime number $p> 2$, denote by $\Sigma_p$ the set of the embeddings of $F$ in $\overline{\Q}_p$ and we fix an embedding $\iota_p:\bar\Q\hookrightarrow\C_p$. For each prime of $\mathfrak{p} \mid p$ let $F_{\mathfrak{p}}$ be the completion of $F$ at $\mathfrak{p}$, $\Sigma_{\mathfrak{p}}$ the set of its embeddings  in $\C_p$, $\mathcal{O}_{\mathfrak{p}}$ its ring of integers, $\kappa_{\mathfrak{p}}$ its residue field, $q_{\mathfrak{p}}= \sharp\kappa_{\mathfrak{p}}$ and $e_{\mathfrak{p}}$ the ramification index. We also fix uniformizers $\varpi_{\mathfrak{p}} \in \mathcal{O}_{\mathfrak{p}}$. Using the embedding $\iota_p$ we identify $\Sigma_F$ with $\Sigma_p:=\bigcup_{\mathfrak{p} \mid p}\Sigma_{\mathfrak{p}}$ in the natural way. Moreover we will use the notation $\cO:= \cO_{F}\otimes \Z_p$ which naturally decompose as $\cO= \prod_{\dP}\cO_\dP$.

In all this text we denote by $\mathfrak{p}_0$ the prime corresponding to $\iota_p$ and $\tau_0$ and we put $F_0:= F_{\mathfrak{p}_0}$, $\cO_0:= \cO_{\mathfrak{p}_0}$ and $\Sigma_0:= \Sigma_{\mathfrak{p}_0}$. Moreover, we suppose the following hypothesis:
\begin{hyp} \label{hypothesis 1} $[F_0: \Q_p]= 1$ and $F$ unramified at $p$. 
\end{hyp}
We denote $\cO^{\tau_0}:=\prod_{\mathfrak{p}\neq \mathfrak{p}_0}\cO_{\mathfrak{p}}$. Thus we have the following decomposition $\cO= \cO_0\times \cO^{\tau_0}= \Z_p\times \cO^{\tau_0}$.

 We also fix a quaternion algebra over $F$ denoted by $B$ such that:
\begin{itemize}
	\item[(i)] split at $\tau_0$ and at each $\mathfrak{p}\mid p$,
	\item[(ii)] is ramified at each $\tau \in \Sigma_F\setminus \{\tau_0\}$.
\end{itemize}

For $\tau \in \Sigma_F$ we put $B_\tau:= B\otimes_{F, \tau}\R$ and fix an identification $B_{\tau_0}\cong \mathrm{M}_2(\R)$ and let $B_{\tau_0}^+\subset B_{\tau_0}^\times$ be the elements of positive norm. Moreover for each $\tau \in \Sigma_F\smallsetminus\{\tau_0\}$ we fix an isomorphism $B\otimes_{\tau}\C\cong \mathrm{M}_2(\C)$  and denote by $\iota_\tau:B_{\tau}^\times\hookrightarrow \GL_2(\C)$ the embedding obtained. We denote by $\mathrm{disc}(B)$ the discriminant of $B$.

 Now we introduce the main reductive groups over $\Q$ used in this text. Firstly, we put $G:={\rm Res}_{\mathbb{Q}}^F B^\times$ and we denote by $G(\Q)^+$ the subgroup of elements of $G(\Q)$ such that its image in $B_{\tau_0}$ is contained in $B_{\tau_0}^+$. Let $\det:G\rightarrow{\rm Res}_{\Q}^F\G_{m,F}$ be the reduced norm and let $G^*=G\times_{{\rm Res}_{\Q}^F\G_{m,F}}\G_{m,\Q}\subseteq G$.

We choose from now on $\lambda \in \Q$ such that $\lambda < 0$ and $p$ split in $\Q(\sqrt{\lambda})$. Let $E:= F(\sqrt{\lambda})$ and denote $z \mapsto \overline{z}$ the not-trivial automorphism of $E/F$. For each $\tau \in \Sigma_F$ let $\tilde\tau:E\rightarrow\C$ be the embedding above $\tau$ such that $\tilde\tau(\sqrt{\lambda})=\sqrt{\lambda}$.

We denote $D:= B\otimes_{F}E$ which is a quaternion algebra over $E$ and $D \rightarrow D$ to the involution defined by $l= b\otimes z \mapsto \overline{l}:= \bar b\otimes \overline{z}$ where $\bar b$ is the canonical involution of $B$. We fix $\delta \in D^\times$ such that $\overline{\delta}= -\delta$\footnote{Our $\delta$ corresponds to the product $\alpha\delta$ with the notation of \cite{carayol86}} and define a new involution on $D$ by $l\mapsto l^{\ast}:= \delta^{-1}\overline{l}\delta$.  We denote by $V$ to the underlying $\Q$-vector space of $D$ endowed with the natural left action of $D$.  
We have a symplectic bilinear form on $V$: 
$$\Theta: V \times V \rightarrow \Q, \ \ \  (v,w) \mapsto \mathrm{Tr}_{E/\Q}(\mathrm{Tr}_{D/E}(v\delta w^{\ast})).$$

Let $G_D:={\rm Res}_{\Q}^{E}D^\times$ and $G'$  be the reductive group over $\Q$ such that for each $\Q$-algebra $R$ we have:
$$G'(R)= \left\lbrace  \mathrm{D-linear \ symplectic \ similitudes \ of } \ (V\otimes_{\Q}R, \Theta\otimes_{\Q}R) \right\rbrace. $$
If $d\in G'\subseteq D^\times$, as $\Theta(vd, wd)=\mu\Theta(v, w)$ with $\mu\in \Q$, we have
	$\mathrm{Tr}_{E/\Q}(\mathrm{Tr}_{D/E}(vd\delta d^\ast w^{\ast}))=\mathrm{Tr}_{E/\Q}( d\bar d\mathrm{Tr}_{D/E}(v \delta w^{\ast}))=\mathrm{Tr}_{E/\Q}(\mu\mathrm{Tr}_{D/E}(v \delta w^{\ast}))$ then $\mu=d\bar d$. Thus, $d=be\in B^\times E^\times\subset D^\times$ with $\det(b)e\bar e=\mu$. 
Denote $T_E:={\rm Res}_{\Q}^E\G_{m.E}/{\rm Res}_{\Q}^F\G_{m,F}$. The above computation show that we have the exact sequence
\begin{equation}\label{eqdescGG'}
0\longrightarrow G^*\longrightarrow G'\stackrel{\pi_T}{\longrightarrow} T_E\longrightarrow 0;\qquad \pi_T(be)=[e].
\end{equation}

\begin{rmk} If there exists a embedding $\varphi:E\hookrightarrow B$ then $D\simeq\M_2(E)$ and $T_E$ acts on $G^\ast$ given by conjugation of $\varphi(e)$. In this situation, we have an isomorphism $G'=G^*\rtimes T_E$ given by $be\mapsto (b\varphi(e),e)$.
\end{rmk}




Let $A\subset B$ be a $\cO_F$-order, and let $A_D:=A\otimes_{\cO_F}\cO_E$.
We introduce a way to cut certain modules endowed with an action of $A_D$. We denote $\psi:\Q(\sqrt{\lambda})\longrightarrow D$ given by $z\longmapsto1\otimes z$ and fix an extension $R/\cO_E$  such that $A\otimes_{\cO_F}R=\M_2(R)$. For any $R$-module $M$ endowed with a linear action of $A_D$, we define 
	$$M^{+}:=\left\{v\in M:\;\psi(e)\ast v=ev,\mbox{ for all }e\in \Z(\sqrt{\lambda})\right\},$$
	$$M^{-}:=\left\{v\in M:\;\psi(e)\ast v=\bar ev,\mbox{ for all } e\in \Z(\sqrt{\lambda})\right\}.$$
	 Each $M^{\pm}$ is equipped with an action of $A\otimes_{\Z}R=\M_2(R)\otimes_{\Z}\cO_F$ and we put $M^{\pm,1}:=(\begin{smallmatrix}1&0\\0&0\end{smallmatrix})M^{\pm}$ and $M^{\pm,2}:=(\begin{smallmatrix}0&0\\0&1\end{smallmatrix})M^{\pm}.$  
	Note that both are isomorphic $R\otimes_{\Z}\cO_F$-modules through the matrix $(\begin{smallmatrix}0&1\\1&0\end{smallmatrix})$. Moreover, by construction we have:
	 $$M\supseteq M^+\oplus M^-= M^{+,1}\oplus M^{-,1}\oplus M^{+,2}\oplus M^{-,2},$$
	 and the inclusion is an equality if ${\rm disc}(\Q(\sqrt{\lambda}))\in R^\times$.


\section{Automorphic forms for $G$}\label{s:Automorphic forms for G} In this section we recall some facts about quaternionic automorphic forms over $F$. Moreover, we introduce some algebraic and analytic operations (triple products) and we recall the Ichino's formula which relate these operations to central $L$-values of certain complex $L$-functions. One the main goals of the main body of this paper is to $p$-adically deform these algebraic operations. 
		
\subsection{Quaternionic Automorphic Forms}\label{AutQuaFrm} We start introducing some notations about local representations. Fix $k>1$ and $\nu$ integers such that $k\equiv \nu ({\rm mod}\;2)$. 
	
	On the one hand we write $\cD(k,\nu)$ for the $(\cG_{\tau_0},O(2))$-module of discrete series of weight $k$ and central character $a\mapsto a^{\nu}$. It is the sub vector space 	of  $C^\infty(\GL_2(\R)^+,\C)$ generated by the holomorphic element $f_k$ defined by: $$f_k\left(\begin{smallmatrix}a&b\\c&d\end{smallmatrix}\right)= (ad-bc)^{\frac{\nu+k}{2}}(ci+d)^{-k}$$
	Let $R,L$ be the Shimura-Mass operators defined in \cite[Proposition 2.2.5]{Bump}, then $f_k$ is also characterized by the relations
	\begin{equation}\label{relhol}
	(\begin{smallmatrix}\cos t &\sin t\\-\sin t&\cos t\end{smallmatrix})f_k=e^{kit}f_k, \qquad af_k=a^{\nu}f_k, \qquad Lf_k=0.
	\end{equation}
	
	On the other hand let $\cP(k,\nu)=\Sym^{k}(\C^2)\otimes\C[\frac{\nu-k}{2}]$ be the space of homogeneous polynomials of degree $k$ endowed with the natural action of $\mathrm{GL}_2(\C)$ i.e. if $\gamma= (\begin{smallmatrix}a&b\\c&d\end{smallmatrix})\in \mathrm{GL}_2(\C)$ and  $P(x,y)\in\cP(k,\nu)$ then we put:
	$$ \gamma P(x, y):=\det(\gamma)^{(\nu-k)/2}P(ax+cy,bx+dy).$$
	For each $\tau \in \Sigma_{F}\smallsetminus \{\tau_0\}$ we write $\cP_\tau(k,\nu)$ to denote the $\C$-vector space $\cP(k,\nu)$ endowed with the action of $B_{\tau}^{\times}$ given through the embedding $\iota_\tau$. Remark that we have an isomorphism:	 
	\begin{equation}\label{dual identification polynomials}
\cP_\tau(k,-\nu)^\vee\simeq \cP_\tau(k,\nu), \ \ \mu\longmapsto P_\mu(X,Y)=\mu\left(\left|\begin{smallmatrix}X&Y\\x&y\end{smallmatrix}\right|^{ k}\right)
	\end{equation}

	We denote by $\cA(\C)$ the $\C$-vector space of functions $f:G(\A)\longrightarrow\C$ such that:
	\begin{itemize}
		\item[(i)] There exists an open compact subgroup $U\subseteq G(\A_f)$ such that $f(g U)=f(g)$, for all $g\in G(\A)$.
		
		\item[(ii)] $f\mid_{B_{\tau}^\times}\in C^\infty(\GL_2(\R),\C)$, here we use the fixed identification $B_{\tau}^\times\simeq\GL_2(\R)$ .
		
		\item[(iii)] We assume that any $f\in\cA(\C)$ is $O(2)$-finite i.e. its right translates by elements of $O(2)\subset B_{\tau}^\times$ span a finite-dimensional vector space. 
		
		\item[(iv)] We assume that any $f\in\cA(\C)$ is ${\mathcal{Z}}$-finite, where ${\mathcal{Z}}$ is the centre of the universal enveloping algebra of $B_{\tau}^\times$.
	\end{itemize}
	Write $\rho$ for the action of $G(\A)$ given by right translation, then $(\cA(\C),\rho)$ defines a smooth $G(\A^{\infty})\times\prod_{\tau \in \Sigma_F\smallsetminus \{\tau_0\}}B_{\tau}^\times$-representation and a $(\cG_{\tau_0},O(2))$-module. Moreover, $\cA(\C)$ is also equipped with a $G(\Q)$-action: if $h\in G(\Q)$,  $g\in G(\A)$, $f\in \cA(\C)$ we put $(h\cdot f)(g)=f(h^{-1} g)$.   \\

	We denote $\Xi=\{\underline{k}\in\N[\Sigma_F]| \,\, k_{\tau_0}> 0, \ k_\tau\equiv k_{\tau'}({\rm mod} \,2)\,\,\mbox{for all }\tau, \tau' \}$ and for each $\nu \in \Z$ we put $\Xi_\nu=\{\underline{k}\in\Xi | \,k_{\tau_0}\equiv\nu({\rm mod}\,2)\}$.  If $\nu \in \Z$ and $\underline{k} \in \Xi_\nu$  we put
	\[
	\cP^{\tau_0}(\underline{k},\nu):= \bigotimes_{\tau \in \Sigma_F\smallsetminus \{\tau_0\}}\cP_\tau(k_\tau,\nu) \qquad \qquad \cD(\underline{k},\nu):=\cD(k_{\tau_0},\nu)\otimes\cP^{\tau_0}(\underline{k},\nu).
	\]
  It is a $G(\R)$-representation i.e. it is a $(\cG_{\tau_0},O(2))$-module endowed with an action of $\prod_{\tau \in \Sigma_F\smallsetminus \{\tau_0\}}B_{\tau}^\times$. Remark that the action of the centre $(F\otimes\R)^\times$ is given by the parallel character $a\mapsto a^\nu$. Using this module we consider the space of automorphic forms:
	\[
	\cA(\underline{k},\nu):=\Homo_{G(\R)}(\cD(\underline{k},\nu),\cA(\C)),
	\]
	which is endowed with natural $G(\Q)$ and $G(\A_f)$-actions, that commute with each other. 
	\begin{defi} The elements of $H^0(G(\Q),\cA(\underline{k},\nu))$ are called \emph{automorphic forms of weight $(\underline{k},\nu)$}.
	\end{defi}
	
\begin{rmk}\label{mod-aut} Fix $\phi \in H^0(G(\Q),\cA(\underline{k},\nu))$. By identifying $B_{\tau_0}^+/\SO(2)F_0^\times$ with the Poincar\'e upper half plane $\dH$, we define the holomorphic function	$f_\phi:\dH\times G(\A_f)\longrightarrow \bigotimes_{\tau \in \Sigma_F\smallsetminus\{\tau_0\}}\cP_\tau(k_\tau,\nu)^\vee$ by: 
		$$f_\phi(z ,g_f)(P):=\phi(f_{k_{\tau_0}}\otimes P)(g_{\tau_0}g_f)f_{k_{\tau_0}}(g_{\tau_0})^{-1}$$
	 here $z= g_{\tau_0} i \in \mathfrak{H}$ for some $g_{\tau_0}\in B_{\tau_0}^{\times}$, $g_f\in G(\A_f)$ and $P\in \bigotimes_{\tau \neq \tau0}\cP_\tau(k_\tau,\nu)$. Then each $\gamma\in G(\Q)^+$ we have:
		$$f_\phi(\gamma z,\gamma g_f)=\det\gamma^{\frac{-\nu-k_{\tau_0}}{2}}(cz+d)^{k_{\tau_0}}\gamma \left(f_\phi(z,g_f)\right)$$ here $\gamma=(\begin{smallmatrix}a&b\\c&d\end{smallmatrix})$ when considered in $\GL_2(\R)$ through $\tau_0$. 
		As $\cD(k_{\tau_0},\nu)$ is generated by $f_{k_{\tau_0}}$ then to provide $f_\phi$ is equivalent to provide $\phi$.\end{rmk}
	

Let $\cN$ be an ideal of $F$ prime to $\mathrm{disc}(B)$, we denote: 
 \[
 K_1^B(\cN):=\left\{g\in G(\hat\Z):\  g_{\mathfrak{n}}=\left(\begin{array}{cc}a&b\\c&d\end{array}\right)\in\GL_2(\prod_{\ell \mid \mathfrak{n}}\Z_{\ell}), \ c\equiv d-1\equiv 0\;{\rm mod}\;\cN\right\}.
 \]
\begin{defi} If $\chi:\A_F^\times/F^\times\rightarrow\bar\Q^\times$ is a finite character we denote
$M_{(\underline{k},\nu)}(\Gamma_1(\cN),\chi)$ for the space of $\phi\in H^0(G(\Q),\cA(\underline{k},\nu))^{K_1^B(\cN)}$ such that $\phi(f)(ag)=\chi(a)|a|^{\nu}\phi(f)(g)$ for $a\in\A_F^\times$, $g\in G(\A)$, $f\in \cD(\underline{k},\nu)$ and for $|\cdot|:\A_F^\times/F^\times\rightarrow\R^\times$ the usual norm character.
\end{defi}

\subsection{Archimedean trilinear products}\label{ss:trilinear products} Firstly we treat the local setting. Let $k_1, k_2, k_3 \in \N_{>0}$ and $\nu_1, \nu_2, \nu_3=\nu_1+\nu_2 \in \Z$ such that $k_i\equiv \nu_i ({\rm mod}\;2)$ for $i= 1, 2, 3$. We consider the following two cases:
	
	\textbf{(1) Unbalanced:} suppose that $k_3\geq k_1+k_2$ and $m:=(k_1+k_2+k_3)/2 \in \Z$. We denote $m_3:=\frac{-k_1- k_2+ k_3}{2}\geq 0$ and consider the map $t_{\tau_0}:\cD(k_3,\nu_1+\nu_2)\longrightarrow \cD(k_1,\nu_1)\otimes\cD(k_2,\nu_2)$ given by:
	\begin{equation}\label{tripleD}
	t_{\tau_0}(f_{k_3})= \sum_{j=0}^{m_3}(-1)^j\binom{m_3}{j}\binom{m-2}{k_1+j-1}R^j(f_{k_1})\otimes R^{m_3-j}(f_{k_2}).
	\end{equation}
	This map is well defined since $t_{\tau_0}(f_{k_3})$ also satisfies the relations \eqref{relhol}. 
	
	\textbf{(2) Balanced:} suppose that for $i= 1,2,3$ we have $2k_i\leq k_1+k_2+k_3$ and $m:=(k_1+k_2+k_3)/2 \in \Z$. We denote $m_i:=m-k_i\geq 0$ for $i= 1,2,3$, and for each $\tau \in \Sigma_F\smallsetminus\{\tau_0\}$ let $t_{\tau}^\vee:\cP_\tau(k_1,\nu_1)^\vee\otimes\cP_\tau(k_2,\nu_2)^\vee\longrightarrow \cP_\tau(k_3,-\nu_1-\nu_2)$ be the map given for $\mu_1 \in \cP_\tau(k_1,\nu_1)^\vee$ and $\mu_2 \in \cP_\tau(k_2,\nu_2)^\vee$ by:
	\[
	\begin{array}{l}
	t_{\tau}^\vee(\mu_1\otimes\mu_2)(x,y)=\mu_1\left(\mu_2\left(\left|\begin{array}{cc}x&y\\x_2&y_2\end{array}\right|^{m_1}\left|\begin{array}{cc}x&y\\x_1&y_1\end{array}\right|^{m_2}\left|\begin{array}{cc}x_1&y_1\\x_2&y_2\end{array}\right|^{m_3}\right)\right).
	\end{array}
	\]
	Using the identification (\ref{dual identification polynomials}) we obtain a $B_{\tau}^\times$-equivariant morphism
	\begin{equation}\label{tripleP}
	t_{\tau}:\cP_{\tau}(k_3,\nu_1+\nu_2)\longrightarrow \cP_{\tau}(k_1,\nu_1)\otimes\cP_{\tau}(k_2,\nu_2).
	\end{equation}
	
Now we consider the global setting. Let $\underline{k}_1, \underline{k}_2,  \underline{k}_3\in\Xi$. 
	\begin{defi}\label{d:unbalanced weights} We say that $\underline{k}_1,\underline{k}_2, \underline{k}_3$ are \emph{unbalanced at $\tau_0$ with dominant weight $\underline{k}_3$} if for each $\tau \in \Sigma_F$ the integer $k_{1, \tau}+k_{2, \tau}+k_{3, \tau}$ is even and
		\begin{itemize}
			\item[(i)] $k_{3, \tau_0}\geq k_{1, \tau_0}+k_{2, \tau_0}$,
			
			\item[(ii)] if $i= 1,2,3$ then $2k_{i, \tau}\leq k_{1, \tau}+k_{2, \tau}+k_{3, \tau}$  for each $\tau \neq \tau_0$.
		\end{itemize} 
	\end{defi}
	
	Assume that $\underline{k}_1\in\Xi_{\nu_1}$, $\underline{k}_2\in\Xi_{\nu_2}$ and $\underline{k}_3\in\Xi_{\nu_1+\nu_2}$ are unbalanced at $\tau_0$ with dominant weight $\underline{k}_3$.
	The products \eqref{tripleD} and \eqref{tripleP} provide a morphism of $G(\R)$-representations
	\[
	t_\infty:\cD(\underline{k}_3,\nu_1+\nu_2)\longrightarrow \cD(\underline{k}_1,\nu_1)\otimes\cD(\underline{k}_2,\nu_2)
	\]
	Thus we obtain a global and $G(\Q)$-equivariant \emph{linear  product}:
	\[
	t:\cA(\underline{k}_1,\nu_1)\otimes \cA(\underline{k}_2,\nu_2)\longrightarrow \cA(\underline{k}_3,\nu_1+\nu_2)
	\]
given by  $t(\phi_1,\phi_2)(f):=\phi_1\phi_2(t_\infty(f))$ for $\phi_1 \in \cA(\underline{k}_1,\nu_1)$, $\phi_2 \in \cA(\underline{k}_2,\nu_2)$ and $f\in \cD(\underline{k}_3,\nu_1+ \nu_2)$. Here $\phi_1\phi_2\left(\sum_jf_j^1\otimes f_j^2\right):=\sum_j\phi_1(f_j^1)\phi_2(f_j^2)$, for any $f_j^i\in \cD(\underline{k}_i,\nu_i)$. From this we obtain a trilinear product between automorphic forms: 
	\begin{equation}\label{triple}
t:H^0(G(\Q),\cA(\underline{k}_1,\nu_1))\times H^0(G(\Q),\cA(\underline{k}_2,\nu_2))\longrightarrow H^0(G(\Q),\cA(\underline{k}_3,\nu_1+\nu_2)). 
	\end{equation}

\subsection{Test vectors and non-archimedean trilinear products}\label{ss:NA-trilinear products}

Let $W$ be an spherical representation of $\GL_2(F_v)$, where $F_v$ is a finite extension of $\Q_p$, and write $\epsilon$ for the central character. Write $\cO_v$ for the integer ring of $F_v$ with uniformizer $\varpi$, let $\kappa$ be the residue field, write $q:=\#\kappa$, and let
$$K:=\GL_2(\cO_v), \qquad K_0(\varpi^n):=\left\{\left(\begin{array}{cc}a&b\\ c&d\end{array}\right)\in K,\;\varpi^n\mid c\right\}.$$ 
Assume that $W$ is equipped with a hermitian (Petersson) inner product
\[
\langle\cdot,\cdot\rangle:W\times W\longrightarrow \C,\qquad \langle gv,gv'\rangle=|\epsilon(\det (g))|\cdot\langle v,v'\rangle.
\]
For a fixed spherical vector $v_0\in V^K$, we construct the test vector
$v_n=\mbox{\tiny$\left(\begin{array}{cc}1&\\ &\varpi^n\end{array}\right)$}v_0\in W^{K_0(\varpi^n)}$.
Clearly we have
\begin{equation}\label{Petprod1}
\langle v_n,v_m\rangle=|\epsilon(\varpi)|^{n-m}\langle v_{n-m},v_0\rangle,\qquad n\geq m.
\end{equation}
Write $T$ for the usual Hecke operator, and denote by $Tv_0=a\cdot v_0$ the corresponding eigenvalue. Notice that
\begin{equation}\label{UvsT}
a\cdot v_0=\frac{1}{q}\left(v_1+\sum_{i\in \kappa}g_i v_0\right)=:q^{-1}v_1+Uv_0;\qquad g_i=\left(\begin{array}{cc}\varpi&i\\&1\end{array}\right).
\end{equation}
The relations
\begin{equation}\label{auxeq}
g_i=k_ig_0,\quad (i\in\kappa),\quad k_i=\left(\begin{array}{cc}1&i\\&1\end{array}\right);\qquad g_0^{-1}=\varpi^{-1}\left(\begin{array}{cc}1&\\ &\varpi\end{array}\right)
\end{equation}
and the property $k_i^{-1}v_n=v_n$ imply that 
\[
\langle Uv_0,v_n\rangle=q^{-1}\sum_{i\in \kappa}\langle g_iv_0,v_n\rangle=|\epsilon(\varpi)|\overline{\epsilon(\varpi)}^{-1}\langle v_0,v_{n+1}\rangle=|\epsilon(\varpi)|^{-1}\epsilon(\varpi)\langle v_0,v_{n+1}\rangle,
\]
and analogously $\langle v_n,U v_0\rangle=|\epsilon(\varpi)|\epsilon(\varpi)^{-1}\cdot\langle v_{n+1},v_0\rangle$. Hence using previous equation
\begin{eqnarray*}
\langle v_{1},v_0\rangle&=&|\epsilon(\varpi)|^{-1}\epsilon(\varpi)\langle v_{0},U v_0\rangle=\bar a|\epsilon(\varpi)|^{-1}\epsilon(\varpi)\langle v_0,v_0\rangle-q^{-1}|\epsilon(\varpi)|^{-1}\epsilon(\varpi)\langle v_0,v_1\rangle\\
&=&\bar a|\epsilon(\varpi)|^{-1}\epsilon(\varpi)\langle v_0,v_0\rangle-q^{-1}\langle Uv_0,v_0\rangle=(\bar a|\epsilon(\varpi)|^{-1}\epsilon(\varpi)-aq^{-1})\langle v_0,v_0\rangle+q^{-2}\langle v_1,v_0\rangle.
\end{eqnarray*}
Since $\bar a|\epsilon(\varpi)|^{-1}\epsilon(\varpi)=a$, we obtain that 
\begin{equation}\label{Petprod2}
\langle v_1,v_0\rangle=\frac{a}{1+q^{-1}}\langle v_0,v_0\rangle.
\end{equation}
On the other side, $\langle v_{n+2},v_0\rangle=|\epsilon(\varpi)|^{-1}\epsilon(\varpi)\langle v_{n+1},Uv_0\rangle$, 
hence
\begin{equation}\label{Petprod3}
\langle v_{n+2},v_0\rangle=a\langle v_{n+1},v_0\rangle-q^{-1}\epsilon(\varpi)\langle v_{n},v_0\rangle
\end{equation}
\begin{lemma}\label{lemapetersson}
Let $\chi=|\epsilon(\varpi)|^{-1}\epsilon(\varpi)$. There exists $\varrho(X,Y)\in \Q(\chi)[X,Y]$ such that
$\langle v_n,v_m\rangle=\varrho(a,|\epsilon(\varpi)|)\cdot\langle v_0,v_0\rangle$,
for all $n$ and $m$.
\end{lemma}
\begin{proof}
Follows directly from \eqref{Petprod1}, \eqref{Petprod2}, \eqref{Petprod3} and the fact that $\langle\phi_1,\phi_2\rangle=\overline{\langle\phi_2,\phi_1\rangle}$.
\end{proof}

Let $W_i$  ($i=1,2,3$) as above and assume that we have a trilinear product
\[
t:W_1\otimes W_2\longrightarrow W_3,\qquad t(gv_1,gv_2)=g t(v_1,v_2),\quad g\in\GL_2(F_v). 
\]
Write $v^i_0\in W_i^K$ for fixed spherical vectors. Let $\alpha_i$, $\beta_i$ be the roots of the Hecke polynomial $X^2-a_iX+\epsilon_i(\varpi)q^{-1}$. Write $Vv:=g_0^{-1}v$. By \eqref{UvsT} and \eqref{auxeq} the vector $v_{\alpha_i}:=(1-\beta_i V)v_0^i=v_0^i-\beta_i\epsilon_i(\varpi)^{-1}v_1^i$ satisfies
\[
Uv_{\alpha_i}=a_iv_0^i-q^{-1}v_1^i-q^{-1}\beta_i\epsilon_i(\varpi)^{-1}\left(\sum_i \varpi k_ig_0g_0^{-1}v_0^i\right)=(a_i-\beta_i)v_0^i-\alpha_i\beta_i\epsilon_i(\varpi)^{-1}v_i^i=\alpha_i\cdot v_{\alpha_i}.
\]
Analogously,
\begin{equation}\label{adjointUp}
U^*v_{\beta_i}^*=\beta_iv_{\beta_i}^*\qquad v_{\beta_i}^*:=v_1^i-\alpha_iv_0^i,\qquad U^*v:=\frac{1}{q}\sum_i\left(\begin{array}{cc}1&\\i\varpi&\varpi\end{array}\right)v.
\end{equation}
It is easy to compute that
$\langle Uv,v'\rangle=\chi_3\cdot\langle  v,U^*v'\rangle$, where $\chi_3:=\frac{\epsilon_3(\varpi)}{|\epsilon_3(\varpi)|}$. Thus, since $\bar\alpha_3\chi_3=\beta_3$, we deduce that whenever $\alpha_3\neq\beta_3$ the map $v\mapsto\langle v, v_{\beta_3}^*\rangle\langle v_{\alpha_3}, v_{\beta_3}^*\rangle^{-1}$ provides the projection into the subspace of $W_3$ where $U$ acts as $\alpha_3$.

We define the \emph{$\varpi$-deplation} $v^{[p]}=(1-VU)v$.
We aim to compute the expressions 
\[
\frac{\langle t(v_{\alpha_1},v_{\alpha_2}), v_{\beta_3}^*\rangle}{\langle v_{\alpha_3},v_{\beta_3}^*\rangle},\qquad \frac{\langle t(v_{\alpha_1}^{[p]},v_{\alpha_2}),v_{\beta_3}^*\rangle}{\langle v_{\alpha_3},v_{\beta_3}^*\rangle},\qquad\mbox{in terms of}\qquad \frac{\langle t(v_{0}^1,v_{0}^2),v_{0}^3\rangle}{\langle v_{0}^3,v_{0}^3\rangle}.
\]
Assume that $v\in W_1^{K_0(\varpi)}$, $v'\in W_2^{K_0(\varpi)}$ are test vectors, we compute using \eqref{auxeq}
\begin{eqnarray}
\langle t(Vv,Vv'), v_{\beta_3}^*\rangle&=&\alpha_3^{-1}\chi_3\langle Vt(v,v'), U^*v_{\beta_3}^*\rangle=\alpha_3^{-1}\langle t(v,v'), v_{\beta_3}^*\rangle\\
Ut(v^{[p]},Vv')&=&\frac{1}{q}\sum_it(g_iv^{[p]},g_iVv')=\frac{1}{q}\sum_it(g_iv^{[p]},v')=t(Uv^{[p]},v')=0\\
\langle t(v,Vv'), v_{\beta_3}^*\rangle&=&\alpha_3^{-1}\langle U(t(v^{[p]},Vv')+t(VUv,Vv')), v_{\beta_3}^*\rangle=\alpha_3^{-1}\langle t(Uv,v'), v_{\beta_3}^*\rangle\\
\langle t(Vv,v'), v_{\beta_3}^*\rangle&=&\alpha_3^{-1}\langle t(v,Uv'), v_{\beta_3}^*\rangle.\label{eqVU}
\end{eqnarray}
Thus we obtain
\begin{eqnarray*}
\langle t(v_{\alpha_1},v_{\alpha_2}), v_{\beta_3}^*\rangle&=&\langle t(v_{0}^1,v_{\alpha_2}), v_{\beta_3}^*\rangle-\beta_1\langle t(Vv_{0}^1,v_{\alpha_2}), v_{\beta_3}^*\rangle=(1-\beta_1\alpha_2\alpha_3^{-1})\langle t(v_{0}^1,v_{\alpha_2}), v_{\beta_3}^*\rangle\\
\langle t(v_{0}^1,v_{\alpha_2}), v_{\beta_3}^*\rangle&=&
(1-\beta_2\alpha_3^{-1}a_1)\langle t(v_{0}^1,v_0^2), v_{\beta_3}^*\rangle+q^{-1}\epsilon_1(\varpi)\beta_2\alpha_3^{-1}\langle t(Vv_{0}^1,v_0^2), v_{\beta_3}^*\rangle\\
\langle t(Vv_{0}^1,v_0^2), v_{\beta_3}^*\rangle&=&\alpha_3^{-1}a_2\langle t(v_{0}^1,v_0^2), v_{\beta_3}^*\rangle-\alpha_3^{-1}q^{-1}\epsilon_2(\varpi)\langle t(v_{0}^1,Vv_0^2), v_{\beta_3}^*\rangle\\
&=&\alpha_3^{-1}\beta_2\langle t(v_{0}^1,v_0^2), v_{\beta_3}^*\rangle+\alpha_3^{-1}\alpha_2\langle t(v_{0}^1,v_{\alpha_2}), v_{\beta_3}^*\rangle,
\end{eqnarray*}
and therefore
\[
\langle t(v_{\alpha_1},v_{\alpha_2}), v_{\beta_3}^*\rangle=\frac{(1-\beta_1\alpha_2\alpha_3^{-1})(1-\alpha_1\beta_2\alpha_3^{-1})(1-\beta_1\beta_2\alpha_3^{-1})}{1-\alpha_1\beta_1\alpha_2\beta_2\alpha_3^{-2}}\langle t(v_{0}^1,v_0^2), v_{\beta_3}^*\rangle.
\]
Since $t(v_0^1,v_0^2)=Cv_0^3$ for some $C\in\C$, we compute using \eqref{Petprod2}:
\[
\langle t(v_{0}^1,v_0^2), v_{\beta_3}^*\rangle=\langle t(v_{0}^1,v_0^2), v_1^3\rangle-\bar\alpha_3\langle t(v_{0}^1,v_0^2), v_0^3\rangle=\left(\frac{\bar a_3}{1+q^{-1}}-\bar\alpha_3\right)\langle t(v_{0}^1,v_0^2), v_0^3\rangle,
\]
and similarly we compute 
\begin{eqnarray*}
\langle v_{\alpha_3}, v_{\beta_3}^*\rangle&=&\langle v_{0}^3,v_1^3\rangle-\beta_3\epsilon_3(\varpi)^{-1}\langle v_{1}^3,v_1^3\rangle-\bar\alpha_3\langle v_{0}^3,v_0^3\rangle+\bar\alpha_3\beta_3\epsilon_3(\varpi)\langle v_{1}^3,v_0^3\rangle\\
&=&\left(\frac{\bar\beta_3-\bar\alpha_3+\bar\alpha_3q^{-1}(\bar\alpha_3\bar\beta_3^{-1}-1)}{1+q^{-1}}\right)\langle v_{0}^3,v_0^3\rangle.
\end{eqnarray*}

\begin{prop}\label{eulerfactors}
Assume that $\alpha_3\neq\beta_3$, then we have that
\begin{eqnarray*}
\frac{\langle t(v_{\alpha_1},v_{\alpha_2}), v_{\beta_3}^*\rangle}{\langle v_{\alpha_3},v_{\beta_3}^*\rangle}&=&\frac{(1-\beta_1\alpha_2\alpha_3^{-1})(1-\alpha_1\beta_2\alpha_3^{-1})(1-\beta_1\beta_2\alpha_3^{-1})}{(1-\alpha_1\beta_1\alpha_2\beta_2\alpha_3^{-2})(1-\beta_3\alpha_3^{-1})}\cdot\frac{\langle t(v_{0}^1,v_{0}^2),v_{0}^3\rangle}{\langle v_{0}^3,v_{0}^3\rangle}\\
\frac{\langle t(v_{\alpha_1}^{[p]},v_{\alpha_2}),v_{\beta_3}^*\rangle}{\langle v_{\alpha_3},v_{\beta_3}^*\rangle}&=&\frac{(1-\beta_1\alpha_2\alpha_3^{-1})(1-\alpha_1\beta_2\alpha_3^{-1})(1-\beta_1\beta_2\alpha_3^{-1})(1-\alpha_1\alpha_2\alpha_3^{-1})}{(1-\alpha_1\beta_1\alpha_2\beta_2\alpha_3^{-2})(1-\beta_3\alpha_3^{-1})}\cdot\frac{\langle t(v_{0}^1,v_{0}^2),v_{0}^3\rangle}{\langle v_{0}^3,v_{0}^3\rangle}.
\end{eqnarray*}
\end{prop}
\begin{proof}
The first equality follows directly from the previous computations. For the second equality, 
\[
\langle t(v_{\alpha_1}^{[p]},v_{\alpha_2}),v_{\beta_3}^*\rangle=\langle t(v_{\alpha_1},v_{\alpha_2}),v_{\beta_3}^*\rangle-\alpha_1\langle t(Vv_{\alpha_1},v_{\alpha_2}),v_{\beta_3}^*\rangle=(1-\alpha_1\alpha_3^{-1}\alpha_2)\langle t(v_{\alpha_1},v_{\alpha_2}),v_{\beta_3}^*\rangle,
\]
by \eqref{eqVU}.
\end{proof}

\subsection{Ichino-Harris-Kudla formula and trilinear products}\label{Ichino}
We recall 
a result of Harris-Kudla and Ichino which gives a formula describing the central critical value of triple product $L$-functions in terms of certain trilinear periods. Moreover, we relate those trilinear periods with the trilinear products introduced above.

 Let $\Pi_1, \Pi_2$ and $\Pi_3$ be three irreducible cuspidal automorphic representations of $\mathrm{GL}_2(\mathbb{A}_F)$. Assume also that the corresponding central characters $\varepsilon_i$ satisfy $\varepsilon_1\cdot\varepsilon_2=\varepsilon_3$.  We denote by $L(s, \Pi_1\otimes \Pi_2\otimes\Pi_3)$ the complex triple product $L$-function attached to the tensor product $\Pi_1\otimes \Pi_2\otimes\Pi_3$.  
 
 Let $\pi_1, \pi_2,\pi_3 \subset H^0(G(\Q),\cA(\C))$ be the irreducible automorphic representations of $G(\mathbb{A})$ associated respectively to $\Pi_1, \Pi_2$ and $\Pi_3$ by the Jacquet-Langlands correspondence.
Notice that $\varepsilon_3=\chi\circ|\cdot|^{\nu_3}$, for a finite character $\chi$, being $|\cdot|:\A_F^\times/F^\times\rightarrow\R$ the usual norm. This implies that $\nu_3=\nu_1+\nu_2$ and $|\varepsilon_3|^2=|\cdot|^{2\nu_3}$.
 Notice that $\Pi_i|\det|^{-\frac{\nu_i}{2}}$ is unitary, and write $\widetilde\Pi_3$ for the contragredient representation of $\Pi_3$. 
 Thus $\Pi:=\Pi_1|\det|^{-\frac{\nu_1}{2}}\otimes\Pi_2|\det|^{-\frac{\nu_2}{2}}\otimes\tilde\Pi_3|\det|^{\frac{\nu_3}{2}}$ defines a unitary automorphic representation of $\GL_2(\A_E)$, where $E=F\times F\times F$, which is trivial at $\A_F^\times$ embedded diagonally. Note that $\pi:=\pi_1|\det|^{-\frac{\nu_1}{2}}\otimes\pi_2|\det|^{-\frac{\nu_2}{2}}\otimes\tilde\pi_3|\det|^{\frac{\nu_3}{2}}$ is the Jacquet-Langlands lift of $\Pi$.  For each pair $\varphi \in \pi$, $\tilde\varphi \in \tilde\pi$, where $ \tilde\pi$ is the contragredient representation, we consider the \emph{trilinear period}: 
 \[
 I(\varphi\otimes\tilde\varphi):=\int_{G(\A)/G(\Q)\A_F^\times}\int_{G(\A)/G(\Q)\A_F^\times}\varphi(g)\tilde\varphi(g')dgdg',
 \]
here $dg$ is the normalized Haar measure and in the integral we consider the natural diagonal embedding $\A_F \hookrightarrow \A_F\times \A_F\times\A_F$.  The following is the main result of Harris-Kudla-Ichino in (see \cite{harris-kudla1}, \cite{harris-kudla2}, \cite[Theorem 1.1, Remark 1.3]{ichino08}):
\begin{prop}\label{IchinoProp}
For any $\varphi \in \pi$, $\tilde\varphi \in \tilde\pi$, we have
\[
\frac{I(\varphi\otimes \tilde\varphi)}{(\varphi,\tilde\varphi)}=\frac{1}{2^3}\cdot\zeta_F(2)^2\cdot\frac{L(1/2,\Pi)}{L(1,\Pi,{\rm Ad})}\cdot\prod_v I_v(\varphi_{v}\otimes\tilde\varphi_{v}),
\]
where $I_v(\varphi_{v}\otimes\tilde\varphi_{v}):=\xi_{F_v}(2)^2\cdot\frac{L_v(1,\Pi_v,{\rm Ad})}{L_v(1/2,\Pi_v)}\cdot\int_{F_v^\times\backslash B_v^\times}\frac{(\pi_v(b_v)\varphi_{v},\tilde\varphi_{v})_v}{(\varphi_v,\tilde\varphi_v)_v} db_v$, for certain pairing $(\;,\;)$ between $\pi$ and $\tilde\pi$ compatible with local pairings $(\;,\;)_v$. 
\end{prop}

In order to interpret this result in terms of the trilinear products introduced in \S\ref{ss:trilinear products} we introduce some notations. If $\nu \in \Z$, $(\underline{k},\nu) \in \Xi_\nu$ and $\phi\in H^0(G(\Q),\cA(\underline{k},\nu))$ then we define $\bar\phi, \phi^\ast\in H^0(G(\Q),\cA(\underline{k},\nu))$ by:
$$\bar\phi(f)(g):=\overline{\phi(\bar f)(g)} \qquad \qquad \phi^\ast(f)(g):={\rm sign}(f)^{\nu}\cdot\phi(f)(g(\begin{smallmatrix}{0}&-1\\ \varpi_\cN& 0 \end{smallmatrix}))\cdot\chi(\det(g))^{-1}$$
here $f\in \cD(\underline{k},\nu)$ and $g\in G(\A)$, $\varpi_\cN=\prod_{v\mid \cN}\varpi_v^{v_v(\cN)}\in \A_F^\times$, $\varpi_v$ is a uniformizer of the finite place $v$, and ${\rm sign}(f)$ is $\pm 1$ if $f\in\cG_{\tau_0}f_{\pm k_{\tau_0}}\otimes\cP^{\tau_0}(\underline{k},\nu)$. Observe that if $\phi\in M_{(\underline{k},\nu)}(\Gamma_1(\cN),\chi)$ then $\bar\phi, \phi^\ast\in M_{(\underline{k},\nu)}(\Gamma_1(\cN),\chi^{-1})$.

\begin{defi} Write $\underline{k}=(k_{\tau_0},\underline{k}^{\tau_0})\in \N[\Sigma_F]$ and let $\phi_1,\phi_2\in H^0(G(\Q),\cA(\underline{k},\nu))$. Assume that $\phi_i\mid_{\A_F^\times}=\varepsilon$, where $\varepsilon=\chi\circ|\cdot|^{\nu}$, for a finite character $\chi$. We define the Hermitian inner product:
\[
\langle\phi_1,\phi_2\rangle:=\int_{G(\Q)\backslash G(\A)/\A_F^\times}\overline{\phi_1}\phi_2(f_{k_{\tau_0}}\otimes\Upsilon^{\tau_0})(g)\;|{\rm det}(g)|^{-\nu}dg,
\] 
where $\Upsilon^{\tau_0}:=|\begin{smallmatrix}X_1&Y_1\\X_2&Y_2\end{smallmatrix}|^{\underline{k}^{\tau_0}}$. 
\end{defi}

Returning to the notations from the beginning of this \S, for  $i= 1, 2, 3$ we suppose $\pi_{i,\infty}\simeq \cD(\underline{k}_i,\nu_i)$ for with $\underline{k}\in\N_{\geq 2}[\Sigma_F]$ and $(\pi_i)_f^{K_1^B(\cN_i)}\simeq\C$ for some ideal $\mathfrak{n}_i$ prime to $\mathrm{disc}(B)$. This implies that we can realize $\pi_i\mid_{G(\A_f)}$ inside the space $H^0(G(\Q),\cA(\underline{k}_i,\nu_i))$ and we denote a generator by $\phi_{i}^0\in M_{(\underline{k}_i,\nu_i)}(\Gamma_1(\cN_i),\chi_i)$ for a certain character $\chi_i$. Moreover, the contragredient representation $\tilde\pi_i^B$ is generated by $\overline{\phi_i^0}|\det|^{-\nu_i}\in M_{(\underline{k}_i,-\nu_i)}(\Gamma_1(\cN_i),\chi_i^{-1})$. 

\begin{lemma}\label{lemmabarnobar}
For $i= 1, 2, 3$ there exists $c_i\in\C^\times$ such that $\overline{\phi_i^0}=c_i\cdot(\phi_i^0)^\ast$.
\end{lemma}
\begin{proof}
The result follows from the fact that both $\bar\phi_i^0$ and $(\phi_i^0)^\ast$ generate irreducible automorphic representations with the same Hecke eigenvalues at finite places not dividing $\cN_i{\rm disc}(B)$.
\end{proof}

We denote $\underline{m}= (m_\tau)_{\tau \in \Sigma_F}:= (\frac{1}{2}(k_{1,\tau}+k_{2,\tau}+k_{3,\tau}))_{\tau \in \Sigma_F}$, $m_{3, \tau_0}= (k_{3, \tau_0}-m_{\tau_0})$ and for $i= 1, 2, 3$ we denote $\underline{m}_i^{\tau_0}= (m_{i,\tau})_{\tau \neq \tau_0}= (m_\tau- k_{i, \tau})_{\tau \ne \tau_0}$.

\begin{prop}\label{IntProp}
Let $\cN={\rm mcm}(\cN_1,\cN_2,\cN_3)$. Assume that $\nu_3=\nu_1+\nu_2$ and $\underline{k}_1$, $\underline{k}_2$ and $\underline{k}_3$ are unbalanced at $\tau_0$ with dominant weight $\underline{k}_3$. 
Then there are test vectors $\phi_i\in H^0(G(\Q),\cA(\underline{k}_i,\nu_i))^{K_1^B(\cN)}$ of $\pi_i\mid_{G(\A_f)}$ for $i=1,2,3$, such that
\[
\langle\phi_3,t(\phi_1,\phi_2)\rangle^2=C\cdot C(\phi_1,\phi_2,\phi_3)\cdot\frac{(-1)^{\nu_3}}{2^{4- 2m_{3, \tau_0}}}\cdot\binom{k_{3,\tau_0}-2}{k_{2,\tau_0}+ m_{3, \tau_0}-1}^2\cdot L\left(\frac{1-\nu_1-\nu_2-\nu_3}{2},\Pi_1\otimes\Pi_2\otimes\Pi_3\right),
\]
here $t$ is the trilinear product introduced in \S\ref{ss:trilinear products}, $C$ is a non-zero constant independent of $(k_{1,\tau_0}, k_{2,\tau_0}, k_{3,\tau_0})$ and $C(\phi_1,\phi_2,\phi_3)=\frac{\langle\phi_1,\phi_1\rangle\langle\phi_2,\phi_2\rangle\langle\phi_3,\phi_3\rangle}{L(1,\Pi_1,{\rm Ad})L(1,\Pi_2,{\rm Ad})L(1,\Pi_3,{\rm Ad})}$.
\end{prop}
\begin{proof}
By definition, we have:
\[
\langle\phi_3,t(\phi_1,\phi_2)\rangle=\int_{G(\Q)\backslash G(\A)/\A_F^\times}\overline{\phi_3}\cdot t(\phi_1,\phi_2)(f_{k_{3,\tau_0}}\otimes \Upsilon^{\tau_0})(g)\;|{\rm det}(g)|^{-\nu_3}dg.
\]
The product of the morphisms $t_{\tau}$ induces a morphism:
\[
t^{\tau_0}=\bigotimes_{\tau\neq\tau_0}t_\tau:\cP^{\tau_0}(\underline{k}_3,\nu_3)\longrightarrow \cP^{\tau_0}(\underline{k}_1,\nu_1)\otimes \cP^{\tau_0}(\underline{k}_2,\nu_2).
\]
 Then by definition $t^{\tau_0}(\Upsilon^{\tau_0})
=\Delta^{\tau_0}\in \cP^{\tau_0}(\underline{k}_1,\nu_1)\otimes \cP^{\tau_0}(\underline{k}_2,\nu_2)\otimes \cP^{\tau_0}(\underline{k}_3,\nu_3)$, where:
\begin{equation}\label{e: important polynomial appearing a lot}
\Delta^{\tau_0}:=\left|\begin{array}{cc}x_3&y_3\\x_2&y_2\end{array}\right|^{\underline{m}_1^{\tau_0}}\left|\begin{array}{cc}x_3&y_3\\x_1&y_1\end{array}\right|^{\underline{m}_2^{\tau_0}}\left|\begin{array}{cc}x_1&y_1\\x_2&y_2\end{array}\right|^{\underline{m}_3^{\tau_0}}.
\end{equation}
This implies that
\[
\langle\phi_3,t(\phi_1,\phi_2)\rangle=\int_{G(\Q)\backslash G(\A)/\A_F^\times}(\bar\phi_3\phi_1\phi_2)(f_{-k_{3,\tau_0}}\otimes t_{\tau_0}(f_{k_{3,\tau_0}})\otimes \Delta^{\tau_0})(g)\;|{\rm det}(g)|^{-\nu_3}dg.
\]
Write $v_\infty:=f_{-k_{3,\tau_0}}\otimes t_{\tau_0}(f_{k_{3,\tau_0}})\otimes \Delta^{\tau_0}\in \cD(\underline{k}_1,\nu_1)\otimes\cD(\underline{k}_2,\nu_2)\otimes\cD(\underline{k}_3,\nu_3)$. 

Let us consider $\tilde\phi_i\in M_{(\underline{k}_i,-\nu_i)}(\Gamma_1(\cN_i),\chi_i^{-1})$, defined by $\tilde\phi_i(f):={\rm sign}(f)^{\nu_i}\cdot\phi_i(f)\cdot(\epsilon_i\circ\det)^{-1}$. By Lemma \ref{lemmabarnobar} and the non-degeneracy of the inner product $\langle\cdot,\cdot\rangle$, we have that $\tilde\phi_i\in \tilde\pi_i^B$. It is clear by definition that 
$\langle\phi_3,t(\phi_1,\phi_2)\rangle=\langle\tilde\phi_3,t(\tilde\phi_1,\tilde\phi_2)\rangle$. Hence since $\widetilde{\bar\phi_3}=(-1)^{\nu_3}\overline{\tilde\phi_3}$, one obtains
\[
\langle\phi_3,t(\phi_1,\phi_2)\rangle^2=(-1)^{\nu_3}\int\int_{G(\Q)\backslash G(\A)/\A_F^\times}(\bar\phi_3\phi_1\phi_2)(v_\infty)(g_1)(\widetilde{\bar\phi_3}\tilde\phi_1\tilde\phi_2)(\tilde v_\infty)(g_2)\;|{\rm det}(g_1^{-1}g_2)|^{\nu_3}dg_1dg_2,
\]
where $\tilde v_\infty\in\cD(\underline{k}_1,-\nu_1)\otimes\cD(\underline{k}_2,-\nu_2)\otimes\cD(\underline{k}_3,-\nu_3)$ is defined analogously.
Notice that, again by Lemma \ref{lemmabarnobar}, we can see $\varphi=(\bar\phi_3\phi_1\phi_2)(v_\infty)|\det|^{-\nu_3}$ as an element of $\pi$, and we can see $\tilde\varphi=(\widetilde{\bar\phi_3}\tilde\phi_1\tilde\phi_2)(\tilde v_\infty)|\det|^{\nu_3}$ as an element of $\tilde\pi$. Hence we can apply Ichino's formula (Proposition \ref{IchinoProp}) to obtain that
\[
\frac{\langle\phi_3,t(\phi_1,\phi_2)\rangle^2}{((\bar\phi_3^0\phi_1^0\phi_2^0)(v_\infty^0),\tau_\cN(\widetilde{\bar\phi_3^0}\tilde\phi_1^0\tilde\phi_2^0)(\tilde v_\infty^0))}=\frac{(-1)^{\nu_3}}{2^3}\cdot\xi_F(2)^2\cdot\frac{L(1/2,\Pi)}{L(1,\Pi,{\rm Ad})}\cdot\prod_v I_v(\varphi_{v}\otimes\tilde\varphi_{v}),
\]
where $\tau_\cN=(\tau_{\cN_1},\tau_{\cN_2},\tau_{\cN_3})\in G(\A\otimes\A\otimes\A)$ with $\tau_{\cN_i}=\left(\begin{array}{cc}&-1\\\varpi_{\cN_i}&\end{array}\right)$,  
\[
v_\infty^0\otimes\tilde v_\infty^0:=(f_{k_{1,\tau_0}}\otimes \tilde f_{-k_{1,\tau_0}})\otimes (f_{k_{2,\tau_0}}\otimes \tilde f_{-k_{2,\tau_0}})\otimes ( f_{-k_{3,\tau_0}}\otimes\tilde f_{k_{3,\tau_0}})\otimes \Upsilon_1^{\tau_0}\otimes\Upsilon^{\tau_0}_2\otimes\Upsilon^{\tau_0}_3\in \pi_\infty\otimes\tilde\pi_\infty,
\]
and the local terms $I_v(\varphi_{v}\otimes\tilde\varphi_{v})=\frac{ L_v(1,\Pi_v,{\rm Ad})}{\xi_{F_v}(2)^2L_v(1/2,\Pi_v)}\cdot J_v$ with
\begin{eqnarray*}
J_v&=&\int_{F_v^\times\backslash B_v^\times}\frac{(\pi_{1,v}(b_v)\phi_{1,v},\tilde\phi_{1,v})_v(\pi_{2,v}(b_v)\phi_{2,v},\tilde\phi_{2,v})_v(\tilde\pi_{3,v}(b_v)\bar\phi_{3,v},\widetilde{\bar\phi_{3,v}})_v db_v}{(\phi_{1,v},\tau_{\cN_1,v}\tilde\phi_{1,v})_v(\phi_{2,v},\tau_{\cN_2,v}\tilde\phi_{2,v})_v(\bar\phi_{3,v},\tau_{\cN_3,v}\widetilde{\bar\phi_{3,v}})_v},\quad (v\nmid\infty)\\
J_{\tau_0}&=&\int_{F_{\tau_0}^\times\backslash B_{\tau_0}^\times}\frac{( b_{\tau_0}\tilde f_{-k_{3,\tau_0}},f_{-k_{3,\tau_0}})_{\tau_0}((b_{\tau_0},b_{\tau_0}) t_{\tau_0}(f_{k_{3,\tau_0}}),\widetilde{t_{\tau_0}(f_{k_{3,\tau_0}})})_{\tau_0}}{(f_{-k_{3,\tau_0}},\tilde f_{k_{3,\tau_0}})_{\tau_0}(f_{k_{1,\tau_0}},\tilde f_{-k_{1,\tau_0}})_{\tau_0}(f_{k_{2,\tau_0}},\tilde f_{-k_{2,\tau_0}})_{\tau_0}} db_{\tau_0},\\
J_\tau&=&\int_{F_\tau^\times\backslash B_\tau^\times}\frac{((b_\tau,b_\tau,b_\tau)\Delta_\tau,\tilde\Delta_\tau)_\tau}{(\Upsilon_{\tau,1})_\tau(\Upsilon_{\tau,2})_\tau(\Upsilon_{\tau,3})_\tau} db_\tau,\quad (\tau\neq\tau_0).
\end{eqnarray*}
Here $\Delta_\tau$ and $\tilde\Delta_\tau$ are both equal to $\left|\begin{array}{cc}x_3&y_3\\x_2&y_2\end{array}\right|^{m_{1,\tau}}\left|\begin{array}{cc}x_3&y_3\\x_1&y_1\end{array}\right|^{m_{2,\tau}}\left|\begin{array}{cc}x_1&y_1\\x_2&y_2\end{array}\right|^{m_{3,\tau}}$ as elements in $\cP_\tau(k_{1,\tau},\nu_1)\otimes \cP_\tau(k_{2,\tau},\nu_2)\otimes \cP_{\tau}(k_{3,\tau},-\nu_3)$ and $\cP_\tau(k_{1,\tau},-\nu_1)\otimes \cP_\tau(k_{2,\tau},-\nu_2)\otimes \cP_{\tau}(k_{3,\tau},\nu_3)$, respectively, and 
\begin{eqnarray*}
t_{\tau_0}(f_{k_{3,\tau_0}})\otimes\tilde f_{-k_{3,\tau_0}}&\in&  \cD(\underline{k}_1,\nu_1)\otimes\cD(\underline{k}_2,\nu_2)\otimes\cD(\underline{k}_3,-\nu_3),\\
 \widetilde{t_{\tau_0}(f_{k_{3,\tau_0}})}\otimes f_{-k_{3,\tau_0}}&\in&  \cD(\underline{k}_1,-\nu_1)\otimes\cD(\underline{k}_2,-\nu_2)\otimes\cD(\underline{k}_3,\nu_3).
\end{eqnarray*}

By the $B_\tau^\times$-invariance of $\Delta_\tau$ we have that $J_\tau=\frac{(\Delta_\tau,\tilde\Delta_\tau)_\tau}{(\Upsilon_{\tau,1})_\tau(\Upsilon_{\tau,2})_\tau(\Upsilon_{\tau,3})_\tau}$. Since we can easily compute that $(\Upsilon_{\tau,i})_\tau=k_{i,\tau}+1$, we obtain by \cite[Lemma 4.12]{hsieh18} (see also \cite[\S 4.9]{hsieh18})
\[
I_\tau(\varphi_{\tau}\otimes\tilde\varphi_{\tau})=\frac{L_\tau(1,\Pi_\tau,{\rm Ad})}{\xi_{F_\tau}(2)^2 L_\tau(1/2,\Pi_\tau)}\cdot\frac{\frac{(m_{\tau}+1)!m_{1,\tau}!m_{2,\tau}!m_{3,\tau}!}{k_{1,\tau}!k_{2,\tau}!k_{3,\tau}!}}{(k_{1,\tau}+1)(k_{2,\tau}+1)(k_{3,\tau}+1)}=\pi^{-1}. 
\]

To compute $J_{\tau_0}$ notice that \cite{Pra90} and 
\cite{Loke} the space ${\rm Hom}_{(\cG_{\tau_0},O(2))}(\cD(\underline{k}_1,\nu_1)\otimes\cD(\underline{k}_2,\nu_2)\otimes\cD(\underline{k}_3,-\nu_3),\C)$ is one dimensional, hence we have that
\[
\int_{F_{\tau_0}^\times\backslash B_{\tau_0}^\times}( b_{\tau_0}f_1,\tilde f_1)_{\tau_0}( b_{\tau_0}f_2,\tilde f_2)_{\tau_0}( b_{\tau_0}\tilde f_3,f_3)_{\tau_0} db_{\tau_0}=C\cdot( f_1\otimes f_2,t_{\tau_0}(\tilde f_3))_{\tau_0}\cdot(\tilde f_1\otimes\tilde f_2, t_{\tau_0}(f_3))_{\tau_0},
\]
for $f_1\otimes f_2\otimes \tilde f_3\in  \cD(\underline{k}_1,\nu_1)\otimes\cD(\underline{k}_2,\nu_2)\otimes\cD(\underline{k}_3,-\nu_3)$, $\tilde f_1\otimes \tilde f_2\otimes f_3\in  \cD(\underline{k}_1,-\nu_1)\otimes\cD(\underline{k}_2,-\nu_2)\otimes\cD(\underline{k}_3,\nu_3)$, and some constant $C$ (depending on $k_{i,\tau_0}$). We can compute $C$ by considering $f_1\otimes f_2 \otimes \tilde f_3=f_{k_{1,\tau_0}}\otimes R^{m_{3, \tau_0}}f_{k_{2,\tau_0}}\otimes\tilde f_{-k_{3,\tau_0}}$ and $\tilde f_1\otimes\tilde f_2 \otimes f_3=\tilde f_{k_{1,\tau_0}}\otimes R^{m_{3, \tau_0}}\tilde f_{k_{2,\tau_0}}\otimes f_{-k_{3,\tau_0}}$. Indeed by \cite[Lemma 3.11]{hsieh18} the left-hand-side is equal to 
\[
\frac{\xi_{F_{\tau_0}}(2)^2L_{\tau_0}(1/2,\Pi_{\tau_0})}{ L_{\tau_0}(1,\Pi_{\tau_0},{\rm Ad})}\cdot 2^{1- 2m_{3, \tau_0}}, 
\]
while the right-hand-side can be computed using the definition of $t_{\tau_0}$ given in \eqref{tripleD}, and the fact that $\langle f_{k_{i,\tau_0}},\tilde f_{-{k_{i,\tau_0}}}\rangle_{\tau_0}=1$, $\langle Lf,\tilde f\rangle_{\tau_0}=-\langle f,L\tilde f\rangle_{\tau_0}$ and $LR^jf_{k_{i,\tau_0}}=j(k_{i,\tau_0}+j-1)R^{j-1}f_{k_{i,\tau_0}}$ (see equality \eqref{eqepsiGM}). We obtain that 
\begin{eqnarray*}
\langle f_{k_{1,\tau_0}}\otimes R^{m_{3, \tau_0}}f_{k_{2,\tau_0}},t_{\tau_0}(\tilde f_{-{k_{3,\tau_0}}})\rangle_{\tau_0}&=&\binom{m-2}{k_{1,\tau_0}-1}\langle f_{k_{1,\tau_0}}\otimes R^{m_{3, \tau_0}}f_{k_{2,\tau_0}},\tilde f_{-{k_{1,\tau_0}}}\otimes L^{m_{3, \tau_0}}\tilde f_{-k_{2,\tau_0}}\rangle_{\tau_0}\\
&=&(-1)^{m_{3, \tau_0}}\frac{(m_{3, \tau_0})!(m-2)!}{(k_{1,\tau_0}-1)!(k_{2,\tau_0}-1)!},
\end{eqnarray*}
and analogously for $\langle\tilde f_{k_{1,\tau_0}}\otimes R^{m_{3, \tau_0}}\tilde f_{k_{2,\tau_0}},t_{\tau_0}(f_{-k_{3,\tau_0}})\rangle_{\tau_0}$. Moreover, we can compute similarly:
\begin{eqnarray*}
\langle t_{\tau_0}(f_{k_{3,\tau_0}}),t_{\tau_0}(\tilde f_{-{k_{3,\tau_0}}})\rangle_{\tau_0}&=&\sum_{j}C_j^2\langle R^j(f_{k_{1,\tau_0}}),L^j(f_{-k_{1,\tau_0}})\rangle_{\tau_0}\langle R^{m_{3, \tau_0}-j}(f_{k_{2,\tau_0}}),L^{m_{3, \tau_0}-j}(f_{-k_{2,\tau_0}})\rangle_{\tau_0}\\
&=&(-1)^{m_{3, \tau_0}}\frac{(m-2)!(m_{3, \tau_0})!}{(k_{1,\tau_0}-1)!(k_{2,\tau_0}-1)!}\sum_jC_j,
\end{eqnarray*}
where $C_j=\binom{m_{3, \tau_0}}{j}\binom{m-2}{k_{1,\tau_0}+j-1}$. Since $\sum_{j=0}^n\binom{A}{j}\binom{B}{n-j}=\binom{A+B}{n}$, we conclude that
\[
\langle t_{\tau_0}(f_{k_{3,\tau_0}}),t_{\tau_0}(\tilde f_{-{k_{3,\tau_0}}})\rangle_{\tau_0}=(-1)^{m_{3, \tau_0}}\frac{(m-2)!(m_{3, \tau_0})!}{(k_{1,\tau_0}-1)!(k_{2,\tau_0}-1)!}\binom{k_{3,\tau_0}-2}{k_{2,\tau_0}+ m_{3, \tau_0}-1}.
\]
Putting all this together, we conclude that 
\[
I_{\tau_0}(\varphi_{\tau_0}\otimes\tilde\varphi_{\tau_0})=2^{1- 2m_{3, \tau_0}}\cdot\binom{k_{3,\tau_0}-2}{k_{2,\tau_0}+ m_{3, \tau_0}-1}^2.
\]

Finally, we can choose the explicit test vectors provided in \cite{hsieh18} to obtain that $I_v(\varphi_v\otimes\tilde\varphi_v)$ is an explicit constant, which is equal to 1 if $v\nmid\cN$ (\cite[Lemma 3.11]{hsieh18}). We obtain that 
\[
\frac{\langle\phi_3,t(\phi_1,\phi_2)\rangle^2}{((\bar\phi_3^0\phi_1^0\phi_2^0)(v_\infty^0),\tau_\cN(\widetilde{\bar\phi_3^0}\tilde\phi_1^0\tilde\phi_2^0)(\tilde v_\infty^0))}=C\cdot\pi^{1-d}\cdot\frac{(-1)^{\nu_3}\cdot\xi_F(2)^2}{2^{4 - 2m_{3, \tau_0}}}\cdot\binom{k_{3,\tau_0}-2}{k_{2,\tau_0}+ m_{3, \tau_0}-1}^2\cdot\frac{L(1/2,\Pi)}{L(1,\Pi,{\rm Ad})},
\]
for some constant $C$ not depending on $k_{i,\tau}$.

The result follows from the fact that $L(\frac{1}{2},\Pi)=L(\frac{1-\nu_1-\nu_2-\nu_3}{2},\Pi_1\otimes\Pi_2\otimes\Pi_3)$ and the denominator $((\bar\phi_3^0\phi_1^0\phi_2^0)(v_\infty^0),\tau_\cN(\widetilde{\bar\phi_3^0}\tilde\phi_1^0\tilde\phi_2^0)(\tilde v_\infty^0))$ is (up-to-constant) $\langle\phi_1,\phi_1\rangle\langle\phi_2,\phi_2\rangle\langle\phi_3,\phi_3\rangle$ by Lemma \ref{lemmabarnobar}.
\end{proof}

\section{Modular forms for $G'$} In this section we introduce unitary Shimura curves. The main reason to introduce these curves is a well behaved moduli interpretation that they satisfies. We define the sheaves which give raise to modular forms for these curves. Moreover, we interpret the triple product defined in \S\ref{s:Automorphic forms for G} in more geometric terms. 

\subsection{Unitary Shimura curves}\label{ss:unitary Shimura curves} 
Let $\cO_B\subset B$ be a maximal order and let $\cO_D:=\cO_B\otimes_{\cO_F}\cO_E$. Notice that $\cO_D\subset D$ may not be maximal in general. In fact, since ${\rm disc}(D)$ is an ideal of $\cO_F$ dividing ${\rm disc}(B)$, for each ideal $\mathfrak{q}$ of $F$ over a odd prime and inert in $E$ such that $\mathfrak{q}\mid {\rm disc}(B){\rm disc}(D)^{-1}$ then $(\cO_D)_{\mathfrak{q}}$ is an Eichler order of level $\mathfrak{q}$. Nevertheless, $\cO_D$ is maximal (locally) at every prime not dividing ${\rm disc}(B){\rm disc}(D)^{-1}$. We denote also by $G_D$ the algebraic group attached to $\cO_D$, namely, $G_D(R):=(\cO_D\otimes_{\Z}R)^\times$, for any $\Z$-algebra $R$.

Let $\cN$ be an integral ideal of $F$ prime to $\mathrm{disc}(B)$ and consider the open compact subgroup of $G_D(\hat\Z)$: $$K_1^D(\cN ):=\left\{\left(\begin{smallmatrix}a&b\\c&d\end{smallmatrix}\right)\in G_D(\hat\Z):\; c\equiv d-1\equiv 0\;{\rm mod}\;\cN\cO_E\right\}.$$  Moreover, we denote $K_1^B(\cN):=K_1^D(\cN)\cap G(\hat\Z)$,  $K_{1}^{'}(\cN ):=K_1^D(\cN )\cap G'(\A_f)$ and  $K_{1,1}^B(\cN ):=K_1^D(\cN )\cap G^\ast(\A_f).$
Since $\pi_T(K_{1}^{'}(\cN ))=(\hat\cO_F+\cN\hat\cO_E)^\times/\hat\cO_F^\times\subseteq T_E(\hat\Z)$ then from \eqref{eqdescGG'} we have $K_{1}^{'}(\cN )/ K_{1,1}^B(\cN )\cong (\hat\cO_F+\cN\hat\cO_E)^\times/\hat\cO_F^\times$. We denote $\Pic(E/F,\cN):= T_E(\A_f)/\left[ (\hat\cO_F+\cN\hat\cO_E)^\times/\hat\cO_F^\times\right] T_E(\Q)$. For each $t\in \Pic(E/F,\cN)$ we fix a representative $b_t t\in G'(\A_f)$ under $\pi_T$ and we denote $\Gamma_{1,1}^t(\cN):=G^\ast(\Q)_+\cap b_t K_{1,1}^B(\cN)b_t^{-1}$ and $\Gamma_{1}^t(\cN):=G(\Q)_+\cap b_tK_1^B(\cN)b_t^{-1}$. 
 
 We define the \emph{unitary Shimura curve} over $\C$ of level $K_1'(\cN )$ as:
\begin{equation}\label{ShimCurv}
X(\C):=G'(\Q)_+\backslash (\dH\times G'(\A_f)/K_{1}^{'}(\cN ))= \bigsqcup_{t\in\Pic(E/F,\cN)}\Gamma_{1,1}^{t}(\cN)\backslash\dH .
\end{equation}
the last decomposition  comes from the fact that by strong approximation $\pi_T$ induces an isomorphism $\pi_T:G'(\Q)_+\backslash G'(\A_f)/K^{'}_{1}(\cN)\stackrel{\simeq}{\longrightarrow}\Pic(E/F,\cN)$.

\begin{defi}
We can define analogously $K_0^D(\cN)$, $K_0^B(\cN)$, $K_0'(\cN)$, $K_{0,1}^B(\cN)$, $\Gamma_{0,1}^t(\cN)$ and $\Gamma_0^t(\cN)$. For any $\dP$ coprime with $\cN$, write also $K'_1(\cN,\dP):=K_1'(\cN)\cap K_0'(\dP)$, and similarly for  $\Gamma_{1,1}^t(\cN,\dP)$ and $\Gamma_{1}^t(\cN,\dP)$.
\end{defi}

In order to deal with the moduli interpretation of $X$, we need to define 
$\cO_{B,\cM}\subseteq \cO_B$ an Eichler order of a well chosen level $\cM\mid\cN$ such that $K_1^B(\cN)\subseteq\hat\cO_{B,\cM}^\times$ and $\cO_{\cM}:=\cO_{B,\cM}\otimes_{\cO_F}\cO_E\subset D$. We write $V_{\Z}:=\cO_\cM\subset V$.

\begin{lemma}\label{choiceTheta}
		For a good choice of $\delta\in B$ and possibly enlarging $\cM$, the involution $l\mapsto l^\ast$ stabilizes $\cO_\cM$ and 
		\[
		\cD_{E/F}^{-1}\cO_\cM=\{v\in V:\;\Theta(v,w)\in\Z,\mbox{ for all }w\in \cO_\cM\},
		\]
		where $\cD_{E/F}^{-1}=\{e\in E:\;\mathrm{Tr}_{E/F}(eo)\in \cO_F,\mbox{ for all }o\in\cO_E\}\subset E$. In particular,
		$\Theta$ has integer values restricted to $V_{\Z}$. 
	\end{lemma}
	\begin{proof}		Recall that $l^\ast=\delta^{-1}\bar l\delta$. Since the involution $l\mapsto\bar l$ stabilizes $\cO_\cM$, we only have to check that $\delta\cO_\cM$ is a bilateral ideal. Since we assume $\delta\in B$, we have to check that $\delta\cO_{B,\cM}$ is a bilateral ideal.
		Let us consider the ideal		\[
		\mathfrak{I}=\{b\in B:\;\mathrm{Tr}_{B/F}(b\alpha)\in \cO_F;\mbox{ for all }\alpha\in\cD_{F/\Q}\cO_{B,\cM}\},		\]
		where $\cD_{F/\Q}$ is the different of $F/\Q$. By \cite[Lemme 4.7(1)]{Vig} $\mathfrak{I}$ is a bilateral ideal of norm $\cM^{-1}{\rm disc}(B)^{-1}\cD_{F/\Q}^{-2}$, hence
		(possibly enlarging $\cM$) we can assume that $\cM^{-1}{\rm disc}(B)^{-1}\cD_{F/\Q}^{-2}$ is principal generated by $d\in F_{<0}$ and, by strong approximation, $\mathfrak{I}=\delta\cO_\cM$. It is easy to check locally that $\delta$ can be chosen to satisfy $\delta^2=d$, hence $\bar\delta=-\delta$.
		
		Since $\cD_{E/\Q}=\cD_{E/F}\cD_{F/\Q}$, we obtain		
		\begin{eqnarray*}
		\cD_{E/F}^{-1}\cO_\cM&=&\{w\in D:\;\mathrm{Tr}_{D/E}(w\bar v\delta)\in\cD_{E/F}^{-1}\cD_{F/\Q}^{-1},\mbox{ for all }v\in\cO_{\cM}\}\\&=&\{w\in D:\;\Theta(d,w)\in\Z,\mbox{ for all }v\in \cO_{\cM}\},
		\end{eqnarray*}		hence the result follows.\end{proof}


Let $L/E$ be a finite extension such that $B\otimes_F L=\M_2(L)$. By \cite[\S 2.3]{carayol86} the Riemann surface $X(\C)$ has a model denoted $X$ defined over $L$. This curve solves the following moduli problem:  if $R$ is a $L$-algebra then $X(R)$ corresponds to the set of the isomorphism classes of tuples $(A,\iota, \theta, \alpha)$ where: 
	\begin{itemize}
		\item[$(i)$] $A$ is an abelian scheme over $R$  of relative dimension $4d$. 
		\item[$(ii)$] $\iota: \cO_\cM \rightarrow \mathrm{End}_R(A)$ gives an action of the ring $\cO_\cM$ on $A$  such that ${\rm Lie}(A)^{-,1}$ is of rank 1 and the action of $\cO_F$ factors through $\cO_F\subset E\subseteq L$.
		\item[$(iii)$] A $\cO_\cM$-invariant homogeneous polarization $\theta$ of $A$ such that the Rosati involution sends $\iota(d)$ to $\iota(d^\ast)$.
		\item[$(iv)$]  A class $\alpha$ modulo $K_{1}^{'}(\cN)$ of $\cO_\cM$-linear symplectic similitudes $\alpha:\hat T(A)\stackrel{\simeq}{\rightarrow}\hat\cO_\cM$.
	\end{itemize}	


 \begin{rmk}\label{rmkonGammac} (The curves $X^{\cC}$) Let $[\cC]\in \Pic(\cO_K)$. Recalling that $\Pic(\cO_K)\simeq \A_{F,f}^\times/F_+^\times\hat\cO_F^\times$ we fix $b_\cC\in G(\A_f)$ such that $\det(b_\cC)= \cC$. We put $\Gamma_{1}^\cC(\cN):=G(\Q)_+\cap b_\cC K_1^B(\cN)b_\cC^{-1}$ and $\Gamma_{1,1}^\cC(\cN):=G^\ast(\Q)_+\cap b_\cC K_{1,1}^B(\cN)b_\cC^{-1}$. Thus $\Gamma_{1,1}^\cC(\cN)$ can be equal to $\Gamma_{1,1}^t(\cN)$ for some $t \in \mathrm{Pic}(E/F, \mathfrak{n})$ and viceversa.

We denote by $X^{\cC}$ the Shimura curve attached to the analogous moduli problem of $X$ but exchanging the order $\cO_\cM$ by $\cO_\cM^\cC:=b_\cC\hat\cO_\cM b_\cC^{-1}\cap D$. Then as above we can verify that the irreducible components of $X^{\cC}$ are in bijection with $\mathrm{Pic}(E/F, \mathfrak{n})$. Moreover, observe that $\Gamma_{1,1}^\cC(\cN)\backslash \dH$ naturally appears as an irreducible component $X^{\cC}(\C)$.

We can give an alternative description of the points of $X^{\cC}$. In fact let $(A,\iota,\theta,\alpha)$ be a point in $X^{\cC}$ then there exists a unique isogenous pair $(A_0,\iota_0)$ with multiplication by $\cO_\cM$ such that $\ker(A\rightarrow A_0)=\{P\in A:\;bP=0,\;\mbox{for all }b\in b_\cC\hat\cO_\cM\cap D\}.$ From $\theta$ we naturally obtain a homogeneous polarization $\theta_0:A_0 \rightarrow A_0^{\vee}$. Moreover, the composition:
\[
\alpha_0:\hat T(A_0)\stackrel{\varphi}{\longrightarrow}\hat T(A)\stackrel{\alpha}{\longrightarrow}b_\cC\hat\cO_\cM b_\cC^{-1}\stackrel{\cdot\frac{b_\cC}{\deg\varphi}}{\longrightarrow}\frac{b_\cC}{\deg\varphi}\hat \cO_\cM,
\]
lies in $\hat\cO_\cM$, and provides a symplectic isomorphism between $(\hat\cO_\cM,\cC^{-1}\Theta)$ and $(\hat T(A_0),\theta_0)$. Thus $X^{\cC}$ classifies quadruples $(A_0,\iota_0,\theta_0,\alpha_0)$, where $(A_0,\iota_0,\theta_0)$ is as for $X$ and $\alpha_0$ is a symplectic isomorphism between $(\hat\cO_\cM,\cC^{-1}\Theta)$ and $(\hat T(A_0),\theta_0)$.
\end{rmk}

\begin{rmk} The universal abelian variety $\pi: A \rightarrow X$ can be described in a more explicit way over $\C$. In fact, we have:
\[
A=\bigsqcup_{t\in\Pic(E/F,\cN)}\Gamma_{1,1}^t(\cN)\backslash\left(\dH\times(\C^2\otimes_{F,\tau_0}E)\times \M_2(\C)^{\Sigma_F\smallsetminus \{\tau_0\}}\right)/D\cap t^{-1}\hat\cO_\cM b_t^{-1},
\]
where $m\otimes s \in D\cap t^{-1}\hat\cO_\cM b_t^{-1}$, $\gamma\in \Gamma_{1,1}^t(\cN)$ act on $(z,(v\otimes e),(M_\tau)_{\tau \neq \tau_0})\in \dH\times(\C^2\otimes_{F,\tau_0}E)\times \M_2(\C)^{\Sigma_F\smallsetminus \{\tau_0\}}$ by
\[
(z,(v\otimes e),(M_\tau)_{\tau \neq \tau_0})\cdot (m\otimes s)=  (z,(v\otimes e+\tau_0(m)\left(\begin{matrix}z\\ 1\end{matrix}\right)\otimes s),(M_\tau+\tilde\tau(m\otimes s))_{\tau \neq \tau_0}),
\]
\[
\gamma\cdot(z,(v\otimes e),(M_\tau)_{\tau \neq \tau_0})= (\gamma z,((cz+d)^{-1}v\otimes e),(M_\tau\gamma^{-1})_{\tau \neq \tau_0}).
\]

here $\tau_0(\gamma)= (\begin{smallmatrix}a&b\\c&d\end{smallmatrix})$ and we use the identifications $\tau_0:B\otimes_{F,\tau_0}\R\simeq\M_2(\R)$ and $\tilde\tau: D\otimes_{E,\tilde\tau}\C\simeq\M_2(\C)$ for $\tau \neq \tau_0$. 
 Its universal polarization is given by the restriction of the form $\Theta$ to $D\cap t^{-1}\hat\cO_\cM b_t^{-1}$, whose group of $\cO_\cM$-linear symplectic endomorphisms is $G'(\Q)\cap b_t\hat\cO_\cM b_t^{-1}$, and the class of $\cO_\cM$-linear symplectic similitudes $\alpha$ is given by $ \alpha:\hat T(A)=t^{-1}\hat\cO_\cM b_t^{-1}\stackrel{\simeq}{\longrightarrow}\hat\cO_\cM$ sending $t^{-1}bb_t^{-1}\longmapsto b$. Recall that $b_tt\in G'(\A_f)$, and notice that $K_{1}^{'}(\cN)\subseteq G'(\Q)\cap b_t\hat\cO_\cM b_t^{-1}$.\end{rmk}

\begin{rmk}\label{rmkonpts} Since $p \nmid{\rm disc}(B)$, a class $\alpha$ of $\cO_\cM$-linear symplectic similitudes $\alpha$ modulo $K_{1}^{'}(\cN)$ is decomposed as $\alpha=\alpha_p\times\alpha^p$ where:
	\[
	\alpha_p:T_p(A)\stackrel{\simeq}{\rightarrow}(\cO_\cM)_p\simeq\M_2(\cO_E\otimes\Z_p),\qquad \alpha^p:\hat T(A)^p\stackrel{\simeq}{\rightarrow}(\hat\cO_\cM)^p.
	\]
	
 $\Theta$ induces a perfect pairing on $V_{\Z_p}=(\cO_{\cM})_p=\M_2(\cO_{E}\otimes\Z_p)$. As $p$ splits in $\Q(\sqrt{\lambda})$ then $G'(\Z_p)\stackrel{\simeq}{\longrightarrow}\GL_2(\cO_F\otimes\Z_p)\times\Z_p^\times$ through $b(t_1,t_2)\longmapsto \left(bt_2,\det(b)t_1t_2\right)$  and this identifies $K_{1}^{'}(\cN)_p$ with $K_1^B(\cN)_p\times \Z_p^\times$. 
 
The morphism $\alpha_p$ identifies each $T_p(A)^\pm$ 
with a copy of $\M_2(\cO_{F}\otimes\Z_p)$. Hence $\alpha_p$ provides a $\M_2(\cO_{F}\otimes\Z_p)$-lineal isomorphism $\alpha_p^-:T_p(A)^-\stackrel{\simeq}{\rightarrow}\M_2(\cO_{F}\otimes\Z_p)$ and, reciprocally, a $\M_2(\cO_{F}\otimes\Z_p)$-lineal isomorphism $\alpha_p^-$ gives rise to a symplectic similitude $\alpha_p=(\alpha_p^+,\alpha_p^-)$, where $\alpha_p^+:T_p(A)^+\rightarrow\M_2(\cO_{F}\otimes\Z_p)$ is provided by the rule $\Theta(\alpha_p^+(v),\alpha_p^-(w))=e_p(v,w)$ with $e_p$ the corresponding perfect dual pairing on $T_p(A)$ (Note that $e_p$ is characterized by its image in $T_p(A)^+\times T_p(A)^-$ because, since the Rosati involution sends $\iota(\sqrt{\lambda})$ to $-\iota(\sqrt{\lambda})$, it vanishes at $T_p(A)^+\times T_p(A)^+$ and $T_p(A)^-\times T_p(A)^-$). The action of $(\gamma,n)\in \GL_2(\cO_{F}\otimes\Z_p)\times\Z_p^\times$ on $\alpha_p$ is given by 
	\[
	\alpha_p^-\longmapsto\alpha_p^-\cdot\gamma,\qquad \alpha_p^+\longmapsto n\alpha_p^+\cdot\bar\gamma^{-1}.
	\]
	Hence, to provide a $\alpha_p$ modulo $K_{1}^{'}(\cN)_p$ amounts to giving $\alpha_p^-$ modulo $K_1^B(\cN)_p$, or equivalently, the point  $P=(\alpha_p^-)^{-1}(\begin{smallmatrix}0&1 \\0&0\end{smallmatrix})({\rm mod}\;\cN)\in A[\cN\cO_{F}\otimes\Z_p]^{-,1}$ that generates a subgroup isomorphic to $(\cO_{F}\otimes\Z_p)/(\cN\cO_{F}\otimes\Z_p)$. We have an analogous description in case of $\Gamma_0(\cN)$-structure.
	
\end{rmk}

\subsection{Modular sheaves}\label{ss:modular sheaves} We introduce the sheaves which give rise to the modular forms for $G'$. Let $L_0/F$ be an extension such that $B\otimes_{F}L_0= \mathrm{M}_2(L_0)$, write $L=L_0(\sqrt{\lambda})\supset E$, and denote by  $X_L$ the  base change to $L$ of the unitary Shimura curve $X$. Using the universal abelian variety $\pi: A\rightarrow  X_L$ we define the following coherent sheaves on $X_L$:
\[
\omega:=\left(\pi_\ast\Omega^1_{A/X_L}\right)^{+,2} \qquad \omega_-:=\left(\left(R^1\pi_\ast\cO_A\right)^{+,2}\right)^\vee \qquad \cH:=\left(\cR^1\pi_\ast\Omega^\bullet_{A/X_L}\right)^{+,2}
\]
Note that we have $ \omega_-\simeq (\pi'_\ast\Omega^1_{A^\vee/X_L})^{+,2}$ and the sheaf  $\cH$ is endowed with a Gauss-Manin connection $\bigtriangledown:\cH\rightarrow\cH\otimes\Omega^1_{X_L}$. The natural $\cO_D$-equivariant exact sequence:
	\begin{equation}\label{exseqdR}
	0\longrightarrow \pi_\ast\Omega^1_{A/X_L} \longrightarrow \cR^1\pi_\ast\Omega^\bullet_{A/X_L} \longrightarrow R^1\pi_\ast\cO_A \longrightarrow 0,
	\end{equation}
induces the Hodge exact sequence (see \cite[\S 2.3.1]{ding17}):
\begin{equation} \label{hodge filtration}
0\longrightarrow \omega \longrightarrow\cH \longrightarrow\omega_-^\vee \longrightarrow 0..
\end{equation}

If $L$ contains the Galois closure of $F$ then the natural decomposition $F\otimes_{\Q}L\cong L^{\Sigma_F}$ induces:
\[
\omega= \bigoplus_{\tau \in \Sigma_F}\omega_{\tau} \qquad\qquad  \cH= \bigoplus_{\tau \in \Sigma_F}\cH_{\tau}\]
As the sheaves $(\Omega^1_{A/X_L})^{+,1}$ and $(\Omega^1_{A/X_L})^{+,2}$ are isomorphic then condition $(ii)(2)$ of the moduli problem of $X$ imply that $\omega_{\tau_0}$ is locally free of rank 1, while $\omega_\tau$ is of rank $2$ for $\tau \neq \tau_0$.  Moreover, since the Rosati involution maps $\sqrt{\lambda}$ to $-\sqrt{\lambda}$, 
we deduce that $\omega_-$ is locally free of rank 1. Thus, $\omega_\tau=\cH_\tau$ for for each $\tau\neq \tau_0$, and we have the exact sequence
\begin{equation}\label{Hodge1}
0\longrightarrow \omega_{\tau_0} \longrightarrow\cH_{\tau_0} \stackrel{\epsilon}{\longrightarrow}\omega_-^{\vee} \longrightarrow 0.
\end{equation}

Let $\underline{k}=(k_\tau)\in \N[\Sigma_F]$, we introduce the \emph{modular sheaves} over $X_L$ considered in this text:
$$\omega^{\underline{k}}:= \omega_{\tau_0}^{\otimes k_{\tau_0}}\otimes \bigotimes_{\tau \neq \tau_0}\Sym^{k_\tau}\omega_\tau$$
$$\cH^{\underline{k}}:=  \bigotimes_{\tau \in \Sigma_F} \Sym^{k_\tau}\cH_\tau= \Sym^{k_{\tau_0}}\cH_{\tau_0}\otimes \bigotimes_{\tau \neq \tau_0}\Sym^{k_\tau}\omega_\tau.$$


\begin{defi} A \emph{modular form} of weight $\underline{k}$ and coefficients in $L$ for $G'$ is a global section of $\omega^{\underline{k}}$, i.e. an element of $\mathrm{H}^{0}(X_L, \omega^{\underline{k}})$. 
\end{defi}


\subsection{Alternating pairings and Kodaira-Spencer}\label{alt-pair} 

Notice that $\cR^1\pi_\ast\Omega^\bullet_{A/X_L}$ is a $\cO_{\cM}\otimes_{\Z}\cO_{X_L}$-module of rank one, hence if we fix a generator $\underline{w}\in \cR^1\pi_\ast\Omega^\bullet_{A/X_L}$ (as $\cO_{\cM}\otimes_{\Z}\cO_{X_L}$-module) we can define the unique symplectic $(\cO_{B,\cM}\otimes_{\Z}\cO_{X_L}$)-linear involution
\[
c:\cR^1\pi_\ast\Omega^\bullet_{A/X_L}\longrightarrow \cR^1\pi_\ast\Omega^\bullet_{A/X_L},
\]
such that $c((b\otimes e)\underline{w})=(b\otimes\bar e)\underline{w}$ for all $b\in \cO_{B,\cM}$ and $e\in \cO_E$.

Since $\pi_\ast\Omega^1_{A/X_L}\simeq {\rm Lie}(A)^\vee$ and $R^1\pi_\ast\cO_A\simeq {\rm Lie}(A^\vee)$, the polarization $\theta :{\rm Lie}(A)\rightarrow{\rm Lie}(A^\vee)$ provides a $\cO_{X_L}$-linear morphism (see \cite[Remarque 2.3.1]{ding17}): 
\[
\Theta: \pi_\ast\Omega^1_{A/X_L}\times R^1\pi_\ast\cO_A\longrightarrow \cO_{X_L}
\]
satisfying $\Theta(\lambda x,y)=\Theta(x,\lambda^\ast y)$ for $\lambda\in\cO_\cM$. From  \eqref{exseqdR} it is equivalent to give an alternating pairing $\Theta$ on $\cR^1\pi_\ast\Omega^\bullet_{A/X_L}$.
The above involution $c$ together with $\Theta$, provide for each $\tau \in \Sigma_F$ an alternating pairing:
\[
\bar\varphi_\tau:\left(\cR^1\pi_\ast\Omega^\bullet_{A/X_L}\right)_\tau^{+,2}\times \left(\cR^1\pi_\ast\Omega^\bullet_{A/X_L}\right)_\tau^{+,2}\longrightarrow \cO_{X_L}
\]
given by $\bar\varphi_\tau(u,v):=\Theta(u,\delta^{-1}(\begin{smallmatrix}0&1\\1&0\end{smallmatrix})c(v))$. If $\tau\neq\tau_0$, this provides an isomorphism:
\begin{equation}\label{isowedge}
\varphi_\tau:\bigwedge^2\left(\cR^1\pi_\ast\Omega^\bullet_{A/X_L}\right)_\tau^{+,2} =\bigwedge^2 \omega_\tau\stackrel{\simeq}{\longrightarrow}\cO_{X_L}
\end{equation}
If $\tau=\tau_0$, we obtain an isomorphism:
\begin{equation}\label{isopol}
\varphi_{\tau_0}:\omega_{\tau_0}\stackrel{\simeq}{\longrightarrow}\omega_-
\end{equation}
given by $\bar\varphi_{\tau_0}(v,w)=\epsilon(v)(\varphi_{\tau_0}(w))$ for $w\in \omega_{\tau_0} \subset \cH_{\tau_0}$ and $v\in\cH_{\tau_0}=\left(\cR^1\pi_\ast\Omega^\bullet_{A/X_L}\right)_{\tau_0}^{+,2}$. Thus, the Kodaira-Spencer isomorphism (see \cite[Lemme 2.3.4]{ding17}) induces the isomorphism:
\[
KS:\Omega^1_{X_L}\stackrel{\simeq}{\longrightarrow} \omega_{\tau_0}\otimes\omega_-\stackrel{\varphi_{\tau_0}^{-1}}{\longrightarrow} \omega_{\tau_0}^{\otimes 2}.
\]


\subsection{Katz Modular forms}\label{KatzModForm}
Let $R_0$ be a $L$-algebra and section $f\in H^0(X/R_0,\omega^{\underline{k}})$. If $R$ is a $R_0$-algebra, $(A,\iota,\theta,\alpha)$ is a tuple corresponding to a point of $X(\Spec(R))$ and $w=(f_{0},(f_\tau,e_\tau)_{\tau\neq \tau_0})$ is a $R$-basis of $\omega_A=\left(\Omega^1_{A/R}\right)^{+,2}$, then there exists $f(A,\iota,\theta,\alpha,w)\in\bigotimes_{\tau\neq \tau_0}\Sym^{k_\tau}\left(R^2\right)$ such that: 
$$f(A,\iota,\theta,\alpha)=f(A,\iota,\theta,\alpha,w)((f_\tau,e_\tau)_{\tau \neq \tau_0})\cdot f_{0}^{^{\otimes}k_{\tau_0}}$$

Thus a section $f\in H^0(X/R_0,\omega^{\underline{k}})$ is characterized as a rule that assigns to any $R_0$-algebra $R$ and $(A,\iota,\theta,\alpha, w)$ over $R$ a polynomial:
\[
f(A,\iota,\theta,\alpha,w)\in\bigotimes_{\tau\neq \tau_0}\Sym^{k_\tau}\left(R^2\right)
\]
such that
\begin{itemize}
\item[(A1)] The element $f(A,\iota,\theta,\alpha,w)$ depends only on the $R$-isomorphism class of $(A,\iota,\theta,\alpha)$;

\item[(A2)] Formation of $f(A,\iota,\theta,\alpha,w)$ commutes with extensions $R\rightarrow R'$ of $R_0$-algebras;

\item[(A3)] If $(t,\underline{g})\in R^\times\times\GL_2(R)^{\Sigma_F\setminus\{\tau_0\}}$ and $w(t,\underline{g})=(te_{0},(f_\tau,e_\tau)g_\tau)$ then:
\[
f(A,\iota,\theta,\alpha,w(t,\underline{g}))=t^{-k_{\tau_0}}\cdot\left(\underline{g}^{-1} f(A,\iota,\theta,\alpha,w)\right).
\]
\end{itemize}
Considering the isomorphism $\varphi_\tau$ of \eqref{isowedge} we can give an alternative description of a section $f\in H^0(X/R_0,\omega^{\underline{k}})$  as a rule that assigns to any $R_0$-algebra $R$ and $(A,\iota,\theta,\alpha,w)$ as above a linear form 
\[
f(A,\iota,\theta,\alpha,w)\in\bigotimes_{\tau\neq {\tau_0}}\Sym^{k_\tau}\left(R^2\right)^\vee
\]
such that $f(A,\iota,\theta,\alpha)=f(A,\iota,\theta,\alpha,w)(P(\underline{x}, \underline{y}))f_{0}^{\otimes k_{\tau_0}}$ with $P(\underline{x}, \underline{y})= P((x_\tau, y_\tau)_{\tau\neq\tau_0})= $ $\prod_{\tau\neq {\tau_0}}|\begin{smallmatrix}f_\tau&e_\tau\\x_\tau&y_\tau\end{smallmatrix}|^{k_\tau}\varphi_{\tau}(f_\tau\wedge e_\tau)^{-k_\tau} \in \Sym^{k_\tau}\left(R^2\right)$. This rule satisfy:
\begin{itemize}
\item[(B1)] The element $f(A,\iota,\theta,\alpha,w)$ depends only on the $R$-isomorphism class of $(A,\iota,\theta,\alpha)$.

\item[(B2)] Formation of $f(A,\iota,\theta,\alpha,w)$ commutes with extensions $R\rightarrow R'$ of $R_0$-algebras.

\item[(B3)] If $(t, \underline{g})\in R^\times\times\GL_2(R)^{\Sigma_F\setminus\{\tau_0\}}$:
\[
f(A,\iota,\theta,\alpha,w(t,\underline{g}))=t^{-k_{\tau_0}}\cdot\left(\underline{g}^{-1} f(A,\iota,\theta,\alpha,w)\right).
\]
\end{itemize}

\begin{rmk}\label{DualvsNODual}
We have two interpretations of a global section as a Katz modular form coming from the fact that we have an isomorphism
\begin{equation}\label{morfastnoast}
\omega^{\underline{k}}\;\simeq \;\omega_{\tau_0}^{k_{\tau_0}}\otimes\bigotimes_{\tau\neq\tau_0}\left(\left(\Sym^{k_\tau}\omega_\tau\right)^\vee\otimes\left(\bigwedge^2\omega_\tau\right)^{\otimes k_\tau}\right)\;\stackrel{\varphi_{\tau}}{\simeq}\;\omega_{\tau_0}^{k_{\tau_0}}\otimes\bigotimes_{\tau\neq\tau_0}\left(\Sym^{k_\tau}\omega_\tau\right)^\vee,
\end{equation}
since we are over a field of zero characteristic. This won't be the case over other base schemes.
\end{rmk}

\subsection{Modular forms for $G'$ vs automorphic forms for $G$}\label{GvsG'}

When $L= \C$ we have the following more familiar interpretation for the space of modular forms for $G'$. 

If $t \in\Pic(E/F,\cN)$ then we denote by $M_{\underline{k}}(\Gamma_{1,1}^t(\cN),\C)$ the $\C$-vector space of holomorphic functions $f:\dH\rightarrow\bigotimes_{\tau \neq \tau_0}\cP_\tau(k_\tau)^\vee$ such that $f(\gamma z)=(cz+d)^{k_{\tau_0}}\gamma f(z)$ for all $\gamma\in \Gamma_{1,1}^t(\cN)$ where $\cP_\tau(k_\tau)= \cP_\tau(k_\tau, k_\tau)$ was introduced in \S\ref{AutQuaFrm} (Notice that $\det(\Gamma_{1,1}^t(\cN))\subseteq \cO_F^\times\cap\Q_+=1$, hence there is no action of the determinant).


\begin{lemma} \label{l:seccions as complex forms} We have a canonical isomorphism of $\C$-vector spaces:
	$$\mathrm{H}^{0}(X_{\C}, \omega^{\underline{k}})= \bigoplus_{t \in\Pic(E/F,\cN)}M_{\underline{k}}(\Gamma_{1,1}^t(\cN),\C)$$
\end{lemma}
\begin{proof} Firstly, for each $t \in\Pic(E/F,\cN)$ let  $X^{t}$ be the corresponding connected component of $X$. From the explicit description of the universal object over $\C$ we have ${\rm Lie}(A)=\left(\C^2\otimes_{F,\tau_0}E\right)\times \M_2(\C)^{\Sigma_F\smallsetminus\{\tau_0\}}$ and ${\rm Lie}(A)^2=\left(\begin{smallmatrix}0&0\\0&1\end{smallmatrix}\right){\rm Lie}(A)=\left(\C\otimes_{F,\tau_0}E\right)\times (\C^2)^{\Sigma_F\smallsetminus\{\tau_0\}}.$ Thus, $(\Omega^1_{A/X_{\C}})_{\tau_0}^{+,2}=({\rm Lie}(A)_{\tau_0}^{+,2})^\vee= \C dx_{\tau_0}$ where $dx_{\tau_0}:=\sqrt{\lambda}(dx\otimes1)+dx\otimes\sqrt{\lambda}\in (\C\otimes_{F,\tau_0}E)^\vee$
	and $\left(\Omega^1_{A/X_{\C}}\right)^{+,2}_{\tau}=\langle dx_{\tau},dy_{\tau}\rangle\in (\C^2)^\vee$ for $\tau \neq \tau_0$ . Then if $\varphi \in H^0(X_{\C},\omega^{\underline{k}})$ then for each $t \in\Pic(E/F,\cN)$ the restriction of $\varphi$ to $X^t$ is given by:
	\begin{equation}\label{e:sections vs modular forms}
	f(z)\left(\prod_{\tau \neq \tau_0}\left| \begin{smallmatrix}dx_\tau&dy_\tau\\X_\tau&Y_\tau\end{smallmatrix}\right|^{k_\tau}\right)dx_{\tau_0}^{k_{\tau_0}},
	\end{equation}
	where  $f:\dH\longrightarrow\bigotimes_{\tau \neq \tau_0}\cP_\tau(k_\tau)^\vee$ is a holomorphic function. Since $\gamma^\ast dx_{\tau_0}=(cz+d)^{-1}dx_{\tau_0}$ and $\gamma^\ast(dx_\tau,dy_\tau)=(dx_\tau,dy_\tau)\tau(\gamma)^{-1}$, for all $\gamma\in \Gamma_{1,1}^t(\cN)$ we deduce that $f(\gamma z)=(cz+d)^{k_\tau}\gamma f(z)$ for each $\gamma\in \Gamma_{1,1}^t(\cN).$ \end{proof}


Now we are going to relate modular forms with coefficients in $\C$ for the groups $G'$ and $G$. Fix $\nu \in \Z$ and $\underline{k} \in \Xi_\nu$.  Let $U_\cN= \{u \in \cO_F^\times | u\equiv 1 \ \mathrm{mod}(\cN)\}$  and define an action of the group:
$$\Delta:=(\cO_F)_+^\times/U_\cN^2$$ on $M_{\underline{k}}(\Gamma_{1,1}^t(\cN),\C)$ as follows: if $[s]\in \Delta$ we fix $\gamma_s\in \Gamma_{1}^t(\cN)$ such that $\det\gamma_s=s$, if we write $\tau_0\gamma_s=\left(\begin{smallmatrix}a&b\\c&d\end{smallmatrix}\right)$, then
for each $f\in M_{\underline{k}}(\Gamma_{1,1}^t(\cN),\C)$ we put:
\begin{equation}\label{eqactDelta}
s\ast_{\nu} f(z):= s^{\frac{-\underline{k}+ 2k_{\tau_0}\tau_0+ \nu\underline{1}}{2}}(cz+d)^{-k_{\tau_0}}\gamma_s^{-1} f(\gamma_sz).
\end{equation}
It is not hard to verify that this action is well defined and in fact we have $s\ast_\nu f\in M_{\underline{k}}(\Gamma_{1,1}^t(\cN),\C)$.
\begin{prop}\label{compautshe} We have a natural isomorphism of $\C$-vector spaces 
	\[
	\iota_{\underline{k}, \nu}: H^0(G(\Q),\cA(\underline{k},\nu))^{K_1^B(\cN)} \stackrel{\simeq}{\longrightarrow} \bigoplus_{\cC\in\Pic(\cO_F)}M_{\underline{k}}(\Gamma_{1,1}^\cC(\cN),\C)^\Delta
	\]
\end{prop}
\begin{proof} To give an automorphic form $\phi\in H^0(G(\Q),\cA(\underline{k},\nu))^{K_1^B(\cN)}$ is equivalent to give a holomorphic function $f_\phi:\dH\times G(\A_f)/K_1^B(\cN)\rightarrow\bigotimes_{\tau \neq \tau_0}\cP_\tau(k_\tau,\nu)^\vee$ as defined in \eqref{mod-aut}. By strong approximation we have an isomorphism $\det:G(\Q)^+\backslash G(\A_f)/K^B_1(\cN)\stackrel{\simeq}{\longrightarrow}\A_{F,f}^\times/F_+^\times\hat\cO_F^\times=\Pic(\cO_F).$
	Hence $f_\phi$ is characterized by the restrictions $f_\phi^\cC:=f_\phi(\cdot,b_\cC):\mathfrak{H}\rightarrow\bigotimes_{\tau\neq\tau_0}\cP_\tau(k_\tau,\nu)^\vee$,  for all $[\cC]\in\Pic(\cO_F)$. We can verify that $f_\phi^\cC(\gamma z)=\det\gamma^{-(\nu+k_{\tau_0})/2}(cz+d)^{k_{\tau_0}}\gamma f_\phi^\cC(\tau)$ if $\gamma\in \Gamma_1^\cC(\cN)$.  Since any $\gamma\in\Gamma_{1,1}^\cC(\cN)\subseteq \Gamma_1^\cC(\cN)$ has reduced norm 1, we have that $f_\phi^\cC\in M_{\underline{k}}(\Gamma_{1,1}^\cC(\cN),\C)$. 
	Since $\Gamma_1^\cC(\cN)/\Gamma_{11}^\cC(\cN)\simeq(\cO_F)_+^\times/U_\cN^2=\Delta$, the extra condition for $f_\phi^\cC$ to be $\Gamma_1^\cC(\cN)$-invariant is translated to being $\Delta$-invariant, and hence the result follows.
\end{proof}

 Now let $[s]\in \Delta$ as above and fix $\gamma_s\in\Gamma_1^\cC(\cN)$ such that $\det(\gamma_s)=s$. Then for a point $(A,\iota,\theta,\alpha)\in X^{\cC}$ we put $\gamma_s\alpha:\hat\cO_\cM\stackrel{\cdot\gamma_s}{\longrightarrow}\hat\cO_\cM\stackrel{\alpha}{\longrightarrow}\hat T(A)$ and:
$$[s]\ast(A,\iota,\theta,\alpha)=(A,\iota,s\theta,\gamma_s\alpha).$$ 
This is well defined as the class $\gamma_s\alpha$ does not depend on the choice of $\gamma_s$, and if $s=\eta^2$ for $\eta\in U_\cN$ then  $s\ast(A,\iota,\theta,\alpha)=(A,\iota,s\theta,\eta\alpha)=\iota(\eta)^\ast(A,\iota,\theta,\alpha)\simeq(A,\iota,\theta,\alpha).$ 
	
We define an action of $\Delta$  on $H^0(X^\cC,\omega^{\underline{k}})$ using the description of \S \ref{KatzModForm}. Let $f \in H^0(X^\cC,\omega^{\underline{k}})$ and a tuple $(A,\iota,\theta,\alpha,w)$ over some ring $R$ as in \S \ref{KatzModForm}.  
For $[s]\in \Delta$ we put: 
	\[
	(s\ast_{\nu}f)(A,\iota,\theta,\alpha,w):=s^{\frac{-\underline{k}+ 2k_{\tau_0}\tau_0+ \nu\underline{1}}{2}}\cdot f(A,\iota,s^{-1}\theta,\gamma_s^{-1}\alpha,w).
	\]
This is well defined as if $s=\eta^2\in U_\cN^2$, we have
\begin{eqnarray*}
	f(A,\iota,\theta,\alpha,w)&=&f(A,\iota,s^{-1}\theta,\gamma_s^{-1}\alpha,\eta^{-1} w)=\tau_0(\eta)^{k_{\tau_0}}\prod_{\tau\neq \tau_0}(\tau(\eta)^{-k_\tau})f(A,\iota,s^{-1}\theta,\gamma_s^{-1}\alpha, w)\\
		&=&{\rm N}_{F/\Q}(\eta)^{-\nu}s^{\frac{-\underline{k}+ 2k_{\tau_0}\tau_0+ \nu\underline{1}}{2}}f(A,\iota,s^{-1}\theta,\gamma_s^{-1}\alpha,w)=s^{\frac{-\underline{k}+ 2k_{\tau_0}\tau_0+ \nu\underline{1}}{2}}f(A,\iota,s^{-1}\theta,\gamma_s^{-1}\alpha,w).
\end{eqnarray*}
In Proposition \ref{proponDelta} of the appendix we prove that both descriptions of the $\Delta$-action coincide. Moreover in  \S \ref{Heckops}, we describe the action of the Hecke operators from the perspective of Katz modular forms due to the isomorphism $\iota_{\underline{k},\nu}$ of Proposition \ref{compautshe}. Thus we obtain: 
\begin{cor}\label{c:automorphic forms as sections} We have the following decomposition compatible with Hecke operators:
$$ H^0(G(\Q),\cA(\underline{k},\nu))^{K_1^B(\cN)} \stackrel{\simeq}{\longrightarrow} \bigoplus_{\cC\in\Pic(\cO_F)} H^{0}(X^{\cC, 0}_{\C}, \omega^{\underline{k}})^\Delta,$$
where $X^{\cC, 0}_{\C}$ is the irreducible component of $X^{\cC}_{\C}$ corresponding to $1 \in \mathrm{Pic(E/F, \mathfrak{n})}$.
\end{cor}


\subsection{Connections and trilinear products}\label{ss:connections and trilinear products} For $\underline{k}\in \mathbb{N}[\Sigma_F]$ and $m\in \Z$ we consider the sheaves: 
$$\cH^{\underline{k}}_m:=\omega_{0}^{k_{\tau_0}-m}\otimes{\rm Sym}^{m}\cH_{\tau_0}\otimes\bigotimes_{\tau \neq \tau_0}{\rm Sym}^{k_\tau}\omega_{\tau}.$$ 
Then from \eqref{Hodge1} we obtain the exact sequence:
\begin{equation}\label{Hodge2}
0\longrightarrow \omega^{\underline{k}} \longrightarrow\cH^{\underline{k}}_m \stackrel{\epsilon}{\longrightarrow}\omega^{\underline{k}-m\tau_0}\otimes\omega_-^{-1}\otimes{\rm Sym}^{m-1}(\cH_{\tau_0})\stackrel{\eqref{isopol}}{\simeq}\cH^{\underline{k}- 2\tau_0}_{m-1}  \longrightarrow 0.
\end{equation}
By Griffiths transversality the Gauss-Manin connection (see \cite{ding17}) induces a connection:
$$
\bigtriangledown_{\underline{k},m}: \cH^{\underline{k}}_m \longrightarrow\omega^{\underline{k}-(m+1)\tau_0}\otimes{\rm Sym}^{m+1}(\cH_{\tau_0})\otimes\Omega^1_{X}
\stackrel{KS}{\longrightarrow}\cH^{\underline{k}+ 2\tau_0}_{m+ 1}
$$
Moreover, for each $j$ we put $\bigtriangledown_{\underline{k}}^j:=\bigtriangledown_{\underline{k}+2(j-1)\tau_0,j-1}\circ \cdots\circ\bigtriangledown_{\underline{k},0}$.

As in \S\ref{KatzModForm}, a global section $f$ of $\cH^{\underline{k}}$ is given by a rule that for each tuple $(A,\iota,\theta,\alpha)$ and $w^0= \{(f_\tau,e_\tau)\}_{\tau\neq \tau_0}$ a basis of $\bigoplus_{\tau\neq\tau_0}\omega_\tau$ assigns a linear form:
$$ f(A,\iota,\theta,\alpha,w^0):\bigotimes_{\tau\neq\tau_0}\Sym^{k_\tau}\left(R^2\right)\longrightarrow {\rm Sym}^{k_{\tau_0}}(\cH_{\tau_0}),$$
such that $f(A,\iota,\theta,\alpha):=f(A,\iota,\theta,\alpha,w^0)\left(\prod_{\tau\neq \tau_0}\left| \begin{smallmatrix}f_{\tau}&e_\tau\\x_\tau&y_\tau\end{smallmatrix}\right|^{k_\tau}\right)$. Thus, if $f$ is in fact a section of $\omega^{\underline{k}}\subseteq\cH^{\underline{k}}$ then for each $P\in \bigotimes_{\tau\neq\tau_0}\Sym^{k_\tau}\left(R^2\right)$ we have $f(A,\iota,\theta,\alpha,w^0)(P)=f(A,\iota,\theta,\alpha,w)(P)f_{0}^{\otimes k_{\tau_0}}$ if $w=(f_{0},w^0)$

Now let $\underline{k}_1 \in \Xi_{\nu_1}$, $\underline{k}_2\in \Xi_{\nu_2}$ and $\underline{k}_3\in \Xi_{\nu_1+\nu_2}$ be unbalanced at $\tau_0$ with dominant weight $\underline{k}_3$. Recall the notations $\underline{m}= (m_\tau)_{\tau \in \Sigma_F}:= ((k_{1,\tau}+k_{2,\tau}+k_{3,\tau})/2)_{\tau \in \Sigma_F}$, $m_{3, \tau_0}= (k_{3, \tau_0}-m_{\tau_0})$ and for $i= 1, 2, 3$ we denote $\underline{m}_i^{\tau_0}= (m_{i,\tau})_{\tau \neq \tau_0}= (m_\tau- k_{i, \tau})_{\tau \ne \tau_0}$.

Recall from \eqref{triple} we provided a trilinear product:
\[
t:H^0(G(\Q),\cA(\underline{k}_1,\nu_1))\times H^0(G(\Q),\cA(\underline{k}_2,\nu_2))\longrightarrow H^0(G(\Q),\cA(\underline{k}_3,\nu_1+\nu_2)). 
\]
The following result provides a geometric interpretation of this product in terms of the isomorphism $\iota_{\underline{k},\nu}$ of Proposition \ref{compautshe}. 
	
\begin{thm}\label{tripleonsheaves} Let $\cC\in \mathrm{Pic}(\cO_F)$. There exists a well defined morphism $t_\cC:\omega^{\underline{k}_1}\times\omega^{\underline{k}_2}\longrightarrow\omega^{\underline{k}_3}$ such that if $f_i \in H^0(X^{\cC}, \omega^{\underline{k}_i})$ for $i= 1, 2$ and $(A,\iota,\theta,\alpha)\in X$ we have: 
$$t_\cC(f_1,f_2)(A,\iota,\theta,\alpha)=
\sum_{j=0}^{m_{3,\tau_0}}c_j\bigtriangledown_{\underline{k}_1}^j(f_1)\bigtriangledown_{\underline{k}_2}^{m_{3, \tau_0}-j}(f_2)(A,\iota,\theta,\alpha,w^0)(\Delta^{\tau_0})$$ 
where $\Delta^{\tau_0}(x_1, y_1, x_2, y_2)= \prod_{\tau\neq \tau_0}|\begin{smallmatrix} f_\tau&e_\tau\\x_{2, \tau}&y_{2, \tau}\end{smallmatrix}|^{m_{1, \tau}}|\begin{smallmatrix}f_\tau&e_\tau\\x_{1, \tau}&y_{1, \tau}\end{smallmatrix}|^{m_{2, \tau}}|\begin{smallmatrix}x_{1, \tau}&y_{1, \tau}\\ x_{2, \tau}&y_{2, \tau}\end{smallmatrix}|^{m_{3, \tau}}$ and $c_j=(-1)^j\binom{m_{3, \tau_0}}{j}\binom{m_{\tau_0}-2}{k_{1, \tau_0}+j-1}$. Moreover, if $\phi_i\in H^0(G(\Q),\cA(\underline{k}_i,\nu_i))^{K_1^B(\cN)}$ for $i= 1, 2$ then we have: 
	\[
	t_\cC(\iota_{\underline{k}_1,\nu_1}(\phi_1)_\cC,\iota_{\underline{k}_2,\nu_2}(\phi_2)_\cC)=\left(\frac{1}{2i}\right)^{m_{3, \tau_0}}\iota_{\underline{k}_3,\nu_1+\nu_2}(t(\phi_1,\phi_2))_\cC.
	\]
\end{thm}
\begin{proof} By construction we have that ${\rm Im}(t_\cC)\subseteq \omega^{\underline{k}_3- m_{3, \tau_0}\tau_0}\otimes\Sym^{m_{3, \tau_0}}(\cH_{\tau_0})$ then to prove that $t_{\cC}$ is well defined we need to check that in fact ${\rm Im}(t_\cC)\subseteq \omega^{\underline{k}_3}$. 
As from \eqref{Hodge2} we have $\omega^{\underline{k}_3}=\ker\left(\omega^{\underline{k}_3- m_{3, \tau_0}\tau_0}\otimes\Sym^{m_{3, \tau_0}}(\cH_{\tau_0})\stackrel{\epsilon}{\longrightarrow} \omega^{\underline{k}_3- (m_{3, \tau_0}+1)\tau_0}\otimes\Sym^{m_{3, \tau_0}-1}(\cH_{\tau_0})\right)$ then we need to prove that $\epsilon(t_\cC(f_1, f_2))= 0$.
    
    Firstly we will prove that for each $\underline{k}$, $f \in H^0(X^{\cC}, \omega^{\underline{k}})$ and  $j \in \N$ we have:
\begin{equation}\label{eqepsiGM}
\epsilon\bigtriangledown_{\underline{k}}^j(f)=j(k_{\tau_0}+j-1)\bigtriangledown_{\underline{k}}^{j-1}(f).
\end{equation}
Hence we will have:
\begin{eqnarray*}
&&\epsilon( t_\cC(f_1, f_2))=\sum_{j=0}^{m_{3, \tau_0}}c_j\left(\epsilon\bigtriangledown_{\underline{k}_1}^j(f_1)\bigtriangledown_{\underline{k}_2}^{m_{3, \tau_0}-j}(f_2)+\bigtriangledown_{\underline{k}_1}^j(f_1)\epsilon\bigtriangledown_{\underline{k}_2}^{m_{3, \tau_0}-j}(f_2)\right)\\
&&=\sum_{n=0}^{m_{3, \tau_0}-1}\mbox{\small$(c_{n+1}(n+1)(k_{1,\tau_0}+n)+ c_n(m_{3, \tau_0}-n)(k_{2,\tau_0}+ m_{3, \tau_0}-n-1))$}\bigtriangledown_{\underline{k}_1}^{n}(f_1)\bigtriangledown_{\underline{k}_2}^{m_{3, \tau_0}-n-1}(f_2)= 0,
\end{eqnarray*}
and the first claim will follow.

To prove \eqref{eqepsiGM} it is enough to work locally, thus let $U=\Spec(R)$ be an open of $X^{\cC}$ such that trivializes the sheaves $\cH_{\tau_0}$ and $\omega$.  Let $u_1,u_2\in\cH_{\tau_0}$ be a basis such that $u_1\in\omega_{\tau_0}$ is a basis and $\bar\varphi_{\tau_0}(u_1,u_2)=1$, and let $\{f_\tau,e_\tau\}$ be an horizontal basis of $\omega_\tau$ for $\tau\neq {\tau_0}$ (remark that it is possible as $\bigtriangledown(\omega_\tau)\subseteq\omega_\tau\otimes\Omega^1_{U}$). 
Let $D\in {\rm Der}(R)$ be the dual of $u_1\otimes u_1\in\omega_{\tau_0}^{\otimes 2}\stackrel{KS}{\simeq} \Omega^1_{U}$. Since the Kodaira-Spencer isomorphism is given by the composition $\omega_{\tau_0}\hookrightarrow\cH_{\tau_0}\stackrel{\bigtriangledown}{\longrightarrow}\cH_{\tau_0}\otimes\Omega^{1}_U\stackrel{(\varphi_{\tau_0}\circ\epsilon)\otimes{\rm id}}{\longrightarrow}\omega_{\tau_0}^{-1}\otimes\Omega_U^1$, we have that $\bigtriangledown(D)(u_1)=(u_2+au_1)$ for some $a\in R$. 
Moreover, from	$0=\bigtriangledown(D)(\bar\varphi_{\tau_0}(u_1,u_2))=\bar\varphi_{\tau_0}(\bigtriangledown(D)u_1,u_2)+\bar\varphi_{\tau_0}(u_1,\bigtriangledown(D)u_2)$ we obtain $\bigtriangledown(D)(u_2)=(c u_1-au_2)$, for some $c\in R$. By changing $u_2$ by $u_2+au_1$, we can suppose that $a=0$, thus:
		\begin{equation}\label{G-Mbasis}
		\bigtriangledown(u_1)=u_2 u_1^2 ;\qquad \bigtriangledown(u_2) =c u_1^3,
		\end{equation}
		For any $f=\left(\sum_{j=0}^{r}b_j u_1^{k_{\tau_0}-j} u_2^j\right)M\in \omega^{\underline{k}-r\tau_0}\otimes{\rm Sym}^r(\cH_{\tau_0})$, where $M\in\bigotimes_{\tau\neq {\tau_0}}\omega_\tau^{k_\tau}$ is a monomial in $w_\tau^j$, write $f(X):=\sum_{j=0}^r b_j X^j\in R[X]$.
		Hence 
		\[
		\bigtriangledown_{\underline{k},r}f=\left(\sum_{j=0}^{r}(Db_j) u_1^{k_{\tau_0}-j+2} u_2^j+(k_{\tau_0}-j)b_j u_1^{k_{\tau_0}-j+1} u_2^{j+1}+j cb_j u_1^{k_{\tau_0}-j+3}u_2^{j-1}\right)M
		\]
		corresponds to $\bigtriangledown_{\underline{k},r}\left(f\right)(X)=Df(X)-(X^2-c)f'(X)+k_{\tau_0}Xf(X)$. Since $\epsilon$ is given by derivation $f(X)\mapsto f'(X)$, we compute that $(\epsilon\circ\bigtriangledown_{\underline{k},r}-\bigtriangledown_{\underline{k}-2,r-1}\circ\epsilon)f(X)=k_{\tau_0}f(X)$.  From this fact and a simple induction we deduce equality \eqref{eqepsiGM}.


		Now we are going to prove the second statement of the theorem.  Let $\rho:\dH\rightarrow \Gamma_{1,1}^t(\cN)\backslash\dH$ be the projection to the connected component of $X_{\C}$ corresponding to $t\in\Pic(E/F,\cN)$.
		Then, by the Riemann-Hilbert correspondence, we have:
		\[
		\rho^\ast\cR^1\pi_\ast\Omega^\bullet_{ A/X_{\C}}=\rho^\ast\cR^1\pi_\ast\Z\otimes_{\C}\cO_\dH=\cO_\dH\otimes_{\C} {\rm Hom}(D,\C),
		\]
		here $\cO_\dH$ is the sheaf of holomorphic functions of $\dH$ and each element in ${\rm Hom}(D,\C)$ provides a horizontal section by inducing the corresponding linear form on $H_1(A,\Z)=D\cap\hat\cO_\cM b_t^{-1}\subset D$. From this description we obtain:
\[
\rho^\ast\omega_{\tau_0}=\cO_\dH dx_0 \ \subset \ \rho^\ast\cH_{\tau_0}=\cO_\dH\alpha+\cO_\dH\beta
\] 
\[          
 \rho^\ast\cH_{\tau}=\rho^\ast\omega_\tau=\cO_\dH dx_\tau^1\oplus\cO_\dH dx_\tau^2
\]

for $\tau\neq \tau_0$. In the formulae above we have $dx_0=\alpha+z\beta$,  $\alpha(\gamma\otimes e)=\tilde\tau_0(e)\cdot d$ and $\beta(\gamma\otimes e)=\tilde\tau_0(e)\cdot c$ for  $e\in E$ and $\gamma\in B$ such that $\tau_0(\gamma)=(\begin{smallmatrix}a&b\\c&d\end{smallmatrix})$. Moreover if $\gamma\in D$, we have $dx_\tau^1(\gamma)=c_\tau$ and $dx_\tau^2(\gamma)=d_\tau$, where $\tilde\tau(\gamma)=(\begin{smallmatrix}a_\tau&b_\tau\\c_\tau&d_\tau\end{smallmatrix})$. Remark that $\alpha$, $\beta$, $dx_\tau^1$ and $dx_\tau^2$ are horizontal, in particular we obtain $\bigtriangledown( dx_0)=\beta\otimes dz\in\rho^\ast\cH_{\tau_0}\otimes\Omega^1_{\dH}$.

  Observe that $({\rm Lie}(A))_{\tau_0}^+$ is generated by the expressions $s(\gamma):= \frac{1}{2}\tau_0(\gamma)\otimes 1+\frac{1}{2\sqrt{\lambda}}\tau_0(\gamma)\otimes\sqrt{\lambda}$ for $\gamma \in B$. Since $\alpha(s(\gamma))=d$ and $\beta(s(\gamma))=c$ if $\tau_0(\gamma)=(\begin{smallmatrix}a&b\\c&d\end{smallmatrix})$ we obtain $\bar\varphi_{\tau_0}(\alpha,\beta)=\Theta(s(\begin{smallmatrix}0&0\\0&1\end{smallmatrix}),\delta^{-1}(\begin{smallmatrix}0&1\\1&0\end{smallmatrix})\overline{s(\begin{smallmatrix}0&0\\1&0\end{smallmatrix})})=-1$
where $\overline{s(\gamma)}=\frac{1}{2}\tau_0(\gamma)\otimes 1-\frac{1}{2\sqrt{\lambda}}\tau_0(\gamma)\otimes\sqrt{\lambda}$. We deduce from \eqref{isopol}: 
		\[
		KS(dx_0^2)=KS(dx_0\otimes dx_0)=\epsilon\bigtriangledown (dx_0)(\varphi_{\tau_0}dx_0)=\epsilon(\beta)(\varphi_{\tau_0}dx_0)dz=\bar\varphi_{\tau_0}(\beta,dx_0)dz=dz.
		\]
 Now let $f\in M_{\underline k}(\Gamma_{1,1}^{t}(\cN),\C)^\Delta$ and denote by $s_f$
 the section attached to $f$ and given by the formula \eqref{e:sections vs modular forms}, then we have:
\begin{equation}\label{e:claim in theorem}
		\bigtriangledown_{\underline{k}}^js_f(z)=\left(\frac{1}{2i}\right)^j\frac{\phi\left(R^j f^\nu_{k_{\tau_0}}\otimes \left(\prod_{\tau\neq \tau_0}|\begin{smallmatrix}dx_\tau^1&dx_\tau^2\\X_\tau&Y_\tau\end{smallmatrix}|^{k_\tau}\right)\right)}{f_{k_{\tau_0}+2j}^\nu(g_{\tau_0})}dx_0^{k_{\tau_0}+2j}+d\bar x_0 w,
\end{equation}
where $d\bar x_0=\alpha+\bar z\beta$ and some $w\in C^\infty(\dH)\otimes{\rm Sym}^{\underline{k}+2j-1}(\cH^1)$. Indeed, since $\beta=(z-\bar z)^{-1}(dx_0-d\bar x_0)=(2iy)^{-1}(dx_0-d\bar x_0)$ we obtain:
		\begin{eqnarray*}
			\bigtriangledown s_f(z)&=&\frac{\partial}{\partial z} f(z)\left(\prod_{\tau\neq \tau_0}|\begin{smallmatrix}dx_\tau^1&dx_\tau^2\\X_\tau&Y_\tau\end{smallmatrix}|^{k_\tau}\right)dx_0^{k_1+2}+ k_{\tau_0}f(z)\left(\prod_{\tau\neq \tau_0}|\begin{smallmatrix}dx_\tau^1&dx_\tau^2\\X_\tau&Y_\tau\end{smallmatrix}|^{k_\tau}\right)dx_0^{k_{\tau_0}+1}\beta\\
			&=&\left(\left(\frac{\partial}{\partial z}+\frac{k_{\tau_0}}{2iy}\right)f(z)\left(\prod_{\tau\neq \tau_0}|\begin{smallmatrix}dx_\tau^1&dx_\tau^2\\X_\tau&Y_\tau\end{smallmatrix}|^{k_\tau}\right)dx_0^{k_{\tau_0}+2}-\frac{k_{\tau_0}f(z)\left(\prod_{\tau\neq \tau_0}|\begin{smallmatrix}dx_\tau^1&dx_\tau^2\\X_\tau&Y_\tau\end{smallmatrix}|^{k_\tau}\right)}{2iy}dx_0^{k_{\tau_0}+1}d\bar x_0\right).
		\end{eqnarray*}
The claim follows from induction and the equality $\frac{\phi(Rf^\nu_{k_{\tau_0}}\otimes P)}{2i f_{k_{\tau_0}+2}^\nu}=\left(\frac{\partial}{\partial z}+\frac{k_{\tau_0}}{2iy}\right)\frac{\phi(f^\nu_{k_{\tau_0}}\otimes P)}{f_{k_{\tau_0}}^\nu}$.

	Finally since $dx_0$ and $d\bar x_0$ are linearly independent, comparing formulas \eqref{e:claim in theorem} and \eqref{tripleD}, we obtain the second claim of the theorem.
\end{proof}
	
\part{$p$-adic families} 
We construct $p$-adic families of modular forms on the unitary Shimura curves $X$. From these families we construct families of autmorphic forms of $B$. Moreover, we construct local pieces of an adic eigenvariety. 

\section{Integral models and canonical groups}\label{s:integral models and canonical groups} In this section we introduce the technical tools necessary to realize the $p$-adic variation of the modular sheaves introduced in \ref{ss:modular sheaves}.

\subsection{Integral models and Hasse invariants}\label{ss:integral models and divisible groups} Let $\cN$ be an integral ideal of $F$ prime to $p$ and $\mathrm{disc}(B)$ and denote by $X$ the unitary shimura curve of level $K_1'\left(\mathfrak{n},\prod_{\dP\neq\dP_0}\dP\right)$  introduced in \S\ref{ss:unitary Shimura curves}.  By \cite[\S 5.3]{carayol86} $X$ admits a canonical model over $\cO_0$, representing the analogue moduli problem described in \S\ref{ss:unitary Shimura curves} but exchanging an $E$-algebra by an $\cO_0$-algebra $R$. Namely, it classifies quadruples $(A,\iota,\theta,\alpha^{\dP_0})$ over $R$, where $\alpha^{\dP_0}$ is a class of $\cO_D$-linear symplectic similitudes outside $\dP_0$. We denote this integral model by $X_{\mathrm{int}}$, which has good reduction (see \cite[\S 5.4]{carayol86}). Let $\pi: {\bf A} \rightarrow X_{\mathrm{int}}$ be the universal abelian variety. Since we have added $\Gamma_0(\dP)$-structure for all $\dP\neq\dP_0$, ${\bf A}$ is endowed with a subgroup $C_\dP\subset {\bf A}[\dP]^{-,1}$ isomorphic to $\cO_\dP/\dP$ by Remark \ref{rmkonpts}. 

  Let $\dX$ be denote the formal scheme over $\Spf(\cO_0)$ obtained as the completion of $X_{\mathrm{int}}$ along its special fiber which is denoted by  $\bar{X}_{\mathrm{int}}$ .

The $p$-divisible group ${\bf A}[p^\infty]$ over $X_{\mathrm{int}}$ is decomposed as:
\[
{\bf A}[p^\infty]={\bf A}[\mathfrak{p}_0^\infty]^+\oplus \left[ \bigoplus_{\mathfrak{p} \neq \mathfrak{p}_0}{\bf A}[\mathfrak{p}^\infty]^+\right] \oplus {\bf A}[\mathfrak{p}_0^\infty]^-\oplus \left[\bigoplus_{\mathfrak{p} \neq \mathfrak{p}_0}{\bf A}[\mathfrak{p}^\infty]^-\right] ,
\] 
We are interested in the $p$-divisible groups $\mathcal{G}_0:= {\bf A}[\mathfrak{p}_0^\infty]^{-, 1}$ and $\mathcal{G}_{\mathfrak{p}}:= {\bf A}[\mathfrak{p}^\infty]^{-, 1}$ if $\mathfrak{p} \neq \mathfrak{p}_0$, which are defined over $X_{\mathrm{int}}$ and endowed with actions of $\cO_0$ and $\cO_{\mathfrak{p}}$ respectively. The sheaves of invariant differentials of the corresponding Cartier dual $p$-divisible groups are denoted by $\omega_{0}:= \omega_{\cG_{0}^D}$ and  $\omega_{\mathfrak{p}}:= \omega_{\cG_{\mathfrak{p}}^D} $ if $\mathfrak{p} \neq \mathfrak{p}_0$. By Lemma \ref{choiceTheta}, the universal polarization $\theta$ is an isomorphism over $\cO_0$. Hence this notations are justified because $\theta$ induce the following identifications:
\begin{equation}\label{eq:omega decomposed}
\omega=  \left(\pi_\ast\Omega^1_{{\bf A}/X_{\mathrm{int}}}\right)^{+,2}\stackrel{\theta^\ast}{\simeq} \left(\pi_\ast\Omega^1_{{\bf A}^\vee/X_{\mathrm{int}}}\right)^{-,1}=\omega_0\oplus \bigoplus_{\mathfrak{p} \neq \mathfrak{p}_0}\omega_{\mathfrak{p}} 
 \end{equation}

The universal polarization provides a pairing (see Remark \ref{rmkonpts})
\[
\Theta:{\bf A}[p^\infty]^+\times{\bf A}[p^\infty]^-\longrightarrow \G_m[p^\infty].
\]
Since $p$ splits in $\Q(\sqrt{\lambda})$, we have that $\cO_E\otimes\Z_p\simeq \cO^2$, and the isomorphism that switches components induces an isomorphism $c:{\bf A}[p^\infty]^-\rightarrow{\bf A}[p^\infty]^+$. Analogously as in \eqref{isopol}, we obtain an isomorphism of $p$-divisible groups
\begin{equation}\label{isopol2}
\theta:\cG_0\stackrel{\simeq}{\longrightarrow}\cG_0^D;\qquad  \theta(P)(Q):=\Theta\left(P,\delta^{-1}\left(\begin{smallmatrix}0&1\\1&0\end{smallmatrix}\right)c(Q)\right).
\end{equation}
Hence we have an isomorphism of sheaves of invariant differentials $\omega_0\simeq\omega_{\cG_0}$ compatible with \eqref{isopol}. 

We denote by $\hat\cO$ a ring containing all the $p$-adic embeddings of $\cO_F\hookrightarrow\cO_{\C_p}$, hence if we extend our base ring $\cO_0$ to $\hat\cO$ then for each $\mathfrak{p} \neq \mathfrak{p}_0$ we have a decomposition $\omega_{\mathfrak{p}}= \oplus_{\tau \in \Sigma_\mathfrak{p}}\omega_{\mathfrak{p}, \tau}$, moreover each $\omega_{\mathfrak{p}, \tau}$ has rank $2$. The alternating pairing $\Theta$ provides, as in the complex setting, an isomorphism $\varphi_\tau:\bigwedge^2\omega_{\dP,\tau}\stackrel{\simeq}{\rightarrow}\cO_{X_{\rm int}}$.

There is a dichotomy in $X_{\mathrm{int}}$ which says that any point in the generic fiber $\bar{X}_{\mathrm{int}}$ is ordinary or  supersingular (with respect to $\mathcal{G}_0$), and there are finitely many supersingular points in $\bar{X}_{\mathrm{int}}$. From \cite[Proposition 6.1]{kassaei04} there exists $\mathrm{Ha}\in H^0\left(\bar{X}_{\mathrm{int}},\omega_{0}^{p-1}\right)$ that vanishes exactly at supersingular geometric points and these zeroes are simple. This is called the \emph{Hasse invariant} and it is characterized as follows: For each open $\Spec(R)\subset \bar{X}_{\mathrm{int}}$ fix $w$ a generator of $\omega_{{0}}\mid_{\Spec(R)}$ and $x$ a coordinate of $\cG_0$ over $R$ such that $w=(1+a_1x+a_2x^2\cdots)dx$, then  $[p](x)=ax^p+\cdots$ for some $a$ and $\mathrm{Ha}\mid_{\Spec(R)}:=aw^{p-1}$.

We denote by $\overline{\rm Hdg}$ the locally principal ideal of $\cO_{\bar{X}_{\mathrm{int}}}$ described as follows: for each $U=\Spec(R)\subset\bar{X}_{\mathrm{int}}$  if $\omega_{0}\mid_U=Rw$ and  $\mathrm{Ha}\mid_U=Hw^{\otimes(p-1)}$ then $\overline{\rm Hdg}\mid_U=HR\subseteq R$. Let $\rm Hdg$ the inverse image of $\overline{\rm Hdg}$ in $\cO_{\dX}$, which is also a locally principal ideal. Note that $\mathrm{Ha}^{p^n}$ extends canonically to a section of $H^0(\dX,\omega_0^{p^n(p-1)}\otimes\Z/p^{n+1}\Z)$, indeed, for any two extensions $\mathrm{Ha}_1$ and $\mathrm{Ha}_2$ of $\mathrm{Ha}$ we have $\mathrm{Ha}_1^{p^n}=\mathrm{Ha}_2^{p^n}$ modulo $p^{n+1}$ by the binomial formula. 

\begin{rmk} From \cite[Prop. 3.4]{brasca13} we obtain the existence of a $p-1$-root of the principal ideal $\mathrm{Hdg}$. This ideal is denoted $\mathrm{Hdg}^{1/(p-1)}$.
\end{rmk}

Now we introduce some formal schemes in order to vary $p$-adically the modular sheaves and to produce $p$-adic families. 
For each integer $r\geq 1$ we denote by $\dX_{r}$ the formal scheme over $\dX$ which represents the functor (denoted by the same symbol $\dX_{r}$) that classifies for each $p$-adically complete  $\hat\cO$-algebra $R$:
$$\dX_{r}(R)= \left\lbrace [(h, \eta)] \ \ | \ \  h\in\dX(R),\; \eta\in H^0(\Spf(R),h^\ast(\omega_{\cG}^{(1-p)p^{r+1}})),\; \eta\cdot{\rm Ha}^{p^{r+1}}=p\mod p^2 \right\rbrace ,$$
here the brackets means the class of the equivalence given by $(h,\eta)\equiv (h',\eta')$ if $h=h'$ and $\eta=\eta'(1+pu)$ for some $u\in R$. The formal scheme $\dX_r$ turns out to be the $p$-adic completion of the partial blow-up of $\dX$ at the zero locus of ${\rm Hdg}^{p^{r+1}}$ and $p$ (see \cite[Definition 3.1]{AIP-halo}).

\subsection{Canonical subgroups}\label{ss:canonical subgroups}   The theory of the canonical subgroup originally developed in the context of $p$-adic topology was generalized to the adic setting in \cite{AIP-halo}:
\begin{prop}\cite[Corollaire A.2]{AIP-halo}\label{p:canonical subgroups} There exists a canonical subgroup $C_n$ of $\cG_0[\mathfrak{p}_0^n]$ for $n\leq r$ over $\dX_{r}$. This is unique and satisfy the compatibility $C_n[\mathfrak{p}_0^{n-1}]=C_{n-1}$. Moreover, if we denote $D_n:= \cG_0[\mathfrak{p}_0^n]/C_n$ then $\omega_{D_n}\simeq\omega_{\cG_0[\dP^n]}/\mathrm{Hdg}^{\frac{p^n-1}{p-1}}$.
\end{prop}
By \cite[Proposition A.3]{AIP-halo}, 
the cokernel of the map $\mathrm{dlog}_{\cG_0^D[\dP^n_0]}: \cG_0[\dP_0^n] \rightarrow \omega_{\cG_0^D[\dP_0^n]}=\omega_{0}/p^n$ is killed by $\mathrm{Hdg}^{1/(p-1)}$. If we write $\Omega_0\subseteq\omega_0$ for the subsheaf generated by the lifts of the image of the Hodge-Tate map, we obtain a morphism
\begin{equation}\label{dlog0}
{\rm dlog}_0:D_n(\dX_{r})\longrightarrow \Omega_0\otimes_{\cO_{\dX_{r}}}(\cO_{\dX_{r}}/\cI_n)\subset \omega_0/p^n\mathrm{Hdg}^{-\frac{p^n-1}{p-1}},
\end{equation}
where $\cI_n:=p^n\mathrm{Hdg}^{-\frac{p^n}{p-1}}$.

In order to carry out the $p$-adic interpolation we work on covers of $\dX_{r}$. First notice that, by the moduli interpretation, the $p$-divisible group $\prod_{\dP\neq\dP_0}\cG_\dP\rightarrow\dX_{r}$ is \'etale isomorphic to $\prod_{\dP\neq\dP_0}(F_\dP/\cO_\dP)^2$. Assume that $r\geq n$. Firstly we add $\dP^{n}$-level corresponding to the primes $\mathfrak{p} \mid p$ such that $\mathfrak{p} \neq \mathfrak{p}_0$: We denote by $\dX_{r,n}\rightarrow\dX_{r}$ the formal scheme obtained by adding to the moduli interpretation a point of order $\mathfrak{p}^{n}$ for each $\mathfrak{p} \neq \mathfrak{p}_0$ whose multiples generate $\cG_\dP[\dP]/C_\dP$ (see Remark \ref{rmkonpts}). It is clear that the extension $\dX_{r,n}\rightarrow\dX_{r}$ is \'etale and it Galois group contains $\prod_{\dP\neq\dP_0}(\cO_\dP/p^{n}\cO_\dP)^\times$ as a subgroup since $F$ is unramified at $p$ by Hypothesis \ref{hypothesis 1}. Moreover, $\dX_{r,n}$ has also good reduction (see \cite[\S 5.4]{carayol86}). Now we trivialize the subgroup $D_n$:
Let  $\cX_{r,n}$ be the adic generic fiber of $\dX_{r,n}$. By \cite[Corollaire A.2]{AIP-halo}, the group scheme $D_{n}\rightarrow\cX_{r,n}$ is also \'etale isomorphic to $p^{{-n}}\cO_0/\cO_0$. We denote by $\mathcal{IG}_{r,n}$ the adic space over $\cX_{r,n}$ of the trivializations of $D_{n}$. Then  the map $\mathcal{IG}_{r,n} \rightarrow\cX_{r,n}$ is a finite \'etale and with a Galois group $(\cO_0/p^{n}\cO_0)^\times$. We denote by $\mathfrak{IG}_{r,n}$ the normalization $\mathcal{IG}_{ r,n}$ in $\dX_{ r,n}$ which is finite over $ \dX_{ r,n}$ and it is also endowed with an action of $(\cO_{0}/p^{n}\cO_{0})^\times$. These constructions are captured by the following tower of formal schemes: 
$$\mathfrak{IG}_{r,n} \longrightarrow\dX_{r,n}\longrightarrow\dX_{r},$$
endowed with a natural action of $(\cO/p^{n}\cO)^\times\simeq\prod_\dP(\cO_\dP/p^{n}\cO_\dP)^\times$.

\section{Overconvergent modular forms for $G'$}\label{s:over modular sheaves}   Following the approach introduced in \cite{AI17} we deform the modular sheaves of $G'$ which  allow us to define overconvergent modular forms for $G'$ and families of them. We also construct other overconvergent sheaves which will be usefull to construct triple product $p$-adic $L$-functions.

\subsection{Formal vector bundles}\label{ss:formal vector bundles}
	In this subsection we slightly modify the construction performed in \cite[\S 2]{AI17} which we briefly recall first. Let $S$ be a formal scheme, $\cI$ its (invertible) ideal of definition and $\mathcal{E}$ a locally free $\cO_S$-module of rank $n$. We write $\bar S$ the scheme with structural sheaf $\mathcal{O}_S/\cI$ and put $\overline{\mathcal{E}}$ the corresponding $\cO_{\overline{S}}$-module. We fix marked sections $s_1,\cdots, s_m$ of $\overline{\mathcal{E}}$, namely, the sections $s_1, \cdots , s_m$ define a direct sum decomposition $\overline{\mathcal{E}} = \cO_{\bar S}^m\oplus Q$, where $Q$ is a
locally free $\cO_{\bar S}$-module of rank $n-m$.

Let $S-\mathrm{Sch}$ be the category of the  formal $S$-schemes. There exists a formal scheme $\V(\mathcal{E})$ over $S$ called the \emph{formal vector bundle} attached to $\mathcal{E}$ which represents the functor, denoted by the same symbol, $S-\mathrm{Sch} \rightarrow \mathrm{Sets}$, and defined by $\V(\mathcal{E})(t:T\rightarrow S):=  H^0(T,t^\ast(\mathcal{E})^\vee)=  {\rm Hom}_{\cO_T}(t^\ast(\mathcal{E}),\cO_T)$.  Crucial in \cite{AI17} is the construction of the so called \emph{formal vector bundles with marked sections} which is the formal scheme $\V_0(\mathcal{E},s_1,\cdots,s_m)$ over $\V(\mathcal{E})$ that represents the functor $S-\mathrm{Sch} \rightarrow \mathrm{Sets}$ defined by 
   $$ \V_0(\mathcal{E}, s_1,\cdots,s_m)(t:T\rightarrow S)=\left\lbrace \rho\in H^0(T,t^\ast(\mathcal{E})^\vee) \ | \  \bar\rho(t^\ast(s_i))=1, \  i=1,\cdots,m \right\rbrace, $$
here $\bar\rho$ is the reduction of $\rho$ modulo $\cI$. The construction of $\V_0(\mathcal{E},s_1,\cdots,s_m)$ is as follows: the projection map $\overline{\mathcal{E}} \rightarrow Q$ defines a quotient map $\bigoplus_k\Sym^k(\overline{\mathcal{E}})\rightarrow\bigoplus_k\Sym^k(Q)$ whose kernel is the ideal $(s_1, \cdots , s_m)$ and, hence, defines a closed
subscheme $C \subset \V(\overline{\mathcal{E}})=\widehat{\Sym(\mathcal{E})}$ with corresponding ideal sheaf denoted by $\mathcal{J}$. Then  $\V_0(\mathcal{E},s_1,\cdots,s_m)$ is the $\cI$-adic completion
of the open formal subscheme of the blow up of $\V(\mathcal{E})$ with respect to the ideal $\mathcal{J}$, open defined
by the requirement that the ideal generated by the inverse image of $\mathcal{J}$ coincides with the ideal generated by the inverse image of $\cI$.

Given the fixed decomposition $\overline{\mathcal{E}}=Q\oplus \langle s_i\rangle_i$,
let us consider now the sub-functor $\hat\V_Q(\mathcal{E},s_1,\cdots,s_m)$ that associates to any formal $S$-scheme $t:T\rightarrow S$ the subset of sections $\rho\in \V_0(\mathcal{E},s_1,\cdots,s_m)(T)$ whose reduction $\bar\rho$ modulo $\cI$ also satisfies $\bar\rho(t^\ast(m))=0$ for every $m\in Q$.
\begin{lemma} The morphism $\hat\V_{Q}(\mathcal{E},s_1,\cdots,s_m)\rightarrow \V_0(\mathcal{E},s_1,\cdots,s_m)$ is represented by a formal subscheme.
\end{lemma}
\begin{proof}
Since we have the direct sum decomposition $\overline{\mathcal{E}} = \cO_{\bar S}^m\oplus Q$, we have the closed subscheme $\bigoplus_k\Sym^k(Q)\hookrightarrow\bigoplus_k\Sym^k(\overline{\mathcal{E}})$ and, hence, a closed
subscheme $V \subset \V_0(\overline{\mathcal{E}},s_1,\cdots,s_m)$. Let $\mathcal{J}$ be the corresponding ideal sheaf. Then we write
$\hat\V_{Q}(\mathcal{E},s_1,\cdots,s_m)$ for the closed formal subscheme given by the inverse image of $\mathcal{J}$ in $\V_0(\mathcal{E},s_1,\cdots,s_m)$.

Let $U=\Spf(R)$ be a formal affine open such that $\cI$ is generated by $\alpha\in R$ and $\mathcal{E}|_U$ is free of rank $n$ with basis $e_1,\cdots,e_n$ such that $e_i\equiv s_i$ modulo $\cI$, for $i=1,\cdots,m$, and $Q=\langle e_{m+1},\cdots,e_n\rangle$ modulo $\cI$. Thus $\V(\mathcal{E})|_U=\Spf(R\langle X_1,\cdots,X_n\rangle)$ and $\V_0(\mathcal{E},s_1,\cdots,s_m)|_U=\Spf(R\langle Z_1,\cdots,Z_m,X_{m+1},\cdots,X_n\rangle)$ with the corresponding morphism $\V_0(\mathcal{E},s_1,\cdots,s_m)\rightarrow \V(\mathcal{E})$ given by $X_i\mapsto 1+\alpha Z_i$, for $i=1,\cdots,m$ (see \cite[\S 2]{AI17}). Since the inverse image of $\mathcal{J}$ corresponds to $(\alpha,X_{m+1},\cdots,X_{n})$, we deduce that  $\hat\V_{Q}(\mathcal{E},s_1,\cdots,s_m)|_U=\Spf(R\langle Z_1,\cdots,Z_m,T_{m+1},\cdots T_{n}\rangle)$ with the corresponding morphism $\hat\V_{Q}(\mathcal{E},s_1,\cdots,s_m)\rightarrow \V_0(\mathcal{E},s_1,\cdots,s_m)$ given by $X_{j}\mapsto \alpha Z_{j}$, for $j=m+1,\cdots,n$. 

Given a formal scheme $t:T\rightarrow U$, a section $\rho\in \V_0(\mathcal{E},s_1,\cdots,s_m)(T)$ defined by the images $\rho^\ast(Z_i)=a_i$ and $\rho^\ast(X_j)=b_j$ provides the morphism $\rho\in{\rm Hom}_{\cO_T}(t^\ast(\mathcal{E}),\cO_T)$ satisfying $\rho(t^\ast(e_i))=1+\alpha a_i$, for $i=1,\cdots,m$, and $\rho(t^\ast(e_j))=b_j$, for $j=m+1,\cdots,n$. Thus, sections coming from $\hat\V_{Q}(\mathcal{E},s_1,\cdots,s_m)(T)$ will correspond to those $\rho$ also satisfying $b_{j}=\alpha c_{j}$, for $j=m+1,\cdots,n$ and some $c_j\in R$. Hence, to sections such that $\bar\rho(t^\ast(m))=0$ for every $m\in Q$.
\end{proof}

\begin{rmk}
Notice that this construction is also functorial with respect to $(\mathcal{E},Q,s_1,\cdots,s_m)$. Indeed, given a morphism $\varphi:\mathcal{E}'\rightarrow\mathcal{E}$ of locally free $\cO_S$-modules of finite rank and marked sections $s_1,\cdots,s_m\in \overline{\mathcal{E}}$, $s_1',\cdots,s_m'\in \overline{\mathcal{E}'}$ such that $\bar\varphi(s_i')=s_i$ and $\bar\varphi(Q')\subseteq Q$, we have the morphisms making the following diagram commutative
\[
\xymatrix{
\hat\V_{Q}(\mathcal{E},s_1,\cdots,s_m)\ar[r]\ar[d]&\V_{0}(\mathcal{E},s_1,\cdots,s_m)\ar[r]\ar[d]&\V(\mathcal{E})\ar[d]\\
\hat\V_{Q}(\mathcal{E}',s_1',\cdots,s_m')\ar[r]&\V_{0}(\mathcal{E}',s_1',\cdots,s_m')\ar[r]&\V(\mathcal{E}')
}
\]
\end{rmk}
\begin{rmk}\label{DepOnQuot}
In fact, $\bar\V_{Q}(\mathcal{E},s_1,\cdots,s_m)$ depends on $Q\subset \overline{\mathcal{E}}$ and the image of $s_i$ in $\overline{\mathcal{E}}/Q$. Indeed, given $m_i\in Q$ and $t:T\rightarrow S$, since $\bar\rho(t^\ast(m_i))=0$,
\[
\hat\V_{Q}(\mathcal{E},(s_i+m_i)_i)(T)=\left\{\rho\in H^0(T,t^\ast(\mathcal{E})^\vee);\quad\bar\rho(t^\ast(s_i+m_i))=1,\; \bar\rho(t^\ast(Q))=0\right\}=\hat\V_{Q}(\mathcal{E},(s_i)_i)(T).
\]
\end{rmk}

\subsection{Weight space for $G'$}\label{ss:weight space}   We fix a decomposition: 
$$\cO^\times\cong \cO^0\times H,$$
where $H$ is the torsion subgroup of $\cO^\times$ and $\cO^0\simeq 1+p\cO$ is a free $\Z_p$-module of rank $d$.
We put $\Lambda_F:= \Z_p[[\cO^\times]]$ and $\Lambda_F^0:= \Z_p[[\cO^0]]$. Then remark that the choice of a basis $\{e_1, ..., e_d\}$ of  $\cO^0$ furnishes an isomorphism $\Lambda_F^0\cong  \Z_p[[T_1, ....,T_d]]$ given by  $1+ T_i= e_i$ for $i= 1,..., d$. Moreover, for each $n\in \N$ we consider the algebras:
\[
\Lambda_{n}:= \Lambda_F\left\langle\frac{T_1^{p^n}}{p},\cdots,\frac{T_d^{p^n}}{p}\right\rangle\qquad\qquad \Lambda_{n}^0:=\Lambda_F^0\left\langle\frac{T_1^{p^n}}{p},\cdots,\frac{T_d^{p^n}}{p}\right\rangle
\]

The formal scheme $\dW=\Spf(\Lambda_F)$ is our formal \emph{weight space} for $G'$.  Thus for each complete $\Z_p$-algebra $R$ we have: 
$$\dW(R)={\rm Hom}_{\mathrm{cont}}(\cO^\times,R^\times).$$
We also consider the following formal schemes $\dW^0=\Spf(\Lambda_F^0)$, $\dW_{n}:=\Spf(\Lambda_{n})$ and $\dW^0_{n}:=\Spf(\Lambda_{n}^0)$. By construction we have $\dW=\bigcup_n\dW_{n}$ and $\dW^0=\bigcup_n\dW^0_{n}$. Moreover, we have the following explicit description:
\[
\dW_n^0(\C_p)=\{k\in{\rm Hom}_{\mathrm{cont}}(\cO^0,\C_p^\times): \;|k(e_i)-1|\leq p^{-p^{-n}},\;i=1,\cdots,d\}
\] 
We denote by ${\bf k}:\cO^\times\rightarrow\Lambda_F^\times$ the universal character of $\dW$, which decomposes as ${\bf k}= {\bf k}^0\otimes{\bf k}_f$ where: 
 \[
 {\bf k}_f:H\longrightarrow \Z_p[H]^\times\qquad\qquad {\bf k}^0:\cO^0\longrightarrow (\Lambda_F^0)^\times.
 \]
Let ${\bf k}_{n}^0: \cO^0 \rightarrow (\Lambda^0_n)^\times $ be  given by the composition of ${\bf k}^0$ with the inclusion $(\Lambda_F^0)^\times\subseteq (\Lambda^0_n)^\times$ and we put ${\bf k}_n:={\bf k}_n^0\otimes{\bf k}_f:\cO^\times\rightarrow\Lambda_n^\times$.

 \begin{lemma}\label{weightextends} Let $R$ be a $p$-adically complete $\Lambda^0_{n}$-algebra. Then ${\bf k}_{n}^0$ extends locally analytically to a character $\cO^0(1+p^{n}\lambda^{-1}\cO_{F}\otimes R)\rightarrow R^{\times}$, for any $\lambda\in R$ such that $\lambda^{p-1}\in p^{p-2}m_R$, where $m_R$ is the maximal of $R$. In particular, ${\bf k}_{n}^0$ is analytic on $1+p^{n}\cO$.
\end{lemma}	
\begin{proof} By definition we have ${\bf k}_{n}^0\left(\sum_i\gamma_ie_i\right)=\prod_{i=1}^{d}(T_i+1)^{\gamma_i}\in \Z_p[[T_1,\cdots,T_d]]\simeq\Lambda_F^0\subseteq\Lambda_{n}^0$.
	Formally, we have that 
	$(T_i+1)^{\gamma_i}=\exp(\gamma_i\log(T_i+1))$. For any $m=p^{n+s}m'$, $(m',p)=1$ we have:
	\[
	p^n\frac{T_i^m}{m}=\frac{T_i^{m-(s+1)p^n}}{m'}\left(\frac{T_i^{p^n}}{p}\right)^{s+1}p\in p\Lambda_{n}^0,
	\]
	and $m-(s+1)p^n=p^n(p^sm'-(s+1))\geq 0$. Hence $u_i:=p^n\log(T_i+1)\in p\Lambda_n^0$. We have that 
	\[
	{\bf k}_{n}^0\left(p^n\alpha\right)=\prod_{i=1}^{h_1}\exp(u_i\gamma_i),\qquad \alpha=\sum_i\gamma_i e_i.
	\]
	This implies that, for any adic 
	$\Lambda_{n}^0$-algebra $R$, the character ${\bf k}^0_{n}$ can be evaluated at the $\Z_p$-submodule $\cO^0\otimes_{\Z_p}\lambda^{-1} p^nR$, for any $\lambda$ with valuation $\nu(\lambda)<\frac{p-2}{p-1}$. Note that, by means of the exponential map we can identify $\cO^0\otimes_{\Z_p}\lambda^{-1} p^nR$ with $(1+\lambda^{-1} p^n\cO\otimes_{\Z_p}R)\subset (\cO\otimes_{\Z_p}R)^\times$. We conclude that ${\bf k}_{n}^0$ can be evaluated at
	$\cO^0\cdot(1+\lambda^{-1} p^{n}\cO\otimes_{\Z_p}R)\subset (\cO\otimes_{\Z_p}R)^\times$,
	and the result follows.
\end{proof}

 Recall that hypothesis \ref{hypothesis 1} imply a decomposition of rings $\cO=\cO_{\dP_0} \times\prod_{\mathfrak{p}\neq \mathfrak{p}_0}\cO_{\mathfrak{p}}= \Z_p\times\cO^{\tau_0}$ and put $\cO^{\tau_0,0}:=1+p\cO^{\tau_0}$. Analogously as above we introduce:
 \begin{eqnarray*}
 \begin{gathered}
 \Lambda_{\tau_0}:=\Z_p[[\Z_p^\times]]\qquad\qquad \Lambda_{\tau_0}^0:=\Z_p[[1+p\Z_p]]\simeq\Z_p[[T]]\\
\Lambda^{\tau_0}:=\Z_p[[\cO^{\tau_0\times}]]\qquad\qquad  \Lambda^{\tau_0,0}:=\Z_p[[\cO^{\tau_0,0}]]\simeq\Z_p[[T_2,\cdots,T_d]]\\
\Lambda_{\tau_0,n}:= \Lambda_{\tau_0}\left\langle\frac{T^{p^n}}{p}\right\rangle\qquad\qquad \Lambda_{\tau_0,n}^0:= \Lambda_{\tau_0}^0\left\langle\frac{T^{p^n}}{p}\right\rangle\\
 \Lambda^{\tau_0}_n:=\Lambda^{\tau_0}\left\langle\frac{T_2^{p^n}}{p},\cdots,\frac{T_d^{p^n}}{p}\right\rangle\qquad\qquad \Lambda^{\tau_0,0}_n:=\Lambda^{\tau_0,0}\left\langle\frac{T_2^{p^n}}{p},\cdots,\frac{T_d^{p^n}}{p}\right\rangle.
 \end{gathered}
\end{eqnarray*}
Thus, we have decompositions $\Lambda_n^0=\Lambda_{\tau_0,n}^0\hat\otimes\Lambda_n^{\tau_0,0}$ and $\Lambda_n=\Lambda_{\tau_0, n}\hat\otimes\Lambda_n^{\tau_0}$. We denote by ${\bf k}_{\tau_0, n}^{0}:(1+p\Z_p)\longrightarrow\Lambda_{\tau_0,n}^0$, ${\bf k}_{n}^{\tau_0,0}:\cO^{\tau_0,0}\longrightarrow\Lambda_n^{\tau_0,0}$, ${\bf k}_{\tau_0, n}:\Z_p^\times\rightarrow \Lambda_{\tau_0,n}$ and ${\bf k}_n^{\tau_0}:\cO^{\tau_0\times}\rightarrow \Lambda^{\tau_0}_n$ the universal characters. Then we have decompositions $ {\bf k}_{n}^0= {\bf k}_{\tau_0, n}^0\otimes {\bf k}_{n}^{\tau_0,0}$ and ${\bf k}_n={\bf k}_{\tau_0, n}\otimes {\bf k}_n^{\tau_0}$. Moreover,
 $${\bf k}_{\tau_0, n}\otimes {\bf k}_n^{\tau_0}= {\bf k}_{n}= {\bf k}_{n}^0\otimes{\bf k}_{f}={\bf k}_{\tau_0, n}^0\otimes {\bf k}_{n}^{\tau_0,0}\otimes{\bf k}_{f}.$$

\subsection{Overconvergent modular sheaves}\label{ss:overconvergent modular sheaves} 

We fix $L$  a finite extension of $\Q_p$ containing all the $p$-adic embedding of $F$ and we work over the ring of integers of $L$. Let $r\geq n$. As  in \S \ref{ss:canonical subgroups}, we denote by $\mathcal{I}_n:=p^{n} \mathrm{Hdg}^{-\frac{p^{n}}{p-1}}$ considered now as an ideal of $\cO_{\mathfrak{IG}_{ r,n}}$, which is our ideal in order to perform the construction of \S\ref{ss:formal vector bundles}. Then using notations from \S\ref{ss:formal vector bundles} we put $\overline{\mathfrak{IG}_{r,n}}$ for the corresponding reduction modulo $\cI_{n}$. From Equation \eqref{dlog0}, we have an isomorphism 
\begin{equation}\label{omega canonical vs G}
\mathrm{dlog}_0:D_{n}(\mathfrak{IG}_{r,n})\otimes_{\Z_p}(\cO_{\mathfrak{IG}_{r,n}}/\cI_n)\stackrel{\simeq}{\longrightarrow}\Omega_0\otimes_{\cO_{\mathfrak{IG}_{ r,n}}}(\cO_{\mathfrak{IG}_{r,n}}/\cI_n),
\end{equation}
where $\Omega_0$ is the $\cO_{\mathfrak{IG}_{ r,n}}$-submodule of $\omega_0$ generated by all the lifts of $\mathrm{dlog}_0(D_{n})$. 
By construction, there exist on $\mathfrak{IG}_{r,n}$ a universal canonical generator $P_{0,n}$ of $D_{n}$, and universal points $P_{\dP,n}$ of order $p^{n}$ in $\cG_\dP[p^{n}]$. We put:
\begin{equation}\label{e:omega for the machinery}
\Omega:= \Omega_0\oplus \bigoplus_{\mathfrak{p} \neq \mathfrak{p}_0}\omega_{\mathfrak{p}}= \Omega_0\oplus\Omega^0
\end{equation}
 where $\Omega^0= \bigoplus_{\mathfrak{p} \neq \mathfrak{p}_0}\omega_{\mathfrak{p}}$, and we denote $\overline{\Omega}$ the associated $\cO_{\overline{\mathfrak{IG}_{r,n}}}$-module. Now we produce marked sections in $\overline{\Omega}$ in the sense of \S\ref{ss:formal vector bundles} as follows. Let $\mathfrak{p}\mid p$ and we consider two cases: 
\begin{itemize}
	\item if $\mathfrak{p}= \mathfrak{p}_0$ we denote by $s_0 \in \overline{\Omega}$ the image $\mathrm{dlog}_0(P_{0,n})$ using the isomorphism (\ref{omega canonical vs G}). 
	\item if $\mathfrak{p} \neq \mathfrak{p}_0$ we firstly consider the decomposition $\omega_{\mathfrak{p}}= \oplus_{\tau \in \Sigma_\mathfrak{p}}\omega_{\mathfrak{p}, \tau}$ over $\mathfrak{IG}_{r,n}$ and the dlog map: 
	\begin{equation}\label{dlogP}
	\mathrm{dlog}_{\mathfrak{p}}: \cG_{\mathfrak{p}}[p^{n}] \longrightarrow \omega_{\cG_{\mathfrak{p}}[p^{n}]^D}= \omega_{\mathfrak{p}}/p^{n}\omega_{\mathfrak{p}}= \bigoplus_{\tau \in \Sigma_\mathfrak{p}}\omega_{\mathfrak{p}, \tau}/p^{n}\omega_{\mathfrak{p}, \tau}.
	\end{equation}
Hence the image of $P_{\mathfrak{p}, n}$ through $\mathrm{dlog}_{\mathfrak{p}}$ provides a set of sections $\{s_{\mathfrak{p}, \tau}\}_{\tau \in \Sigma_\mathfrak{p}}$ of $\omega_{\mathfrak{p}, \tau}/\cI_n\omega_{\mathfrak{p}, \tau}\subseteq\omega_{\mathfrak{p}}/\cI_{n}\omega_{\mathfrak{p}}$. 
\end{itemize}

 The following proposition allow us to use the constructions recalled in \S\ref{ss:formal vector bundles}: 

\begin{lemma}\label{l:they are marked sections} We have the following facts:
\begin{itemize}
	\item the $\cO_{\mathfrak{IG}_{r,n}}$-module $\Omega_0$ is locally free of rank $1$, and $\Omega$ is locally free of rank $2d-1$.
	\item the $\cO_{\overline{\mathfrak{IG}_{r,n}}}$-module generated by the set $\{s_0\}\cup\bigcup_{\mathfrak{p}\neq \mathfrak{p}_0} \{s_{\mathfrak{p}, \tau}\}_{\tau \in \Sigma_\mathfrak{p}}$ is a locally direct summand of $\overline{\Omega}$.
\end{itemize}	
\end{lemma}
\begin{proof} 
Follows directly from isomorphism \eqref{omega canonical vs G} and the fact that, since $\cG_\dP$ is \'etale, ${\rm dlog}_\dP$ provides an isomorphism between $\omega_\dP/p^{n}\omega_\dP$ and $\cG_\dP[p^{n}]\otimes_{\Z_p}(\cO_{\mathfrak{IG}_{r,n}}/p^{n}\cO_{\mathfrak{IG}_{r,n}})$.
\end{proof}

Applying the construction recalled in \S\ref{ss:formal vector bundles} to the locally free $\cO_{\mathfrak{IG}_{r,n}}$-module $\Omega$ and the marked sections ${\bf s}:=\{s_0\}\cup\bigcup_{\mathfrak{p}\neq \mathfrak{p}_0} \{s_{\mathfrak{p}, \tau}\}_{\tau \in \Sigma_\mathfrak{p}}$ of $\overline{\Omega}$ we obtain the formal scheme $\V_0(\Omega,{\bf s})$ over $\mathfrak{IG}_{r,n}$. By construction we have the following tower of formal schemes:  
$$\xymatrix{\V_0(\Omega,{\bf s})\ar[r]&\mathfrak{IG}_{r,n} \ar^{g_n}[r]&\dX_{r}.}$$
By construction for any $\dX_r$-scheme $T$ we have:
\[
\V_0(\Omega,{\bf s})(T)=\{(\rho,\varphi)\in \mathfrak{IG}_{r,n}(T)\times\Gamma(T,\rho^\ast\Omega^\vee);\quad  \varphi(\rho^\ast s_i)\equiv1\mod \cI_n\}
\]

Let ${\bf s}^0:=\bigcup_{\mathfrak{p}\neq \mathfrak{p}_0} \{s_{\mathfrak{p}, \tau}\}_{\tau \in \Sigma_\mathfrak{p}}$ then from (\ref{e:omega for the machinery}) we have:
\[
\V_0(\Omega,{\bf s})=\V_0(\Omega_0,s_0)\times_{\mathfrak{IG}_{r,n}}\V_0(\Omega^0,{\bf s}^0).
\]

As we are interested in locally analytic distributions (rather than functions) we perform the following construction. Let $t_0\in \overline{\Omega_0}^\vee$ be such that $t_0(s_0)=1$, $t_{\mathfrak{p}, \tau}\in\bar\omega_{\dP,\tau}$ be any section such that $\varphi_\tau(s_{\mathfrak{p}, \tau}\wedge t_{\mathfrak{p}, \tau})=1$ and $Q_{{\bf s}^0}\subset \overline{\Omega^0}$ be the direct summand generated by the sections in ${\bf s}^0$. We put ${\bf t}^0:=\bigcup_{\mathfrak{p}\neq \mathfrak{p}_0} \{t_{\mathfrak{p}, \tau}\}_{\tau \in \Sigma_\mathfrak{p}}$ and we define: 
\[
\widehat\V(\Omega,{\bf s}):=\V_0(\Omega_0^\vee,t_0)\times_{\mathfrak{IG}_{r,n}}\widehat\V_{Q_{{\bf s}^0}}(\Omega^0,{\bf t}^0)\stackrel{f_0}{\longrightarrow}\mathfrak{IG}_{r,n}\stackrel{g_n}{\longrightarrow}\dX_{r}.
\]
The $t_{\mathfrak{p}, \tau}$ are defined modulo $Q_{{\bf s}^0}$ which is fine because remark \ref{DepOnQuot}. By construction we have $\widehat\V(\Omega,{\bf s})(T)=$
\[\{(\rho,\varphi)\in \mathfrak{IG}_{r,n}(T)\times\Gamma(T,\rho^\ast(\Omega_0^\vee\oplus\Omega^0)^\vee);\quad  \varphi(\rho^\ast t_0)\equiv\varphi(\rho^\ast t_{\dP,\tau})\equiv1,\; \varphi(\rho^\ast s_{\dP,\tau})\equiv0\mod \cI_n\}.
\]

Recall that the morphism $g_n$ is endowed with an action of $(\cO/p^{n}\cO)^\times$. And then $\widehat\V(\Omega,{\bf s})/\dX_r$ is equipped with an action of $\cO^\times(1+\cI_n{\rm Res}_{\cO_F/\Z}\G_a)\subseteq{\rm Res}_{\cO/\Z_p}\G_m$. In fact if $(\rho,\varphi)\in \widehat\V(\Omega,{\bf s})(T)$ and $\lambda(1+\gamma) \in\cO^\times(1+\cO_F\otimes_{\Z}\cI_n\cO_T)$ we put:
$$\lambda (1+\gamma)\ast(\rho,\varphi)=(\lambda\rho,\lambda(1+\gamma)\ast \varphi),$$
where 
 $\ast$ denotes the extension of the natural action of $\cO$ on $\Omega^\vee$ given by $\lambda\ast \varphi(w):=\varphi(\lambda w)$. This action is well defined since for each $\tau\neq \tau_0$ we have $\left(\lambda_\tau(1+\gamma_\tau))\ast \varphi)\right)((\lambda_\tau\rho)^\ast s_{\dP,\tau})\equiv\lambda_\tau\varphi(\rho^\ast \lambda_\tau s_{\dP_\tau})\equiv\lambda_\tau^2\varphi(\rho^\ast s_{\dP_\tau})\equiv 0\mod \cI_n$, and for each $\tau$ we have $\left(\lambda_\tau(1+\gamma_\tau)\ast \varphi)\right)((\lambda_i\rho)^\ast t_{\tau})\equiv\lambda_i\varphi(\rho^\ast \lambda_\tau^{-1} t_{\tau}) =\varphi(\rho^\ast  t_{\tau})\mod \cI_n.$


Since $r\geq n$ by Lemma \ref{weightextends} (with $\lambda={\rm Hdg}^{\frac{p^n}{p-1}}$) the character ${\bf k}_n^0$ extends to a locally analytic character 
\[
{\bf k}_n^0:\cO^\times(1+\cO_F\otimes_{\Z}\cI_n\cO_{\dX_r}\otimes_{\Z_p}\Lambda_n^0)\longrightarrow \cO_{\dX_r}\otimes_{\Z_p}\Lambda_n^0.
\]

\begin{defi} We consider the following sheaves over $\dX_{r}\times\dW_n$:
$$\cF_n:= \left((g_n\circ f_{0})_{\ast}\cO_{\widehat\V(\Omega,{\bf s})}\hat\otimes\Lambda_n\right)[{\bf k}_{n}^0] \qquad \qquad \Omega^{{\bf k}_{f}}:=\left( g_{1, \ast}(\cO_{\mathfrak{IG}_{1}})\hat\otimes\Lambda_n\right)[{\bf k}_{f}].$$
The formal \emph{overconvergent modular sheaf} over $\dX_{r}\times\dW_n$ is defined as $\Omega^{{\bf k}_{n}}:= \Omega^{{\bf k}_{n}^0}\otimes_{\cO_{\dX_{r}\times\dW_n}}\Omega^{{\bf k}_{f}}$, where
\[
\Omega^{{\bf k}_{n}^0}:= \cF_n^\vee=\Hom_{\dX_{r}\times\dW_n}\left(\cF_n,\cO_{\dX_{r}\times\dW_n}\right).
\]
\end{defi}

Observe that $\cF_n$ is the sheaf on $\dX_{r}\times\dW_n$ given by the sections $s$ of $(g_n\circ f_{0})_{\ast}\cO_{\widehat\V(\Omega,{\bf s})}\hat\otimes\Lambda_n$ such $t\ast s={\bf k}_n^0(t)\cdot s$, for all $t\in\cO^\times(1+\cI_r{\rm Res}_{\cO_F/\Z}\G_a)$.


\begin{defi} A section in $\mathrm{H}^0(\dX_{r}\times\dW_{n}, \Omega^{{\bf k}_{n}})$ is called a \emph{family of locally analytic Overconvergent modular forms} for the group $G'$.
\end{defi}


\begin{rmk} The present construction performs the $p$-adic variation of the modular sheaves for unitary Shimura curves over a $d$-dimensional weight space.  Compare with \cite{brasca13} and \cite{ding17} where the $p$-adic variation is essentially constructed over a curve inside our weight space.
\end{rmk}

\begin{rmk}\label{proof thm 1 intro} Considering the adic generic fiber of $\cF_n$ we obtain a sheaf of Banach modules over $\cX_r\times \cW_n$ which by construction satisfies the properties stated in \ref{t:main thm 1}(i). The sheaves $\Omega^{{\bf k}_{n}}$ are introduced to construct of triple product $p$-adic $L$-functions.
\end{rmk}

\subsection{Local description of $\cF_n$}\label{locdesc} 
We use the notations as in the last section and we consider the sheaf over $\mathfrak{IG}_{r,n}\times\dW_n^0$ given by $\cF_n'= \left((f_{0})_{\ast}\cO_{\widehat\V(\Omega,{\bf s})}\hat\otimes\Lambda_n\right)[{\bf k}_{n}^0]$ where the action is that of $1+\cI_n{\rm Res}_{\cO_F/\Z}\G_a$.

 Let $\rho: \Spf(R)\rightarrow\mathfrak{IG}_{r,n}\times\dW_n^0$ be a morphism of formal schemes over $\dW_n^0$ without $p$-torsion such that $\rho^\ast\Omega_{0}^\vee$ and $\rho^\ast\omega_{\mathfrak{p}, \tau}$ are free $R$-module of rank $1$ and $2$ respectively. We fix basis $\rho^\ast\Omega_{0}^\vee=Re_0$ and $\rho^\ast\omega_{\mathfrak{p}, \tau}=Rf_{\tau} \oplus Re_{\tau}$ for $\mathfrak{p}\neq \mathfrak{p}_0$ such that $f_\tau\equiv s_{\dP,\tau}$, $e_0\equiv t_0$, and $e_\tau\equiv t_{\dP,\tau}$ modulo $\rho^\ast\mathcal{I}_n$. Moreover we also assume that $\rho^\ast\mathcal{I}_n$ is generated by some $\alpha_n\in R$.

We denote by $e_0^\vee, e_\tau^\vee,f_\tau^\vee\in\rho^\ast(\Omega_0^\vee\oplus\Omega^0)^\vee$ the dual $R$-basis, then by definition, we have that
\begin{eqnarray*}
	\widehat\V(\Omega,{\bf s})(\Spf(R))&=&\left\{\sum_{i=\tau}a_if_i^\vee+\sum_{i=0,\tau}b_ie_i^\vee;\qquad b_0,b_\tau\in (1+\alpha_nR),\quad a_\tau\in \alpha_nR\right\},
\end{eqnarray*}
By Hypothesis \ref{hypothesis 1} we have the isomorphism of algebras $\cO_{F}\otimes_{\Z}R\simeq R^{\Sigma_p}$ and, under this isomorphism $(1+\cO_F\otimes_{\Z}\cI_nR)$ corresponds to $(1+\alpha_n R)^{\Sigma_p}$. 

By \S \ref{ss:formal vector bundles} we have that $\rho^\ast\cO_{\widehat\V(\Omega,{\bf s})}=R\langle Y_{\tau_0},(Z_\tau,Y_{\tau})_{\tau\in\Sigma_\dP,\dP\neq\dP_0}\rangle$, where 
\begin{eqnarray*}
	&&\sum_{i=\tau}a_if_i^\vee+\sum_{i=0,\tau}b_ie_i^\vee\in \widehat\V(\Omega,{\bf s})(\Spf(R))\subseteq\Gamma(\Spf(R),\rho^\ast(\Omega_0^\vee\oplus\Omega^0)^\vee), 
\end{eqnarray*}
corresponds to the point $(Y_\tau=\frac{b_\tau-1}{\alpha_n};\;Z_\tau=\frac{a_\tau}{\alpha_n})$.
Recall that the action of $(1+\lambda\otimes x)\in(1+\cO_F\otimes \cI_nR)$ (corresponding to $(1+x\tau(\lambda))_\tau\in(1+\cI_nR)^{\Sigma_p}$) on $\phi\in\V_0(\Omega,{\bf s})(\Spf(R))$ is given by 
\[
(1+\lambda\otimes x) \ast\phi(w)=\phi(w)+x\phi(\lambda w).
\]
Hence, we deduce that $(t_\tau)\in (1+\cI_nR)^{\Sigma_p}$ acts on $f_\tau^\vee$ and $e_\tau^\vee$ by multiplication by $t_\tau$, and on $e_0^\vee$ by multiplication by $t_0$. Therefore $(t_\tau)_\tau$ acts on the variables $Z_\tau$ for  $\tau\neq\tau_0$  and $Y_\tau$ for $\tau\in \Sigma_F$ by
 $Y_\tau\mapsto (t_\tau-1)\alpha_n^{-1}+t_\tau Y_\tau$, and $Z_\tau\mapsto t_\tau Z_\tau$. 

\begin{lemma}\label{l:locdesc}
There exists $P_{\alpha_n}\in R\langle x_\tau\rangle_{\tau\in\Sigma_p}$ such that ${\bf k}_{n}^{0}(1+\alpha_nz_\tau)=P_{\alpha_n}((z_\tau)_{\tau\in\Sigma_p})$ and:
	\begin{eqnarray*}
		\rho^\ast\cF_n'&=& P_{\alpha_n}((Y_\tau)_{\tau\in\Sigma_\dP})\cdot R\left\langle\frac{Z_\tau}{1+\alpha_nY_\tau}\right\rangle_{\tau\in\Sigma_\dP,\dP\neq\dP_0}.
	\end{eqnarray*}
\end{lemma}
\begin{proof} It is clear that $P_{\alpha_n}(Y_\tau)\cdot R\left\langle\frac{Z_\tau}{1+\alpha_nY_\tau}\right\rangle\subseteq\rho^\ast(\cF_n')$. On the other side, if $f\in \rho^\ast(\cF_n')=R\langle Z_{\tau}, Y_{\tau}\rangle[{\bf k}_{n}^{0}]$, then $f/P_{\alpha_n}(Y_\tau)$ lies in
	\[
	R\langle Z_{\tau}, Y_\tau\rangle^{(1+\alpha_nR)^{\Sigma_p}}=R\langle Y_{\tau_0}\rangle^{1+\alpha_nR}\hat\otimes \hat{\bigotimes}_{\tau\in\Sigma_\dP,\dP\neq\dP_0}R\langle Z_\tau,Y_\tau\rangle^{1+\alpha_nR}.
	\]
	Similarly as in \cite[Lemma 3.9 and Lemma 3.13]{AI17} one proves that $R\langle Y_{\tau_0}\rangle^{1+\alpha_nR}=R$ and $R\langle Z_\tau,Y_\tau\rangle^{1+\alpha_nR}=R\langle Z_\tau/(1+\alpha_nY_\tau)\rangle$, hence the result follows. 
	\end{proof}
\begin{rmk} To obtain a description of $\cF_n$ we need to consider the action of $\cO^\times$ and descend using the morphism $g_n: \mathfrak{IG}_{r,n} \rightarrow \mathfrak{X}_r$ in the same way as in \cite[\S 3.2.1]{AI17}
\end{rmk}


\subsection{Overconvergent modular forms \`a la Katz} \label{ss:overconvergent modular forms a la katz}
Here we give a moduli description of the families of overconvergent modular forms introduced above.
\subsubsection{Notations}\label{NotCD}
 Let $R$ be a complete local $\hat\cO$-algebra. 
 \begin{defi}
Let $k:\cO^{\tau_0\times}\rightarrow R$ be a character and $n \in \N$. We denote by  $C^{k}_n(\cO^{\tau_0},R)$ the $R$-module of the functions $f:\cO^{\tau_0\times}\times \cO^{\tau_0}\rightarrow R$ such that: 
\begin{itemize}
\item $f(tx, ty)=k(t)\cdot f(x, y)$  for each  $t\in\cO^{\tau_0\times}$ and $(x, y) \in \cO^{\tau_0\times}\times \cO^{\tau_0}$;
\item  the function $y\mapsto f(1, y)$ is analytic on the disks $y_0+ p^n\cO^{\tau_0}$  where $y_0$ below to a system of representatives of $\cO^{\tau_0}/p^n\cO^{\tau_0}$.
\end{itemize}
The space of \emph{distributions} is defined by $D^{k}_n(\cO^{\tau_0},R):={\rm Hom}_R(C^{k}_n(\cO^{\tau_0},R),R).$
\end{defi}

\begin{rmk}\label{rmkanalicity}
Note that $C^{k}_n(\cO^{\tau_0},R)\subseteq C^{k}_{n+1}(\cO^{\tau_0},R)$ and if $k=\underline{k}\in\N[\Sigma_F\smallsetminus \{\tau_0\}]$ is a classical weight then $C^{\underline{k}}_0(\cO^{\tau_0},R)$ is the module of analytic functions and naturally contains $\Sym^{\underline{k}}(R^2)$. We obtain a natural projection $D^{\underline{k}}_0(\cO^{\tau_0},R)\rightarrow\Sym^{\underline{k}}(R^2)^\vee$. 
\end{rmk}

We have a natural action of the subgroup $K_0(p)^{\tau_0}\subset\GL_2(\cO^{\tau_0})$ of upper triangular matrices modulo $p$ on $C^{k}_n(\cO^{\tau_0},R)$ and $D^{k}_n(\cO^{\tau_0},R)$ given by: 
\[
(g\ast f)(x,y)= f((x,y)g) \qquad \qquad (g\ast\mu)(f):=\mu(g^{-1}\ast f),
\]
where $g\in K_0(p)^{\tau_0}$, $f\in C^{k}_n(\cO^{\tau_0},R)$ and $\mu\in D^{k}_n(\cO^{\tau_0},R)$. Since $y \mapsto f(1, y)$ is analytic on the disks $y_0+ p^n\cO^{\tau_0}$ this action extends to an action of $K_0(p)^{\tau_0}(1+p^n\M_2(R\otimes_{\Z_p}\cO^{\tau_0}))$.

\subsubsection{Locally analytic functions} For $r \geq n$ we are going to describe the elements of 
$H^0(\dX_{r}\times\dW_{n}, \cF_n)$ as rules, extending the classical interpretation due to Katz. 

Now let $R$ be a $\Lambda_n$-algebra and $P\in(\dX_{r}\times\dW_n)(\Spf(R))$ be a point corresponding to the isomorphism class of  $(A,\iota,\theta,\alpha^{\dP_0})$. By the moduli interpretation of $\dX_r$, $A$ is endowed with subgroups $C_\dP\subset A[\dP]^{-,1}$ isomorphic to $\cO_\dP/\dP$, for $\dP\neq\dP_0$. Write $w=(f_0,\{(f_\tau,e_\tau)\}_\tau)$  where  $f_0$ is a basis of $\Omega_0$ and $\{e_\tau, f_\tau\}$  is a basis of  $\omega_{\dP,\tau}$, for $\dP\neq\dP_0$ and $\tau \in \Sigma_{\mathfrak{p}}$, such that sections $f_\tau$ generate $\omega_{C_\dP^D}$. We suppose that reductions modulo $\cI_n:=p^{n}{\rm Hdg}(A)^{-\frac{p^{n}}{p-1}}\subset R$ of $f_0$ and $(f_\tau,e_\tau)$ lie in the image of:
\[
{\rm dlog}_0:D_{n}(R)\longrightarrow\Omega_0\otimes_R(R/\cI_n)\qquad \text{and} \qquad {\rm dlog}_{\dP,\tau}:\mathcal{G}_{\mathfrak{p}}[p^{n}]\longrightarrow\omega_{\dP,\tau}\otimes_R(R/\cI_n),
\]
here as above $\mathcal{G}_{\mathfrak{p}}=A[\mathfrak{p}^\infty]^{-,1}$, $\mathcal{G}_{0}=A[\mathfrak{p}_0^\infty]^{-,1}$, $C_n$ is the canonical subgroup of $\mathcal{G}_0[\mathfrak{p}_0^{n}]$ and $D_n= \mathcal{G}_0[\mathfrak{p}_0^{n}]/C_n$. Our assumptions imply that the elements $p^{n-1}{\rm dlog}_{\dP,\tau}^{-1}(f_\tau)$ generate $C_\dP$.



Any linear combination of pre-images in $A[p^n]^{-,1}$ of  $f_0$ and $\{f_\tau\}_{\tau}$ provides a point in $\mathfrak{IG}_{r,n}(\Spf(R))$. 
Given $(x,y)\in\cO^{\tau_0\times}\times\cO^{\tau_0}$, we define $P_{(x,y)}^w\in \mathfrak{IG}_{r,n}(\Spf(R))$ to be the point given by the combination:
\[
\left({\rm dlog}^{-1}(f_0),\varphi_\tau(f_\tau\wedge e_\tau)^{-1}\left|\begin{array}{cc}{\rm dlog}^{-1}f_\tau&{\rm dlog}^{-1}e_\tau\\ \tau(x)&\tau(y)\end{array}\right|_\tau\right)\in  A[p^n]^{-,1}.
\]
It is clear that $\langle p^{n-1}(\tau(y){\rm dlog}^{-1}f_\tau-\tau(x){\rm dlog}^{-1}e_\tau)_\tau\rangle$ generate $A[\dP]^{-,1}/C_\dP$, 
thus the point $P_{(x,y)}^w\in \mathfrak{IG}_{r,n}(R)$ lies above $P\in \dX_r(R)$.

Thus for each $s \in H^0(\dX_{r}\times\dW_{n}, \cF_n)$ we assigns the rule mapping each tuple $(A,\iota,\theta,\alpha^{\dP_0},w)$ as above to the locally analytic function $s(A,\iota,\theta,\alpha^{\dP_0},w)$ given by:
\[
s(A,\iota,\theta,\alpha^{\dP_0},w)(x,y):=s(P_{(x,y)}^w,f_0+\sum_\tau(\tau (x)f_\tau^\vee+\tau(y)e_\tau^\vee)) \in C^{{\bf k}_n^{\tau_0,0}}_n(\cO^{\tau_0},R).
\]

To verify that it is well defined firstly observe that $(P^w_{(x,y)},f_0+\sum_\tau\tau (x)f_\tau^\vee+\tau(y)e_\tau^\vee)\in \hat V(\Omega,{\bf s})(R)$. In fact, we have $s_0={\rm dlog}_0(P^w_{(x,y)})=f_0$, $s_{\dP,\tau}={\rm dlog}_{\dP, \tau}(P^w_{(x,y)})=\varphi_\tau(f_\tau\wedge e_\tau)^{-1}(\tau(y)\cdot f_\tau-\tau(x)\cdot e_\tau)$ and  $t_{\dP,\tau}=\tau(r)\cdot f_\tau+\tau(s)\cdot e_\tau$ for any $r,s\in\cO^{\tau_0}$ such that $rx+psy=1$. Hence,
\[
(f_0+\sum_\tau\tau (x)f_\tau^\vee+\tau(y)e_\tau^\vee)(s_{\dP,\tau})\equiv0;\qquad (f_0+\sum_\tau\tau (x)f_\tau^\vee+\tau(y)e_\tau^\vee)(t_{i})\equiv1\;(i=0,(\dP,\tau)).
\]
Moreover, by the local description of \S \ref{locdesc}, the function $y \mapsto s(A,\iota,\theta,\alpha^{\dP_0},w)(1, y)$ is analytic over any open $y_0+ p^n\cO^{\tau_0}$ (where $P_{(x,y)}^w$ is constant). Finally if $t\in\cO^{\tau_0\times}$ then:
\begin{eqnarray*}
s(A,\iota,\theta,\alpha^{\dP_0},w)(tx, ty)&=&s\left(t P^w_{(x,y)},f_0+t\left(\sum_\tau\tau (x)f_\tau^\vee+\tau(y)e_\tau^\vee\right)\right)\\
&=&s(t\ast (P^w_{(x,y)},f_0+\sum_\tau\tau (x)f_\tau^\vee+\tau(y)e_\tau^\vee))\\
&=&{\bf k}_n^{\tau_0,0}(t)\cdot s(A,\iota,\theta,\alpha^{\dP_0},w)(x,y).
\end{eqnarray*}

If $t\in \Z_p^\times(1+\cI_n R)$ and $g\in K_0(p)^{\tau_0}(1+\cI_n \M_2(R))^{\Sigma_F\setminus\{\tau_0\}}$ then $(tf_0, \{(f_\tau,e_\tau)g\}_{\tau})$ satisfies the same properties as $w$. If $t\in \Z_p^\times$ we have:
\[
s(A,\iota,\theta,\alpha^{\dP_0}, tw)(x,y)=s(t\ast (P^w_{(x,y)},f_0+\sum_\tau\tau (x)f_\tau^\vee+\tau(y)e_\tau^\vee))={\bf k}_{n,\tau_0}^0(t)\cdot s(A,\iota,\theta,\alpha^{\dP_0}, w)(x,y).
\]

If $g\in K_0(p)^{\tau_0}$ we have $P_{(x,y)g}^{wg}=P_{(x,y)}^w$ and then:
\[
s(A,\iota,\theta,\alpha^{\dP_0},wg)((x,y)g)=s\left(P_{(x,y)}^w,f_0+\sum_\tau(\tau (x),\tau(y))\tau g\cdot\tau g^{-1}\left(\begin{array}{c}f_\tau^\vee\\e_\tau^\vee\end{array}\right)\right)= s(A,\iota,\theta,\alpha^{\dP_0},w)(x,y).
\]

These two last properties, and the fact that the action of $\Z_p^{\times}\times K_0(p)^{\tau_0}$ can be extended to an action of $R^\times\times K_0(p)^{\tau_0}(1+\cI_n\M_2(R))$ analytically, allows us to extend the rule defined above to any basis $w$ of $\Omega$ satisfying the above properties but not necessarily lying in the image of the dlog maps. Moreover, this rule characterizes the section $s$.

\subsubsection{Distributions} Now let $\mu \in H^0(\dX_{r}\times\dW_n, \Omega^{{\bf k}_{n}})$ then using the construction above we obtain a rule that assigns to each tuple $(A,\iota,\theta,\alpha^{\dP^0},w)$ as above, a distribution $\mu(A,\iota,\theta,\alpha^{\dP^0},w)\in D^{{\bf k}_n^{\tau_0}}_n(\cO^{\tau_0},R)$. 
The rule $(A,\iota,\theta,\alpha^{\dP_0},w)\mapsto \mu(A,\iota,\theta,\alpha^{\dP_0},w)$ characterizes $\mu$ and satisfies:
\begin{itemize}
\item[(B1)] $\mu(A,\iota,\theta,\alpha^{\dP_0},w)$ depends only on the $R$-isomorphism class of $(A,\iota,\theta,\alpha^{\dP_0})$.

\item[(B2)] The formation of $\mu(A,\iota,\theta,\alpha^{\dP_0},w)$ commutes with arbitrary extensions of scalars $R\rightarrow R'$ of $\Lambda_n$-algebras.


\item [(B3-a)] $\mu(A,\iota,\theta,\alpha^{\dP_0}, a^{-1}w)=k_n^{\tau_0}(t)\cdot\mu(A,\iota,\theta,\alpha^{\dP_0},w)$, for all $t\in\Z_p^\times$.

\item [(B3-b)] $g\ast\mu(A,\iota,\theta,\alpha^{\dP_0},wg)=\mu(A,\iota,\theta,\alpha^{\dP_0},w)$, for all $g\in K_0(p)^{\tau_0}$.
\end{itemize}



\begin{rmk}
Note that this description fits with the classical setting of classical Katz modular forms explained in \S \ref{KatzModForm}.
\end{rmk}

\section{$p$-adic families and eigenvarieties for $G$} In this section we construct $d+1$-dimensional families of Hecke eigenvalues of automorphic forms over $(B\otimes \A_F)^\times$. As usual it is consequence of the spectral properties of the $U_p$-operator acting on the space of families of overconvergent modular forms. 

\subsection{Weight space for $G$}\label{ss:weight space for G}

In \S \ref{ss:weight space} we studied the weight space for $G'$, denoted by $\dW$, which is the formal spectrum of  $\Lambda_F=\Z_p[[\cO^\times]]$.  For $G$ we consider the algebras $\Lambda_F^G:=\Z_p[[\cO^\times\times\Z_p^\times]]$ and:
$$\Lambda_F^{G,0}=\Lambda_F^{0}\hat\otimes\Z_p[[1+ p\Z_p]]=\Z_p[[T_1,\cdots,T_d,T]] \qquad\qquad \Lambda_n^G:=\Lambda_F^G\left\langle\frac{T_1^{p^n}}{p},\cdots,\frac{T_d^{p^n}}{p},\frac{T^{p^n}}{p}\right\rangle$$
Similarly as in \S \ref{ss:weight space}, we have a decomposition $\Lambda_F^G=\Z_p[H']\hat\otimes\Lambda_F^{G,0}$ where $H'$ is an abelian finite group. 

Then we consider the formal schemes $\dW^G:=\Spf(\Lambda_F^G)$ and $\dW_n^G:=\Spf(\Lambda_n^G)$. The scheme $\dW^G$ is called the \emph{weight space for $G$}. Similarly as in \S \ref{ss:weight space}, we have the moduli interpretation: $$\dW^G(R)={\rm Hom}_{\rm cont}(\cO^\times\times\Z_p^\times,R).$$

Consider $\cO^\times\rightarrow \cO^\times\times\Z_p^\times$ given by $t\longmapsto (t^{-2}, N(t))$, where $N(t)$ means the norm of $t$. We obtain a morphism of weight spaces:
 \begin{equation}\label{e:map between weight spaces}
 k:\dW^G\rightarrow\dW
 \end{equation}
In terms of characters, we have that if $(r,\nu)\in \dW^G(R)$ and $t \in \cO^\times$ then $k(r,\nu)(t)=\nu(N(t))\cdot r(t)^{-2}.$

\begin{lemma}
The morphism $k$ extends to a well defined morphism
$k:\dW_n^G\rightarrow \dW_n$.
\end{lemma}
\begin{proof}
By definition $\Lambda_F\left\langle\frac{T_1^{p^n}}{p},\cdots,\frac{T_d^{p^n}}{p}\right\rangle$ is the completion of the continuous $\Lambda_F$-algebra 
\[
\Lambda_F\stackrel{\varphi}{\longrightarrow}\Lambda_F\left(\frac{T_1^{p^n}}{p},\cdots,\frac{T_d^{p^n}}{p}\right)\subseteq\Lambda_F\left[\frac{1}{p}\right], 
\]
satisfying that $\frac{T_i^{p^n}}{p}$ is power bounded for all $i=1,\cdots,d$, and the following universal property: For any non-archimedean topological continuous $\Lambda_F$-algebra $f:\Lambda_F\rightarrow B$ such that $f(p)$ is invertible in $B$ and $f(T_i)f(p)^{-1}$ is power bounded in $B$, then there exists a unique continuous homomorphism $g:\Lambda_F\left(\frac{T_1^{p^n}}{p},\cdots,\frac{T_d^{p^n}}{p}\right)\rightarrow B$ with $f=g\circ\varphi$. 

Each variable $T_i$ corresponds to a $\Z_p$-generator $\alpha_i$ of $\Lambda_F^0$, and $T$ corresponds to $\exp(p)\in 1+p\Z_p$.
Note that the corresponding algebra morphism $k^\ast:\Lambda_F\rightarrow\Lambda_F^G$ satisfies $k^\ast(T_i+1)=(T_i+1)^{-2}(T+1)^{\beta_i}$, where $\log_p({\rm Norm}_{F/\Q}\alpha_i)=p\beta_i$. Hence considering the composition
\[
F:\Lambda_F\stackrel{k^\ast}{\longrightarrow}\Lambda_F^G\longrightarrow\Lambda_F^G\left(\frac{T_1^{p^n}}{p},\cdots,\frac{T_d^{p^n}}{p},\frac{T^{p^n}}{p}\right),
\]
we have that 
\begin{eqnarray*}
F(T_i^{p^n})&=&\left((T_i+1)^{-2}(T+1)^{\beta_i}-1\right)^{p^n}=\left(\sum_{n\geq 1}\binom{\beta_i}{n}T^n+\sum_{m\geq 1}\binom{-2}{m}T_i^m\sum_{n\geq 0}\binom{\beta_i}{n}T^n\right)^{p^n}\\
&=&p\left(\frac{T_i^{p^n}}{p}\left(\sum_{m\geq 1}\binom{-2}{m}T_i^{m-1}\right)^{p^n}\left(\sum_{n\geq 0}\binom{\beta_i}{n}T^n\right)^{p^n}+\frac{T^{p^n}}{p}\left(\sum_{n\geq 1}\binom{\beta_i}{n}T^{n-1}\right)^{p^n}+\lambda\right),
\end{eqnarray*}
for some $\lambda\in\Lambda_F^G$. This implies that $p^{-1}F(T_i^{p^n})\in \Lambda_F^G\left(\frac{T_1^{p^n}}{p},\cdots,\frac{T_d^{p^n}}{p},\frac{T^{p^n}}{p}\right)$ and it is power bounded. Thus, $F$ factors through $\varphi$ making the following diagram commutative
\[
\xymatrix{
\Lambda_F\ar[d]_\varphi\ar[r]^{k^\ast}&\Lambda_F^G\ar[d]\\
\Lambda_F\left(\frac{T_1^{p^n}}{p},\cdots,\frac{T_d^{p^n}}{p}\right)\ar[r]^{k'}&\Lambda_F^G\left(\frac{T_1^{p^n}}{p},\cdots,\frac{T_d^{p^n}}{p},\frac{T^{p^n}}{p}\right).
}
\]
The completion of the continuous morphism $k'$ provides the morphism we are looking for.
\end{proof}

\subsection{Overconvergent descent from $G'$ to $G$}

Recall the Shimura curve $X^{\cC}$ for $\cC\in \Pic(\cO_F)$ introduced in remark \ref{rmkonGammac}. Repeating the constructions performed in \S \ref{s:integral models and canonical groups} and \S \ref{s:over modular sheaves} we obtain for $r\geq n$ a formal scheme $\dX_r^\cC$ and a sheaf $\Omega^{{\bf k}_n}$ over $\dX_r^\cC\times\dW_n$. The irreducible components of  $X^{\cC}$ are in correspondence with $\Pic(E/F,\cN)$ as well as for $\dX_r^{\cC}$ and we denote by $\dX_r^{\cC, 0}$ the irreducible component corresponding to $1 \in \Pic(E/F,\cN)$. 
We consider the universal character
\[
({\bf r}_n,{\bf \nu}_n):\cO^\times\times\Z_p^\times\longrightarrow\Lambda_F^G\subset \Lambda_n^G,
\] 
and its image through $k:\dW^G_n\rightarrow\dW_n$ denoted ${\bf k}_n:=k({\bf r}_n,{\bf \nu}_n):\cO^\times\longrightarrow \Lambda_n^G.$ 
We put:
\[
M^r_{{\bf k}_n}(\Gamma_{1,1}^{\cC}(\cN,p),\Lambda_n^G):=H^0(\dX_r^{\cC, 0}\times\dW_n^G,(\mathrm{id}, k)^\ast\Omega^{{\bf k}_n})\otimes_{\Z_p}\Q_p.
\]

In the same way as at the end of \S \ref{GvsG'} we describe an action of $\Delta$ on $M^r_{{\bf k}_n}(\Gamma_{1,1}^{\cC}(\cN,p),\Lambda_n^G)$. For $t= (t_0,t^0)\in \cO^\times=\Z_p^\times\times\cO^{\tau_0\times}$ we put ${\bf r}_n'(t)={\bf r}_n(t^0){\bf \nu}_n(t_0){\bf r}_n(t_0^{-1})$. Now let $\mu\in M^r_{{\bf k}_n}(\Gamma_{1,1}^{\cC}(\cN,p),\Lambda_n^G)$ and let $(A,\iota,\theta,\alpha^{\dP_0},w)$ over some $R$ be as in \S \ref{ss:overconvergent modular forms a la katz} , then for $s\in\Delta$ we put:
\[
(s\ast\mu)(A,\iota,\theta,\alpha^{\dP_0},w):={\bf r}_{n}'(s)\cdot\mu(A,\iota,s^{-1}\theta,\gamma_s^{-1}\alpha^{\dP_0},w),
\] 
where $\gamma_s\in K_1(\cN,p)$ and $\det(\gamma_s)=s$. This action is well defined, in fact 
given  $\eta\in U_{\cN}$, the isomorphism $\eta:A\simeq A$ provides identifications
 \begin{eqnarray*}
\mu(A,\iota,\theta,\alpha^{\dP_0},w)&=&\mu(A,\iota,\eta^{-2}\theta,\eta^{-1}\alpha^{\dP_0},\eta^{-1} w)={\bf k}_{n,\tau_0}(\eta){\bf k}^{\tau_0}_{n}(\eta^{-1})\cdot\mu(A,\iota,\eta^{-2}\theta,\eta^{-1}\alpha^{\dP_0},w)\\
&=&{\bf r}_n'(\eta^2)\cdot\mu(A,\iota,\eta^{-2}\theta,\eta^{-1}\alpha^{\dP_0},w),
\end{eqnarray*}
where the third equality follows from the fact that ${\rm Norm}_{F/\Q}(\eta)=1$. Hence we have the well defined action of 
for any representative $\epsilon\in(\cO_{F})^\times_+$ of $(\epsilon)\in \Delta$. 

Finally, this action is compatible with the action introduced in \S \ref{GvsG'} under specialization to classical weights.

In a compatible way with corollary \ref{c:automorphic forms as sections} we define:
\begin{defi}\label{d:overconvergent families for G}
The space of \emph{families of locally analytic overconvergent automorphic forms for $G$ with depth of overconvergence $r$} is:
\[
M^r_{{\bf k}_n}(\Gamma_{1}(\cN,p),\Lambda_n^G):=\bigoplus_{\mathfrak{c}\in\Pic(\cO_F)}M^r_{{\bf k}_n}(\Gamma_{1,1}^{\mathfrak{c}}(\cN,p),\Lambda_n^G)^{\Delta}= \bigoplus_{\mathfrak{c}\in\Pic(\cO_F)} \left(H^0(\dX_r^{\cC, 0}\times\dW_n^G,(\mathrm{id}, k)^\ast\Omega^{{\bf k}_n})\otimes_{\Z_p}\Q_p\right)^{\Delta}.
\]
\end{defi}

\subsection{Hecke operators} 

In this section we define the Hecke operators acting on $M^r_{{\bf k}_n}(\Gamma_{1}(\cN,p),\Lambda_n^G)$.

\subsubsection{Hecke operators $T_{g}$, with $g\in G(\A_f^p)$}

Given a point $(A,\iota,\theta,\alpha^{\dP_0})\in(\dX_r^{\cC})^0(R)$ (the connected component of $\dX_r^{\cC}$), and the double coset $K_1^B(\cN)gK_1^B(\cN)=\bigsqcup_ig_iK_1^B(\cN)$ in $G(\A_f^p)$, we can construct the tuple $(A^{g_i},\iota^{g_i},\theta^{g_i},(\alpha^{g_i})^{\dP_0})\in(\dX_r^{\cC_i\cC})^0(R)$ analogously as in \S \ref{Heckops}, where $\cC_i=\det(g_i)$. In the way as in Proposition \ref{HeckopKatz}, we define:
\[
(T_g \mu)_\cC(A,\iota,\theta,\alpha^{\dP_0},w)=\sum_i{\bf r}_n'(\det(\gamma_i))\cdot\mu_{\cC'}(A^{g_i},\iota^{g_i},\det(\gamma_i)^{-1}\theta^{g_i},k_i^{-1}(\alpha^{g_i})^{\dP_0},w),
\]
for any $\mu\in M^r_{{\bf k}_n}(\Gamma_{1}(\cN,p),\Lambda_n^G)$, where $\cC'$ is the class of $\cC\det(g)=\cC\cC_i$, the elements $\gamma_i\in G(F)_+$ and $k_i\in K_1^B(\cN)$ satisfy $b_\cC g_i=\gamma_i^{-1}b_{\cC'}k_i$ (We fix $b_\cC\in G(\A_f^p)$ with $\det(b_\cC)=\cC$), and  the basis $w$ in $\Omega^1_{A^{g_i}}$ is the image of the basis $w$ in $\Omega^1_A$ through the natural isogeny $A\rightarrow A^{g_i}$, that is an isomorphism in tangent spaces.
The fact that the expression is well defined follows from Remark \ref{rmkondefHecke}.

\subsubsection{Hecke operators $U_{\dP}$, where $\dP\mid p$ and $\dP\neq \dP_0$} 
We define $U_\dP$ for any $\dP\neq\dP_0$ dividing $p$ by means of the formula:
\[
(U_\dP \mu)_\cC(A,\iota,\theta,\alpha^{\dP_0},w)=
\sum_{i\in\kappa_\dP} (g_i\ast\mu_{\cC})(A_i,\iota_i,\varpi_\dP\theta_i,\alpha_i^{\dP_0},wg_i),\qquad g_i=\left(\begin{array}{cc}\varpi_\dP&i\\&1\end{array}\right),
\]
where  
$A_i=A/C_i$ with $C_i$ defined as in \S \ref{Upop}, 
following remark \ref{rmkonpts} the morphism $\alpha_i^{\dP_0}$ is  provided by the subgroup $A[\dP]^{-.1}/C_i$, and $wg_i$ is the basis of $\Omega_{A_i}^1$ given by $w$ by means of the natural isogeny $A\rightarrow A_i$ as in remark \ref{rmkbasisUp}. Moreover, the action of $g_i$ on $D_n^{{\bf k}_n^{\tau_0}}(\cO^{\tau_0},R)$ is given by:
\[
\int_{\cO^{\tau_0\times}\times\cO^{\tau_0}}\phi \;d(g_i\ast \mu):=\int_{\cO^{\tau_0\times}\times\cO^{\tau_0}}(\varpi_\dP g_i^{-1}\ast\phi) \;d\mu.
\]
By \S \ref{Upop} this definition is consistent with the classical definition up to constant.

\subsubsection{Hecke operator $U_{\dP_0}$}
We define $U_{\dP_0}$ by means of the formula
\[
(U_{\dP_0} \mu)_\cC(A,\iota,\theta,\alpha^{\dP_0},w)=\frac{1}{p}\sum_{i\in \mathbb{F}_p} \mu_{\cC}(A_i,\iota_i,\varpi_{\dP_0}\theta_i,\alpha_i^{\dP_0},w),
\]
where 
the basis $w$ are defined similarly as above, and $A_i=A/C_i$ is characterized by the fact that $C_i^{-,1}:=C_i\cap A[\dP_0]^{-,1}$ has zero intersection with the canonical subgroup (see \S \ref{Upop}). By \cite[lemma 7.5]{brasca13}, $(A_i,\iota_i,\varpi_{\dP_0}\theta_i,\alpha_i^{\dP_0})\in\dX_r^{\cC}$, thus $U_{\dP_0}$ defines a well defined operator in $M^{r}_{{\bf k}_n}(\Gamma_1(\cN,p),\Lambda_n^G)$.

\begin{rmk}\label{rmkonU}
As noticed above, the specialization of the operators $U_\dP$ at classical weights and the operator $U$ of \S \ref{ss:NA-trilinear products} coincide up to constant. Indeed by \S \ref{Upop}, for any classical weight $x=(\underline{k},\nu)\in \dW^G_n$ we have: 
\[
(U_\dP\mu)_x=\left\{\begin{array}{cc}
\varpi_\dP^{\frac{(\nu+2) \underline{1}+\underline{k}_\dP}{2}}\cdot U\mu_x, &\dP\neq\dP_0\\
p^{\frac{\nu +k_{\tau_0}}{2}}\cdot U\mu_x, &\dP=\dP_0.
\end{array}\right.
\] 
This implies that the pair of eigenvalues $\alpha_\dP$ and $\beta_\dP$ of $U_\dP$ satisfy the Hecke polynomials:
\[
X^2-a_\dP X+\varpi_\dP^{\underline{k}_\dP+\underline{1}},\quad \dP\neq\dP_0,\qquad\qquad
X^2-a_{0} X+p^{k_{\tau_0}-1}, \quad\dP=\dP_0.
\]
\end{rmk}

\subsubsection{Hecke operator $V_{\dP_0}$}
We analogously define $V_{\dP_0}$ as follows:
\[
(V_{\dP_0} \mu)_\cC(A,\iota,\theta,\alpha^{\dP_0},w)= \mu_{\cC'}(A_\infty,\iota_\infty,\varpi_{\dP_0}\theta_\infty,\alpha_\infty^{\dP_0},w),
\]
where 
the basis $w$ are defined similarly as above, and $A_\infty=A/C_\infty$ is characterized by the fact that $C_\infty^{-,1}:=C_\infty\cap A[\dP_0]^{-,1}$ is the canonical subgroup. Since one can prove that $(A_\infty,\iota_\infty,\varpi_{\dP_0}\theta_\infty,\alpha_\infty^{\dP_0})\in \dX_{r-1}^{\cC}$, one obtains that for $r\geq 1$ $V_{\dP_0}$ defines a well defined morphism 
\[
V_{\dP_0}:M^{r-1}_{{\bf k}_n}(\Gamma_1(\cN,p),\Lambda_n^G)\longrightarrow M^{r}_{{\bf k}_n}(\Gamma_1(\cN,p),\Lambda_n^G).
\]

\subsection{Banach space structure and compactness}

Let us consider $\cY:=\dX_r\times\dW_n^G$. Given a point $y=(z,(r,\nu))\in \cY^{rig}$ with residue field $\kappa$, let us consider the corresponding point $y\in \cY$. By the interpretation as Katz modular forms, we have an identification
\[
H^0(\Spf(\cO_\kappa),y^\ast\Omega^{k})\otimes\Q_p\simeq D_n^{k^{\tau_0}}(\cO^{\tau_0},\cO_\kappa)\otimes\Q_p,\qquad k:= k(r, \nu)=(k_{\tau_0},k^{\tau_0}),
\] 
which is a Banach $\kappa$-module. We denote by $|\cdot |_y$ the induced norm, which in fact does not depend on the identification.

For any $\mu\in M^r_{{\bf k}_n}(\Gamma_{1}(\cN,p),\Lambda_n^G)$, we define the supremum norm 
\[
|\mu|={\rm sup}\left\{|y^\ast \mu_\cC|_y;\;\; y\in \cY^{rig,0},\;\cC\in\Pic(\cO_F)\right\},\qquad y^\ast \mu_\cC\in H^0(\Spf(\cO_\kappa),y^\ast\Omega^{k})\otimes\Q.
\]
\begin{lemma}\label{lemmanorm}
$|\cdot |$ is a norm on $M^r_{{\bf k}_n}(\Gamma_{1}(\cN,p),\Lambda_n^G)$ which makes it into a potentially orthonormalizable $\Lambda_n^G\otimes\Q$-Banach module.
\end{lemma}
\begin{proof}
See \cite[Lemma 2.2]{Kassaei06}. It is clear that $|\cdot |$ is an ultrametric norm. We have to prove that it is finite, complete and separated. Let $\{U_i=\Spf(A_i)\}_{i\in I}\subset\cY$ be a finite trivialization of $\Omega$ by affinoids. Thus, by Katz modular form interpretation $H^0(U_i,\Omega^{{\bf k}_n})\stackrel{\sigma_i}{\simeq} D_n^{{\bf k}_n^{\tau_0}}(\cO^{\tau_0},A_i)$. The supremum norm on $A_i\otimes\Q$ induces a finite norm $|\cdot|_i$ on $D_n^{{\bf k}_n^{\tau_0}}(\cO^{\tau_0},A_i)\otimes\Q$. Moreover, we have that $I$ is finite and
$|\mu|={\rm max}\{|\sigma_i\mu_\cC\mid_{U_i}|_i;\;i\in I,\cC\in\Pic(\cO_F)\}$, thus $|\cdot|$ is finite. Note that $|\cdot|_i$ is complete and separated on reduced affinoids. Since $\cY$ is reduced, $|\cdot|$ is separated. We deduce that $|\cdot|$ is also complete since every Cauchy sequence induce a Cauchy sequence in $D_n^{{\bf k}_n^{\tau_0}}(\cO^{\tau_0},A_i)\otimes\Q$ whose limits can be used, via identifications $\sigma_i$, to produce a section in $M^r_{{\bf k}_n}(\Gamma_{1}(\cN,p),\Lambda_n^G)$. which lies in the limit of the original Cauchy sequence.

Finally, the fact that $|\cdot |$ is potentially orthonormalizable follows from the results in \cite[\S A]{Col97}, since $L$ is discretely valued and $D_n^{{\bf k}_n^{\tau_0}}(\cO^{\tau_0},A_i)\simeq D_n^{{\bf k}_n^{\tau_0}}(\cO^{\tau_0},k^\ast(A_i))\hat\otimes_{\Z_p}\Lambda_n^G$, for any $k\in \Lambda_n^G(\cO_L)$.
\end{proof}

\begin{rmk}
We have not worried about the $\Delta$-invariance because it correspond to a unitary continuous action.
\end{rmk}

\begin{prop}\label{Upcomp}
The composition $U_p:=\prod_{\dP\mid p}U_{\dP}$ acting on $M^r_{{\bf k}_n}(\Gamma_{1}(\cN,p),\Lambda_n^G)$ corresponds to a compact operator of $\Lambda_n^G\otimes\Q$-Banach spaces.	
\end{prop}
\begin{proof} 
By the description of the norm $|\cdot|$ given in the proof of Lemma \ref{lemmanorm}, it is enough to prove compactness when restricted to the affinoid $U_i=\Spf(A_i)$ where $\Omega$ is trivialized. In this case,
\[
H^0(U_i,\Omega^{{\bf k}_n})\otimes\Q\simeq D_n^{{\bf k}_n^{\tau_0}}(\cO^{\tau_0},A_i)\otimes\Q\simeq (A_i\otimes\Q)\hat\otimes_L\hat\bigotimes_{\dP\neq\dP_0}D_n^{{\bf k}_{n,\dP}}(\cO_\dP,\cO_L\otimes\Lambda_{n,\dP}^G)\otimes\Q,
\]
where $C_n^{{\bf k}_{n,\dP}}(\cO_\dP,\cdot)$ and $D_n^{{\bf k}_{n,\dP}}(\cO_\dP,\cdot)$ are defined as in \S \ref{NotCD}.
The operator $U_p$ is composed by the operator $U_{\dP_0}$ acting on $W_{\dP_0}:=A_i\otimes\Q$, and the operators $U_{\dP}$ acting on $W_\dP:=D_n^{{\bf k}_{n,\dP}}(\cO_\dP,\cO_L\otimes\Lambda_{n,\dP}^G)$, if $\dP\neq\dP_0$,
\[
\int_{\cO_\dP^\times\times\cO_\dP}\phi\; dU_\dP\mu:=\sum_{i\in\kappa_\dP}\int_{\cO_\dP^\times\times\cO_\dP}\bar g_i\ast\phi \;d\mu;\qquad \varpi_\dP g_i^{-1}=\bar g_i=\left(\begin{array}{cc}1&i\\&\varpi_\dP\end{array}\right).
\]
It is enough to prove that each $U_\dP$ is compact in each $(\cO_L\otimes\Lambda_{n,\dP}^G)$-Banach module $W_\dP$.

By \cite[Lemma 7.5]{brasca13}, the action of $U_{\dP_0}$ factors through the restriction $\dX_{r+1}\rightarrow\dX_r$. Similarly as in \cite[Proposition 2.20]{Kassaei09}, this implies that the action of $U_{\dP_0}$ on $W_{\dP_0}$ is compact.

Notice that, for any locally analytic character $k:\cO_\dP^\times\rightarrow R$,
\[
\imath: C_n^{k}(\cO_\dP,R)\stackrel{\simeq}{\longrightarrow} C_n(\cO_\dP,R);\qquad \imath\phi(a,b)=\phi(b/a),
\]
where $C_n(\cO_\dP,R)$ is the space of functions on $\cO_\dP$ that are analytic in any ball of radius $n$.
Let us consider $C_{n,m}^{k}(\cO_\dP,R):=\imath^{-1}(C_m(\cO_\dP,R))\subseteq C_{n}^{k}(\cO_\dP,R)$, for any $m\leq n$. An orthonormal basis for $C_{n,m}^{k}(\cO_\dP,R)$ is given by 
\[
\phi_{r,f}^m(x,y)=k(x)\cdot f\left(\frac{\frac{y}{x}-r}{\varpi_\cP^m}\right)\sum_{a\in(\cO_\dP/\dP^m)^\times}1_{B_m(a,ar)}(x,y);\qquad r\in\cO_\dP/\dP^m,
\]
for some $f\in R\langle T\rangle$. We compute that $\bar g_i\ast\phi^m_{r,f}=0$, if $i\neq r\;{\rm mod}\;\dP$, and $\bar g_i\ast\phi^m_{r,f}=\phi^{m-1}_{r',f}$, if $i\equiv r$ mod $\dP$ where $r'=\varpi^{-1}(r-i)$. 
This implies that the morphism $U_\dP$ is the composition of a continuous morphism with the restriction morphism ${\rm res}:D_{n,m-1}^{{\bf k}_{n,\dP}}(\cO_\dP,\cO_L\otimes\Lambda_{n,\dP}^G)\otimes\Q\rightarrow D_{n,m}^{{\bf k}_{n,\dP}}(\cO_\dP,\cO_L\otimes\Lambda_{n,\dP}^G)\otimes\Q$, where $D_{n,m}^{{\bf k}_{n,\dP}}(\cO_\dP,\cO_L\otimes\Lambda_{n,\dP}^G):=C_{n,m}^{{\bf k}_{n,\dP}}(\cO_\dP,\cO_L\otimes\Lambda_{n,\dP}^G)^\vee$. An easy computation shows that ${\rm res}$ is compact. Thus, the action of $U_\dP$ on $W_\dP$ is compact and the result follows.
\end{proof}


\subsection{Eigenvarieties} \label{ss:eigenvarieties}
The following proposition allows the construction of the eigenvariety as stated in the introduction. 
\begin{prop}\label{p:overconvergent sheaf is projective} The $\Lambda_n^G[1/p]$-module of families of overconvergent forms $M^r_{{\bf k}_n}(\Gamma_{1}(\cN,p),\Lambda_n^G)$ is projective.
\end{prop}
\begin{proof}  For each $\cC$ the neighborhood of the ordinary locus $\mathcal{X}_r^{\cC}$ is an affinoid and we use Banach sheaves over them, hence $M^r_{{\bf k}_n}(\Gamma_{1,1}^{\mathfrak{c}}(\cN,p),\Lambda_n^G)^{\Delta}$ is projective. Thus $M^r_{{\bf k}_n}(\Gamma_{1}(\cN,p),\Lambda_n^G)$ is a finite sum of projective spaces and then it is projective too. 
\end{proof}

As $M^r_{{\bf k}_n}(\Gamma_{1}(\cN,p),\Lambda_n^G)$ is a projective $\Lambda_n^G[1/p]$-module and the Hecke operator $U_p$ acting on $M^r_{{\bf k}_n}(\Gamma_{1}(\cN,p),\Lambda_n^G)$ is compact (see proposition \ref{Upcomp}) then from \cite{AIP-halo}[\S B.2.4] there exists the Fredholm determinant of $U_p$, which is denoted by $F_{I}(X):= \mathrm{det}(1- XU_p | \ M^r_{{\bf k}_n}(\Gamma_{1}(\cN,p),\Lambda_n^G))\in \Lambda_n^G[1/p]\{\{X\}\}$.  We denote by $\mathcal{W}^G_n$ the corresponding adic weight space, and let $\mathcal{Z}_{n} \subset \mathbb{A}_{\mathcal{W}^G_n}^1$ be the spectral variety attached to $F_{n}(X)$ (see \cite{AIP-halo}[\S B.3]), then the natural map $\mathcal{Z}_{n} \rightarrow \mathcal{W}^G_n$ is locally finite and flat. By construction $\mathcal{Z}_{n}$ parametrizes reciprocals of the non-zero eigenvalues of $U_p$. In order to parametrize eigenvalues for the \emph{all} Hecke operators we construct a finite cover of $\mathcal{Z}_{n}$ as follows. 

The operator $U_p$ is compact on $M^r_{{\bf k}_n}(\Gamma_{1}(\cN,p),\Lambda_n^G)$  then using \cite{AIP-halo}[Cor. B.1, Thm. B.2] we obtain a natural coherent sheaf denoted $\mathcal{M}_n$ over $\mathcal{Z}_{n}$ endowed with an action of the Hecke operators.  The image of the Hecke operators in $\mathrm{End_{\cO_{\mathcal{Z}_{n}}}}(\mathcal{M}_n)$ generates a coherent $\cO_{\mathcal{Z}_{n}}$-algebra and the adic space attached to it is denoted $\mathcal{E}_n$, which by construction is equidimensional of dimension $d+1$ and endowed with a natural weight map $w: \mathcal{E}_n \rightarrow \mathcal{W}_n^G$ which is locally free and without torsion. This is the  \emph{adic eigenvariety} introduced in theorem \ref{t:main thm 1} (ii) of the introduction. 

\subsection{Classicity}\label{ss:classicity}
Let $\dX^\cC(\dP_0)$ be the Shimura curve obtained adding $\Gamma_0(\dP_0)$-level structure to the curve $\dX^\cC$ introduced in \ref{rmkonGammac}. Recall that by the definition of $\dX^\cC$ the curve $\dX^\cC(\dP_0)$ has the full $\Gamma_0(p)$-level structure. Observe that its irreducible components are in bijection with $\mathrm{Pic}(E/F, \mathfrak{n})$ and let $\dX^\cC(\dP_0)^0$ be the irreducible component corresponding to $1$. If $(\underline{k},\nu)\in \dW^G_n(\bar\Q_p)$ is a classical weight i.e. $\underline{k}\in \Xi_\nu$, then the space of \emph{classical forms of weight  $(\underline{k},\nu)$ and level $\Gamma_{1}(\cN,p)$} is:
 \[
M_{(\underline{k},\nu)}(\Gamma_{1}(\cN,p),\bar\Q_p)=\bigoplus_{\cC\in\Pic(\cO_F)}H^0(\cX^\cC(\dP_0)^0,\omega^{\underline{k}})^\Delta,
\]
where $\omega^{\underline{k}}$ is defined in the same way as in \S \ref{ss:modular sheaves}. Observe that by corollary \ref{c:automorphic forms as sections} elements of $M_{(\underline{k},\nu)}(\Gamma_{1}(\cN,p),\bar\Q_p)$ which are defined over number fields determine automorphic forms for $G$.  

The space of \emph{locally analytic overconvergent automorphic forms for $G$ with depth of overconvergence $r$} which we denote by $M^r_{(\underline{k},\nu)}(\Gamma_{1}(\cN,p),\bar\Q_p)$ is defined as the specialization  of the $\Lambda_n^G[1/p]$-module $M^r_{{\bf k}_n}(\Gamma_{1}(\cN,p),\Lambda_n^G)$ at $(\underline{k},\nu)$.  Observe that a form $\mu \in M^r_{(\underline{k},\nu)}(\Gamma_{1}(\cN,p),\bar\Q_p)$ is \emph{overconvergent} in the sense that it is defined in some neighborhood of the ordinary locus and \emph{locally analytic} in the sense that it is a section of a Banach sheaf. Remark the contrast with the coherent modular sheaves used to define classical modular forms.

We also define the space of  \emph{overconvergent automorphic forms for $G$ with depth of overconvergence $r$} as:
$$M^{r, \mathrm{alg}}_{(\underline{k},\nu)}(\Gamma_{1}(\cN,p),\bar\Q_p)=\bigoplus_{\cC\in\Pic(\cO_F)}H^0(\cX^{\cC, 0}_r,\omega^{\underline{k}})^\Delta,$$

Using the existence of the canonical subgroup in $\dX_r$ we obtain a canonical section of the natural morphism $\dX(\dP_0)^{\cC} \rightarrow \dX^{\cC}$ over $\mathfrak{X}_r$, this implies the following diagram: 
$$
\xymatrix{
M^r_{(\underline{k},\nu)}(\Gamma_{1}(\cN,p),\bar\Q_p)\ar^{\alpha}[d] &M_{(\underline{k},\nu)}(\Gamma_{1}(\cN,p),\bar\Q_p)\ar@{_{(}->}^{\beta}[ld]\\
M^{r, \mathrm{alg}}_{(\underline{k},\nu)}(\Gamma_{1}(\cN,p),\bar\Q_p)
}
$$
In the next result if we work with respect to the morphism $\beta$ we consider the Hecke operator $U_{\mathfrak{p}_0}$ and moreover, for $h \in \R_{\geq 0}$ the superscript $h$ means the generalized eigenspace for $U_{\mathfrak{p}_0}$ of slope less than $h$. In the same way if we work with respect to the morphism $\alpha$ we consider the Hecke operators $U_{\mathfrak{p}}$ for $\mathfrak{p}\neq\mathfrak{p}_0$ and in this case we use superscripts $\underline{h}^0= (h_{\mathfrak{p}})_{\mathfrak{p}\neq\mathfrak{p}_0}\in\R_{\geq 0}^{\{\mathfrak{p}\neq\mathfrak{p}_0\}}$

\begin{prop}\label{propclassty} Let $(h_0, \underline{h}^0)= (h_0, h_{\mathfrak{p}})_{\mathfrak{p}\neq\mathfrak{p}_0}\in \R_{\geq 0}\times\R_{\geq 0}^{\{\mathfrak{p}\neq\mathfrak{p}_0\}}$.
\begin{itemize}
\item[(i)] If for each $\mathfrak{p}\neq \mathfrak{p}_0$ we have $h_{\mathfrak{p}} < \mathrm{min}\{k_{\tau}+1: \tau \in \Sigma_{\mathfrak{p}}\}$ then $\alpha$ induces an isomorphism:
$$\xymatrix{M^r_{(\underline{k},\nu)}(\Gamma_{1}(\cN,p),\bar\Q_p)^{< \underline{h}^{0}}\ar^{\sim}[r]&  M^{r, \mathrm{alg}}_{(\underline{k},\nu)}(\Gamma_{1}(\cN,p),\bar\Q_p)^{< \underline{h}^{0}}}$$
\item[(ii)] If $h_{0} < k_{\tau_0}-1$ then $\beta$ induces an isomorphism: 
$$\xymatrix{M_{(\underline{k},\nu)}(\Gamma_{1}(\cN,p),\bar\Q_p)^{< h_{0}}\ar^{\sim}[r]&  M^{r, \mathrm{alg}}_{(\underline{k},\nu)}(\Gamma_{1}(\cN,p),\bar\Q_p)^{< h_{0}}}.$$
\end{itemize}
\end{prop}
\begin{proof} The first statement is a consequence of the locally analytic BGG-resolution and our the description in \S \ref{ss:overconvergent modular forms a la katz}. More precisely, we have seen in the proof of Proposition \ref{Upcomp} that $U_\dP$ factors through ${\rm res}:D_{n-1}^{\underline{k}_{\dP}}(\cO_\dP,\bar\Z_p)\otimes\Q_p\rightarrow D_{n}^{\underline{k}_{\dP}}(\cO_\dP,\bar\Z_p)\otimes\Q_p$, indeed since $\underline{k}_\dP$ is analytic $C_{n,m}^{\underline{k}_{\dP}}(\cO_\dP,\bar\Q_p)=C_{m}^{\underline{k}_{\dP}}(\cO_\dP,\bar\Q_p)$. Repeating this for each $\mathfrak{p}\neq \mathfrak{p}_0$ it is enough to consider analytic distributions instead of locally analytic ones. 

We consider the three last terms of the BGG-resolution (see for example \cite{Jon} or \cite{Urb11}), which form a exact sequence:
\[
\bigoplus_{\sigma \in \Sigma_{F}-\{\tau_0\}} D_{0}^{\underline{k}^{\tau_0}}(\cO_\dP,\bar\Z_p)\otimes\Q_p \stackrel{(\Theta_\sigma^{\vee})_{\sigma}}\longrightarrow \bigoplus_{\sigma \in \Sigma_{F}-\{\tau_0\}} D_{0}^{\underline{k}^{\tau_0}}(\cO_\dP,\bar\Z_p)\otimes\Q_p \longrightarrow {\rm Sym}^{\underline{k}^{\tau_0}}(\bar\Q_p^2) \longrightarrow 0
\]
here $\Theta_\sigma^{\vee}$ is obtained from the morphism on the locally analytic functions given by applying $\frac{d^{k_{\sigma}+ 1}}{dy_{\sigma}^{k_{\sigma}+ 1}}$ i.e. taking $k_{\sigma}+ 1$-derivatives in the variable $y_{\sigma}$ of the second coordinate. Remark that we have $\varpi_\dP^{k_\sigma+1}\Theta_\sigma^{\vee}\circ U_\dP=U_\dP\circ\Theta_\sigma^{\vee}$ if $\sigma \in \Sigma_{\dP}$ and $\Theta_\sigma^{\vee}\circ U_\dP=U_\dP\circ\Theta_\sigma^{\vee}$ if not. Thus by the hypothesis on $\underline{h}^0$ we obtain an isomorphism
\[
D_{0}^{\underline{k}^{\tau_0}}(\cO^{\tau_0},\bar\Z_p)\otimes\Q_p^{<\underline{h}^0}\simeq {\rm Sym}^{\underline{k}^{\tau_0}}(\bar\Q_p^2)^{<\underline{h}^0}.
\]
This directly implies the first claim by the Katz modular form interpretation.
 This argument is used in an analogous way in  \cite[\S 7.2 and \S 7.3]{AIP-siegel}.

 The second statement below to a general type of results generalizing the geometric approach to classicity introduced by Kasseai in \cite{Kassaei06}. Our result is a consequence for example of the main results of \cite{BPS} . 
\end{proof}


\part{$p$-adic $L$-functions}

We construct ordinary triple product $p$-adic $L$-functions. For this purpose we introduce the space of (classical and dual) $p$-adic modular forms for $G'$. We also introduce the $q$-expansion of a $p$-adic modular by using Serre-Tate coordinates.

\section{$p$-adic modular forms}

\subsection{Definitions}

Let $\cC\in\Pic(\cO_F)$ and denote by $\dX_{\rm ord}^{\cC}$ the \emph{ordinary locus} of $\dX^{\cC}$, namely, the (dense) open subscheme of $\dX^{\cC}$ obtained by removing the points specializing to supersingular points. Since it can be also defined as the open formal subscheme where ${\rm Hdg}$ is invertible then $\dX_{\rm ord}^{\cC}\subset \dX_r^{\cC}$ for each $r>0$. Note that $\dX_{\rm ord}^{\cC}$ classifies quadruples $(A,\iota,\theta,\alpha)$ where $A$ has ordinary reduction. This implies that the $p$-divisible group $\cG_0:={\bf A}[\dP_0^\infty]^{-,1}\rightarrow\dX_{\rm ord}^{\cC}$, attached to the universal abelian variety over $\dX_{\rm ord}^{\cC}$, lies in the exact sequence:
\[
0\longrightarrow \G_m[p^\infty]\stackrel{}{\longrightarrow} \cG_0\longrightarrow \Q_p/\Z_p \longrightarrow 0.
\]
Since ${\rm Hdg}$ is invertible, we have that  $\Omega_0\mid_{\dX_{\rm ord}^{\cC}}=\omega_0\mid_{\dX_{\rm ord}^{\cC}}=\omega_{\cG_0^D}\mid_{\dX_{\rm ord}^{\cC}}$ and then $\Omega\mid_{\dX_{\rm ord}^{\cC}}=\omega\mid_{\dX_{\rm ord}^{\cC}}$. Thus, it does not cause any confusion to put $\omega^{{\bf k}_n}:=\Omega^{{\bf k}_n}\mid_{\dX^{\cC}_{\rm ord}\times\dW_n}$.
\begin{defi}
The space of families of \emph{$p$-adic modular forms} is:
\[
M^{p-\rm{adic} }_{{\bf k}_n}(\Gamma_1(\cN,p),\Lambda_n^G):=\bigoplus_{\cC\in\Pic(\cO_F)}{\rm H}^0(\dX^{\cC,0}_{\rm ord}\times\dW^G_n,\omega^{{\bf k}_n})^\Delta\otimes\Q.
\]
\end{defi}

\subsection{$q$-expansions and Serre-Tate coordinates}\label{ss:q-expansions and serre-tate}

Let $W$ be the Witt vectors of $\bar{\mathbb{F}}_p$ and let us consider $\dX$ over $W$. Let $\cX_{\rm ord}$ be the adic generic fiber of $\dX_{\rm ord}$. Thus, 
the universal abelian variety ${\bf A}/\cX_{\rm ord}$ satisfies that the $p$-divisible group $\cG={\rm A}[p^\infty]^{-,1}$ lies in an exact sequence:
\[
0\longrightarrow \G_m[p^\infty]\stackrel{\iota}{\longrightarrow} \cG\longrightarrow \Q_p/\Z_p\times\prod_{\dP\neq \dP_0}(F_{\dP}/\cO_\dP)^2\longrightarrow 0.
\]
The image of $\iota$ is the canonical subgroup of $\cG$. 
As above, we have $\omega_{\cG^D}=\omega_0\oplus\bigoplus_{\tau\in\Sigma_\dP,\dP\neq\dP_0} \omega_{\dP,\tau},$
over $W$ since $F$ is unramified at $p$.

\begin{defi}
Let $\cX(n)$ be the adic space representing the functor classifying $\cO_F/p^n\cO_F$-equivariant frames over $\cX_{\rm ord}$ i.e. morphisms $\imath:\G_m[p^n]\rightarrow \cG[p^n]$ and $\pi:\cG[p^n]\rightarrow p^{-n}\Z_p/\Z_p\times\prod_{\dP\neq \dP_0}(p^{-n}\cO_\dP/\cO_\dP)^2$
such that the following sequence is exact:
\[
0 \longrightarrow\G_m[p^n] \stackrel{\imath}{\longrightarrow} \cG[p^n] \stackrel{\pi}{\longrightarrow}p^{-n}\Z_p/\Z_p\times\prod_{\dP\neq \dP_0}(p^{-n}\cO_\dP/\cO_\dP)^2\longrightarrow 0,
\]
and $(p^{-1},0)\in (p^{-n}\cO_\dP/\cO_\dP)^2$ generates $C_\dP$.
\end{defi} 
The space $\cX(n)$ is finite \'etale over $\cX_{\rm ord}$. Let $\dX(n)\rightarrow\dX_{\rm ord}$ be the normalization of $\cX(n)$ in $\dX_{\rm ord}$ and consider the projective limit of $\mathrm{Spf}(W)$-schemes $\dX(\infty)=\varprojlim_n\dX(n)$ which is an affine formal scheme over $W$. For each $n \in \N$ we introduce the following module (whose elements are called \emph{Universal convergent modular forms} in \cite{liu-zhang-zhang17}):
$$\mathcal{M}(\infty,\Lambda_n):=H^0(\dX(\infty),\cO_{\dX(\infty)})\hat\otimes\Lambda_n$$
Moreover, we put $D^{{\bf k}_n^{\tau_0}}_n(\cO^{\tau_0},\mathcal{M}(\infty,\Lambda_n)):= {\rm Hom}_{\Lambda_n}(C^{{\bf k}_n^{\tau_0}}_n(\cO^{\tau_0},\Lambda_n),\mathcal{M}(\infty,\Lambda_n))$.

 Over $\dX(\infty)$ we have a universal abelian variety with extra structures $({\bf A},\iota,\theta,\alpha^{\dP_0})$ and we keep denoting $\cG={\bf A}[p^\infty]^{-,1}$ to the $p$-divisible group attached to it which is equipped with a universal frame: 
\[
0 \longrightarrow\G_m[p^\infty] \stackrel{\imath}{\longrightarrow} \cG \stackrel{\pi}{\longrightarrow}\Q_p/\Z_p\times\prod_{\dP\neq\dP_0}(F_{\dP}/\cO_\dP)^2\longrightarrow 0,
\]
We also have a canonical basis ${\bf f}_0:={\rm dlog}_0(1)$ and ${\bf f}_{\tau}={\rm dlog}_{\mathfrak{p}, \tau}(1,0)$, ${\bf e}_{\tau}={\rm dlog}_{\mathfrak{p}, \tau}(0,1)$ for each $\mathfrak{p}\neq \mathfrak{p}_0$ and $\tau \in \Sigma_{\mathfrak{p}}$. Observe that $\varphi_\tau({\bf f}_\tau\wedge{\bf e}_\tau)=1$ and $f_{\tau}$ generate $\omega_{C_\dP}$. Finally we put ${\bf w}=({\bf f}_0,\{({\bf f}_\tau,{\bf e}_\tau)\}_{\tau\neq \tau_0})$. 


\begin{defi}
Using the notations introduced in \ref{ss:overconvergent modular forms a la katz}, for each $\mu \in H^0(\dX_{\rm ord}\times\dW_n, \omega^{{\bf k}_{n}})$ we call the \emph{$q$-expansion} of $\mu$ to the following distribution:
\[
\mu(q):=\mu({\bf A},\iota,\theta,\alpha,{\bf w})\in D^{{\bf k}_n^{\tau_0}}_n(\cO^{\tau_0},\mathcal{M}(\infty,\Lambda_n)).
\]
\end{defi}
Repeating the construction for the curves $\dX_{\rm ord}^{\cC, 0}$ we obtain a \emph{$q$-expansion morphism}: 
\begin{equation}\label{e:q-expansion morphism}
M^{p-\rm{adic} }_{{\bf k}_n}(\Gamma_1(\cN),\Lambda_n^G) \rightarrow  D^{{\bf k}_n^{\tau_0}}_n(\cO^{\tau_0},\mathcal{M}(\infty,\Lambda_n))^{\rm{Pic}(\cO_F)}
\end{equation}

Now we are going to introduce Serre-Tate coordinates and introduce fundamental notions which will be crucial to construct $p$-adic $L$-functions.  In the same way as in \cite{liu-zhang-zhang17}, 
we obtain a classifying morphism: 
$$c: \dX(\infty)\longrightarrow \G_m.$$
 This morphism can be described as follows: Firstly observe that $\G_m/W$ classifies the extensions of $\G_m[p^\infty]$ by $\Q_p/\Z_p$. In fact, let $R$ be an Artinian ring with maximal ideal $m_R$ and suppose  $m_R^{m+1}=0$ for some $m> 0$. For $s\in \G_m(R)=m_R$, we can construct the extension $E_s/R$:
\begin{equation}\label{descGs}
E_s:=(\G_{m,R}[p^{\infty}]\oplus\Q_p)\slash\langle((1+s)^z,-z),\;z\in\Z_p\rangle,\qquad 0\longrightarrow\G_{m,R}[p^\infty]\stackrel{\imath}{\longrightarrow} E_s\stackrel{\pi_s}{\longrightarrow}\Q_p/\Z_p\longrightarrow 0,
\end{equation}
where $\imath (a)=(a,0)$ and $\pi_s(a,b)=b$. Indeed, $\ker(\pi_s)=\{(a,z)\in E_s;\;z\in\Z_p\}=\G_{m,R}[p^\infty]$ because $(a,z)=(a(1+s)^z,0)$. 
Reciprocally, given such an exact sequence over $R$,
\[
0\longrightarrow\\G_{m,R}[p^\infty]\simeq 1+m_R\stackrel{\imath}{\longrightarrow} E\stackrel{\pi_E}{\longrightarrow}\Q_p/\Z_p\longrightarrow 0,
\]
we can define the corresponding point $s=p^m\pi_E^{-1}\left(\frac{1}{p^m}\right)-1\in m_R=\G_m(R)$, and this definition does not depend on $m$.

Hence, given $x\in \dX(\infty)$ corresponding to a tuple $(A,\iota,\theta, \alpha^{\dP_0})$, a morphism $\pi_\dP:\cG_\dP[p^\infty]\stackrel{\simeq}{\longrightarrow}(F_{\dP}/\cO_\dP)^2$ for each $\dP\neq \dP_0$, and an exact sequence
\[
S_0:0 \longrightarrow\G_m[p^\infty] \stackrel{\imath}{\longrightarrow} \cG_0[p^\infty] \stackrel{\pi_0}{\longrightarrow}\Q_p/\Z_p\longrightarrow 0,
\]
where as above we use the notations $A[p^\infty]^{-,1}= \cG= \cG_0[p^\infty]\oplus\bigoplus_{\dP\neq\dP_0}\cG_\dP[p^\infty]$, the morphism $c$ maps $x$ to the point classifying the extension $S_0$.
The local coordinates given by the morphism $c$ are called \emph{Serre-Tate coordinates}.

\begin{prop}\label{compXGm} There exists a morphism $\beta:\dX(\infty)\times_{\mathrm{Spf}(W)}\G_m\longrightarrow \dX(\infty)$ such that the following diagram is commutative:
\[
\xymatrix{
\dX(\infty)\times_{\mathrm{Spf}(W)}\G_m\ar[d]_{(c\times{\rm id})}\ar[r]^{\beta}&\dX(\infty)\ar[d]_{c}\\
\G_m\times_{\mathrm{Spf}(W)}\G_m\ar[r]&\G_m,
}
\]
where the bottom arrow is the formal group law. Moreover, for every closed point $x\in \dX(\infty)(\bar\kappa)$, the morphism $c$ induces an isomorphism $c\mid_x:\dX(\infty)_x\rightarrow\G_m$ over $W$, where $\dX(\infty)_x$ denotes the formal completion of $\dX(\infty)$ at $x$.  
\end{prop}
\begin{proof}
The proof can be found essentially in \cite[Prop. 2.3.4, Prop. 2.3.5]{liu-zhang-zhang17}. Notice that, since $F_{\dP_0}=\Q_p$, the Lubin-Tate $p$-divisible group is just $\G_m$ in this case.
\end{proof}

For the following definition see \cite[Definition 2.3.10]{liu-zhang-zhang17}.
\begin{defi}
A function $f\in\mathcal{M}(\infty,\Lambda_n)$ is \emph{stable} if $\sum_{i=0}^{p-1}f(\beta(x,\xi_p^i))=0$ for each $\xi_p\in\G_m[p]$.
\end{defi}

We write $\mathcal{M}(\infty,\Lambda_n)^\heartsuit\subseteq \mathcal{M}(\infty,\Lambda_n)$ for the subset of stable functions and we put:  $$D^{{\bf k}_n^{\tau_0}}_n\left(\cO^{\tau_0},\mathcal{M}(\infty,\Lambda_n)^\heartsuit\right):={\rm Hom}_{\Lambda_n}\left(C^{{\bf k}^{\tau_0}_n}_n(\cO^{\tau_0},\Lambda_n),\mathcal{M}(\infty,\Lambda_n)^\heartsuit\right).$$
We write $H^0(\dX_{\rm ord}\times\dW_n, \omega^{{\bf k}_{n}})^\heartsuit$ for the subset of elements $\mu\in H^0(\dX_{\rm ord}\times\dW_n, \omega^{{\bf k}_{n}})$ such that $\mu(q)\in D^{{\bf k}_n^{\tau_0}}_n\left(\cO^{\tau_0},\mathcal{M}(\infty,\Lambda_n)^\heartsuit\right)$. 

Remark that we can repeat the same constructions using $\cX_{\rm ord}^{\cC}$ instead of $\cX_{\rm ord}$. Then we define $M^{p-\rm{adic}}_{{\bf k}_n}(\Gamma_1(\cN,p),\Lambda_n^G)^\heartsuit$ as the module of those $p$-adic modular forms $\mu$ whose sections $\mu_\cC$ lie in $H^0(\dX_{\rm ord}^{\cC,0}\times\dW_n, \omega^{{\bf k}_{n}})^\heartsuit\otimes_{\Z_p}\Q_p$ for each $\cC \in \mathrm{Pic}(\cO_F).$

\subsection{Action of $U_{\dP_0}$} The previous description of the Serre-Tate coordinates allows us to compute the $U_{\dP_0}$-operator in coordinates and to prove the following:
\begin{lemma}\label{Up-0} Let $\mu \in M^{p-\rm{adic}}_{{\bf k}_n}(\Gamma_1(\cN,p),\Lambda_n^G)$. Then $\mu\in M^{p-\rm{adic}}_{{\bf k}_n}(\Gamma_1(\cN,p),\Lambda_n^G)^\heartsuit$ if and only if $U_{\dP_0}\mu=0$.
\end{lemma}
\begin{proof}
 The description of $E_s$ given in equation \eqref{descGs} implies that, over a big enough extension of $R$,
\[
E_s[p]=\left\{\left(\xi_p^i(1+s)^{\frac{j}{p}},-\frac{j}{p}\right)_{i,j=0,\cdots,p-1}\right\},
\]
where $\xi_p$ is a primitive $p$-root of unity.
The subgroup $\mu_p=\{(\xi_p^j,0)_{j=0,\cdots,p-1}\}$ is the canonical subgroup. Thus, the subgroups that intersect trivially with $\mu_p$ are precisely 
\[
C_i=\left\{\left(\xi_p^{ij}(1+s)^{\frac{j}{p}},-\frac{j}{p}\right)_{j=0,\cdots,p-1}\right\},\qquad i=0,\cdots,p-1.
\]
We compute that $\pi_i: E_s/C_i\longrightarrow\Q_p/\Z_p$ is given by $\pi_i(a,b)=pb\mod\Z_p$, where
\[
E_s/C_i\simeq ((1+m_R)\oplus\Q_p)\slash\left\langle\left(\xi_p^{in}(1+s)^{\frac{n}{p}},-\frac{n}{p}\right),\;n\in\Z_p\right\rangle.
\]
Thus, $E_s/C_i$ corresponds to the point 
\[
\imath^{-1}\left(p^m\pi_i^{-1}\left(\frac{1}{p^m}\right)\right)-1=\imath^{-1}(p^m(1,p^{-m-1}))-1=\imath^{-1}(1,p^{-1})-1=(1+s)^{\frac{1}{p}}\xi_p^i-1.
\]
Hence $U_{\dP_0}$ acts on $\cO_{\G_m}=W[[q]]$ as follows
\begin{equation}\label{q-expU}
U_{\dP_0}f(q)=\frac{1}{q_{\dP_0}}\sum_{i\in\cO_F/\dP_0}f(\xi_p^i(1+q)^{1/p}-1)=\frac{1}{q_{\dP_0}}\sum_{i\in\cO_F/\dP_0}f(\hat\G_m((1+q)^{1/p}-1,\xi_p^i));\qquad f\in W[[q]],
\end{equation}
where $\hat\G_m$ is the formal group law of the multiplicative group. 
The result follows directly from the above computation together with Proposition \ref{compXGm}.
\end{proof}

\begin{rmk}\label{remactUp}
Notice that $W[[q]]$ is topologically generated by $f_n:=(1+q)^n$, with $n\in\Z_p$. The above computation show that $U_{\dP_0}f_n=0$, if $p\nmid n$, and $U_{\dP_0}f_n=f_{n/p}$, if $p\mid n$.
\end{rmk}

\begin{rmk}
In the same way we can describe the action of the operator $V_{\dP_0}$ in coordinates. In fact, remark that the quotient by the canonical subgroup is given by: 
\[
G_s/\mu_p\simeq ((1+m_R)\oplus\Q_p)\slash\left\langle\left((1+s)^{np},-n\right),\;n\in\Z_p\right\rangle.
\]
Thus, $G_s/\mu_p$ corresponds to the point 
\[
\imath^{-1}\left(p^m\pi^{-1}\left(\frac{1}{p^m}\right)\right)-1=\imath^{-1}(1,1)-1=(1+s)^p-1.
\]
Hence $V_{\dP_0}$ acts on $ f\in \cO_{\G_m}=W[[q]]$ by $V_{\dP_0}f(q)=f((1+q)^{p}-1).$ Thus, $V_{\dP_0}f_n=f_{np}$, where the $f_n$ are defined as in the above remark. One obtain the classical relation $U_{\dP_0}\circ V_{\dP_0}={\rm Id}$.
\end{rmk}

\section{Connections and Unit root splittings}\label{s:connections and unit root}
\subsection{Unit root splitting}\label{ss:unit root splitting}

Let $\cH^1:=\cH^1_{\cG^D}$ denote the contravariant Dieudonn\'e module attached to $\cG^D\rightarrow\dX_{\rm ord}$, and remark that we have a decomposition $\cH^1=\cH^1_{0}\oplus\bigoplus_{\tau\in\Sigma_\dP,\dP\neq\dP_0}\omega_{\dP,\tau}$.
Hence the Hodge filtration
\[
0\longrightarrow \omega\longrightarrow \cH^1 \longrightarrow \omega_{\cG}^\vee\longrightarrow 0,
\]
restricts to the exact sequence
\begin{equation}\label{Hdg(1)}
0\longrightarrow \omega_{0} \longrightarrow \cH^1_{0} \stackrel{\epsilon}{\longrightarrow} \omega_{\cG}^\vee\longrightarrow 0.
\end{equation}
Moreover, we have fixed an isomorphism $\omega_\cG=\omega_{\cG_0}\simeq\omega_{0}$, as expained in \S \ref{alt-pair}. Such isomorphism together with the Gauss-Manin connection
\[
\bigtriangledown:\cH_0^1\longrightarrow\cH_0^1\otimes\Omega^1_{\dX_{\rm ord}},
\]
provides the Kodaira-Spencer isomorphism
\[
KS:\omega_{0}\hookrightarrow\cH^1_{0}\stackrel{\bigtriangledown}{\longrightarrow}\cH^1_{0}\otimes\Omega^1_{\dX_{\rm ord}}\stackrel{\epsilon}{\longrightarrow}\omega_{\cG_0}^\vee\otimes\Omega^1_{\dX_{\rm ord}},\qquad KS:\Omega^1_{\dX_{\rm ord}}\simeq \omega_{0}\otimes\omega_{\cG_0}\simeq \omega_{0}^{\otimes 2}.
\]
\begin{lemma}\cite[Lemma 2.3.1]{liu-zhang-zhang17}
There exists a unique morphism $\Phi:\dX_{\rm ord}\longrightarrow\dX_{\rm ord}$ lifting the Frobenius morphism $x\mapsto x^p$ on the special fiber such that $\Phi^\ast\cG\simeq \cG/\cG^0[p]$, where $\cG^0$ is the formal part of $\cG$. In particular, $\Phi$ induces an endomorphism $\Phi^\ast$ on $\cH_{\cG}^1$. Moreover, there exists a unique $\Phi^\ast$-stable splitting:
\[
\cH^1_{\cG}=\omega_{0}\oplus \mathcal{L},
\]
where $\mathcal{L}$ is an invertible quasi-coherent formal sheaf over $\dX_{\rm ord}$. In addition, $\mathcal{L}$ is horizontal with respect to the Gauss-Manin connection, that is, $\bigtriangledown\mathcal{L}\subset\mathcal{L}\otimes\Omega^1_{\dX_{\rm ord}}$.
\end{lemma}

\begin{rmk}
Since $\cG^0\simeq\G_m$, we have a decomposition $\mathcal{L}\simeq\mathcal{L}_{0}\oplus\bigoplus_{\tau\in\Sigma_\dP,\dP\neq\dP_0}\omega_{\dP,\tau}$ and $\mathcal{L}_{0}$ defines a Unit Root splitting $\cH_0^1=\omega_0\oplus\mathcal{L}_{0}$. We denote by $\pi_{\mathcal{L}_{0}}:\cH_0^1\longrightarrow\omega_0$ the natural projection obtained from such splitting. 
\end{rmk}

\subsection{The overconvergent projection}\label{ss:the overconvergent projection} As in \S\ref{ss:connections and trilinear products}, for $\underline{k}$ and $m \in \Z$ we consider the sheaves:
\[
\omega^{\underline{k}}\subseteq\cH^{\underline{k}}_m:=\omega_{0}^{k_{\tau_0}-m\tau_0}\otimes{\rm Sym}^{m}\cH_{0}^1\otimes\bigotimes_{\tau \in \Sigma_{\mathfrak{p}}, \mathfrak{p}\neq \mathfrak{p}_0}{\rm Sym}^{k_\tau}\omega_{\mathfrak{p}, \tau}.
\]
The elements of of $H^0(\cX_r,\cH^{\underline{k}}_m)$ are called \emph{nearly overconvergent modular forms}. Using the Unit Root Splitting, we obtain a morphism:
\[
\gamma_{\underline{k}, m}: H^0(\cX_r,\cH^{\underline{k}}_m)\longrightarrow H^0(\cX_{\rm ord},\cH^{\underline{k}}_m)\longrightarrow H^0(\cX_{\rm ord},\omega^{\underline{k}}).
\]
\begin{prop}\label{p:nearly as p-adic}
The map $\gamma_{\underline{k}, m}$ is injective.
\end{prop}
\begin{proof}
The proof is completely analogous to \cite[Proposition 3.2.4]{Urb2} using \cite[Proposition 3.1.3]{Urb2}.
\end{proof}
As in \S\ref{ss:connections and trilinear products} we have the morphism  $\epsilon:\cH_{m}^{\underline{k}}\rightarrow\cH_{m-1}^{\underline{k}-2\tau_0}$, the Gauss-Manin connection induces morphisms $\bigtriangledown_{\underline{k},m}:\cH^{\underline{k}}_m\rightarrow\cH^{\underline{k}+2\tau_0}_{m+1}$  and for $j \in \N$ we put $\bigtriangledown_{\underline{k}}^j:=\bigtriangledown_{\underline{k}+2(j-1)\tau_0, j-1}\circ\cdots\circ\bigtriangledown_{\underline{k},0}$. The following result is analogous to \cite[Lemma 3.3.4]{Urb2}:
\begin{lemma}\label{lemmaNOC}
For each $f\in H^0(\cX_r,\cH^{\underline{k}}_m)$ where $2m<k_{\tau_0}$ there exist $g_j\in H^0(\cX_r,\omega^{\underline{k}-2j\tau_0})$ for $j= 0, ..., m$ such that we can write in a unique way:
\[
f=g_0+\bigtriangledown_{\underline{k}-2\tau_0}(g_1)+\bigtriangledown_{\underline{k}-4\tau_0}^2(g_2)+\cdots+\bigtriangledown_{\underline{k}-2m\tau_0}^m(g_m).
\] 

\end{lemma}
\begin{proof}
We prove this result by induction on $m$. If $m=0$ the result is clear, then we suppose that $m> 0$. Remark that $g_m:=c^{-1}\epsilon^m f\in H^0(\cX_r,\omega^{\underline{k}-2m\tau_0})$ where $c= \frac{m!(k_{\tau_0}- m-1)!}{(k_{\tau_0}-2m-1)!}$. From \eqref{eqepsiGM} we deduce that $\epsilon^m \bigtriangledown_{\underline{k}-2m\tau_0}^m(g_m)=c\cdot g_m= \epsilon^m f$. Then we obtain $\epsilon^m\left(f-\bigtriangledown_{\underline{k}-2m\tau_0}^mg_m\right)=0.$ Thus we deduce $f-\bigtriangledown_{\underline{k}-2m\tau_0}^mg_m\in H^0(\cX_r,\cH^{\underline{k}}_{m-1})$, and the result follows from a simple induction.
\end{proof}
\begin{defi}
Given  $f\in H^0(\cX_r,\cH^{\underline{k}}_m)$ where $2m<k_{\tau_0}$, we put $\cH^r(f):=g_0$ which is called the \emph{overconvergent projection} of $f$.
\end{defi}

The following result will be proved at the end of the \S:
\begin{lemma}\label{l:ordinary projector and overconvergent projection}
Let $e_{\rm ord}=\varprojlim_n U_p^{n!}$ be the standard ordinary operator, let
 $f\in H^0(\cX_r,\cH^{\underline{k}}_m)$ with $2m<k_{\tau_0}$ then we have  $e_{\rm ord}(\gamma_{\underline{k}, m}(f))=e_{\rm ord}(\cH^r(f)).$
\end{lemma}

\subsection{Connections on Formal Vector Bundles}\label{ss:connections od formal vector bundles}
Let us consider the formal vector bundle $f:\V(\cH_0^1)\rightarrow \dX_{\rm ord}$.
By \cite[\S 2.4]{AI17} the Gauss Manin connection extends to a connection:
\[
\bigtriangledown:f_{\ast}\cO_{\V(\cH_0^1)}\longrightarrow f_{\ast}\cO_{\V(\cH_0^1)}\hat\otimes \Omega^1_{\dX_{\rm ord}}. 
\]
Using the Unit Root Splitting we have that $\V(\cH_0^1)\simeq\V(\omega_0)\times\V(\mathcal{L}_0)$, hence we can define the \emph{Serre operator}
\[
\Theta:f_{\ast}\cO_{\V(\omega_0)}\hookrightarrow f_{\ast}\cO_{\V(\cH_0^1)}\stackrel{\bigtriangledown}{\longrightarrow}f_{\ast}\cO_{\V(\cH_0^1)}\hat\otimes \Omega^1_{\dX_{\rm ord}}\longrightarrow f_{\ast}\cO_{\V(\omega_0)}\hat\otimes \Omega^1_{\dX_{\rm ord}}
\]
\begin{lemma}\label{Serreopext}
The Serre operator can be extended to $(g_n\circ f_0)_\ast\cO_{\V_0(\omega_0,s_0)}$, namely, there exist a morphism (also denoted by $\Theta$) making the following diagram commutative
\[
\xymatrix{
f_{\ast}\cO_{\V(\omega_0)}\ar[d]\ar[r]^\Theta&f_{\ast}\cO_{\V(\omega_0)}\hat\otimes \Omega^1_{\dX_{\rm ord}}\ar[d]\\
(g_n\circ f_0)_\ast\cO_{\V_0(\omega_0, s_0)}\ar[r]^\Theta &(g_n\circ f_0)_\ast\cO_{\V_0(\omega_0, s_0)}\hat\otimes \Omega^1_{\dX_{\rm ord}}
}
\]
\end{lemma}
\begin{proof}
Let us check this locally. Let $\mathfrak{IG}_n$ be the formal scheme over $\dX_{\rm ord}$ defined analogously as $\mathfrak{IG}_{r,n}$.
Let $\rho^\ast:S=\Spf(R)\longrightarrow\mathfrak{IG}_n$ be a neighborhood such that $\cH_0^1\mid_S= Rf_0\oplus Re_0$, with $f_{0}\in\omega_0$ congruent to $s_0$ modulo $p^n$, and $e_0\in\mathcal{L}_0$. Let $D$ ve a derivation dual to a generator of $\Omega_{\dX_{\rm ord}}^1$. Write 
\[
\bigtriangledown(D) f_0=a_0 f_0+b_0 e_0;\qquad \bigtriangledown(D) e_0=c_0 f_0+d_0 e_0.
\] 
The key point is that, since $s_0\in{\rm dlog}_0(T_p(\cG_0))$, we have that $\bigtriangledown(D)s_0\in \overline{\mathcal{L}_0}$. Indeed, if we denote by $\sigma$ the Frobenius morphism acting on $W$, denote by $\Sigma:(\cG_0^D)^\sigma=\cG_0^D\times_{W,\sigma}W\rightarrow\cG_0^D$ the natural projection, denote by $V:\cG_0^\sigma\rightarrow\cG_0$ the Verschiebung morphism, and we write $\varphi=D(V^D)\circ D(\Sigma)$ acting on the Dieudonne module $\cH^1_{\cG_0}=\cH^1_0$, by \cite[Lemma B.3.5]{liu-zhang-zhang17} the image of ${\rm dlog}_0:T_p(\cG_0)\rightarrow \omega_0\subset \cH_0^1$ can be identified with the elements $\xi\in\cH_0^1$ such that $\varphi\xi=p\xi$.
By \cite[Lemma B.3.6]{liu-zhang-zhang17}, we have that $p\varphi\bigtriangledown(D)\xi=\bigtriangledown(D)\varphi \xi=p\bigtriangledown(D)\xi$. Thus, $\varphi\bigtriangledown(D)\xi=\bigtriangledown(D)\xi$, and again by \cite[Lemma B.3.5]{liu-zhang-zhang17} this implies that $\bigtriangledown(D)\xi\in \mathcal{L}_0$. Choosing $\xi$ to be any pre-image of $s_0$, the claim follows.
Hence, this implies that $a_{0}\in p^{n}R$. 

Let $X$ and $Y$ be the variables of $\rho^\ast\cO_{\V(\cH_0^1)}$ corresponding to $f_{0}$ and $e_0$, respectively. Thus, for any $P(X)\in\rho^\ast\cO_{\V(\omega_0)}\subset\rho^\ast\cO_{\V(\cH_0^1)}=R\langle X,Y\rangle$, we have
\[
\bigtriangledown(D) P
=DP+(a_{0}X+b_{0}Y)\frac{\partial}{\partial X}P.
\]
Hence, $\Theta(D) P=DP+a_{0}X\frac{\partial}{\partial X}P$.
Since we have
\[
R\langle X\rangle=\rho^\ast\cO_{\V(\omega_0)}\longrightarrow\rho^\ast\cO_{\V(\omega_0,s_0)}=R\langle Z\rangle;\qquad X\longmapsto 1+p^nZ,
\]
we conclude that the morphism
\[
\Theta(D):R\langle Z\rangle=\rho^\ast\cO_{\V(\omega_0, s_0)}\longrightarrow \rho^\ast\cO_{\V(\omega_0, s_0)}=R\langle Z\rangle;\qquad
\Theta(D)Q=DQ+\frac{a_{0}}{p^n}(1+p^nZ)\frac{\partial}{\partial Z}Q,
\]
satisfies the desired property.
\end{proof}

Notice that Katz interpretation of the sheaves implies that
 $\omega^{{\bf k}_n^0}=\omega_0^{{\bf k}_{n,\tau_0}^0}\otimes \omega^{{\bf k}_{n}^{0,\tau_0}}$ where:
\[
\omega^{{\bf k}_{n}^{0,\tau_0}}:=\left((g_n\circ f_{0})_{\ast}\cO_{\hat\V(\omega,{\bf s}^0)}\hat\otimes\Lambda_n\right)[{\bf k}_{n}^{0,\tau_0}]^\vee,\qquad \omega_0^{{\bf k}_{n,\tau_0}^0}=\left((g_n\circ f_{0})_{\ast}\cO_{\V_0(\omega_0,s_0)}\hat\otimes\Lambda_n\right)[{\bf k}_{n,\tau_0}^{0}].
\]
Hence the connection $\Theta$ gives rise to a connection: 
\begin{equation}\label{conH1}
\Theta: \omega^{{\bf k}_n^0}\longrightarrow \omega^{{\bf k}_n^0}\otimes \Omega_{\dX_{\rm ord}}^1.
\end{equation}

Notice that the covering $\mathfrak{IG}_{1}\rightarrow\dX_{\rm ord}$ is \'etale, hence the derivation on $\cO_{\mathfrak{IG}_{1}}$ induces a connection on $\omega^{{\bf k}_f}=(g_{1,\ast}(\cO_{\mathfrak{IG}_{1}})\hat\otimes\Lambda_n)[{\bf k}_f^{-1}]$:
\begin{equation}\label{conH2}
\bigtriangledown:\omega^{{\bf k}_f}\longrightarrow \omega^{{\bf k}_f}\otimes \Omega_{\dX_{\rm ord}}^1,
\end{equation}
From the connections \eqref{conH1} and \eqref{conH2} we obtain the Serre operator:
\[
\Theta:\omega^{\bf k_n}\longrightarrow\omega^{\bf k_n}\otimes\Omega_{\dX_{\rm ord}}^1\stackrel{KS}{\simeq}\omega^{\bf k_n}\otimes\omega_0^{\otimes 2}=:\omega^{{\bf k}_n+2\tau_0}.
\]
where ${\bf k}_n+2\tau_0:\cO^\times\longrightarrow\Lambda_n\otimes_{\Z_p}\hat\cO$ is given by $({\bf k}_n+2\tau_0)(x)=(\tau_0x)^2\cdot{\bf k}_n(x)$. 
\begin{rmk}
Let $U=\Spf(R)$ be a trivialization of $\omega$, and let $w=(f_0,(f_\tau,e_\tau)_\tau)$ be a basis such that $\bigtriangledown(f_\tau)=0$, $\bigtriangledown(e_\tau)=0$ and $f_\tau$ generate $\omega_{C_\dP}$.
Let ${\bf A}/R$ is the corresponding universal abelian variety, hence a modular form is given by a distribution $\mu({\bf A},\iota,\theta,\alpha^{\dP_0},w)\in D^{{\bf k}^{\tau_0}_n}_n(\cO^{\tau_0},R)$. Given such Katz modular form interpretation, 
we have that $\Theta\mu$ evaluated at a function $\phi\in C^{{\bf k}^{\tau_0}_n}_n(\cO^{\tau_0},\Lambda_n)$ is given by 
\[
\int_{\cO^{\tau_0\times}\times\cO^{\tau_0}}\phi \;d(\Theta(D)\mu)({\bf A},\iota,\theta,\alpha^{\dP_0},w)=\Theta(D)\left(\int_{\cO^{\tau_0\times}\times\cO^{\tau_0}}\phi \;d\mu({\bf A},\iota,\theta,\alpha^{\dP_0},w)\right),
\]
for any derivation $D$.
\end{rmk}

\begin{thm} \label{t: gauss-manin p-adic powers}
Let $\mu\in H^0(\dX_{\rm ord}\times\dW_n, \omega^{{\bf k}_{n}})^\heartsuit$. There exists:
\[
M(\mu)(q)\in D^{{\bf k}^{\tau_0}_n}_n\left(\cO^{\tau_0},\mathcal{M}(\infty,\Lambda_n\hat\otimes_{W}\Lambda_{\tau_0})\right),
\]
such that, for any classical weight $k:\Lambda_{\tau_0}\longrightarrow W$ we have:
\[
(\mathrm{id}\otimes k) (M(\mu)(q))=(\Theta^{k}\mu)(q)\in D^{{\bf k}^{\tau_0}_n}_n\left(\cO^{\tau_0},\mathcal{M}(\infty,\Lambda_n)\right),
\]
where $\Theta^{k}\mu=\left(\Theta\circ\stackrel{k}{\cdots}\circ\Theta\right)\mu\in H^0(\dX_{\rm ord}\times\dW_n, \omega^{{\bf k}_{n}+2k\tau_0})$, and $({\bf k}_n+2k\tau_0)(x)=(\tau_0x)^{2k}\cdot{\bf k}_n(x)$.
\end{thm}
\begin{proof}
The morphism $\beta$ of Proposition \ref{compXGm} provides a morphism
\[
\beta^\ast:\mathcal{M}(\infty,\Lambda_n)\longrightarrow \mathcal{M}(\infty,\Lambda_n)\hat\otimes_W H^0(\G_m,\cO_{\G_m}).
\]
We can identify $H^0(\G_m,\cO_{\G_m})\simeq W[[q]]\simeq W[[\Z_p]]$. By \cite[Lemma 2.1.6 and Remark 2.3.11]{liu-zhang-zhang17}, $\beta^\ast$ restricts to a morphism:
\[ 
\beta^\ast:\mathcal{M}(\infty,\Lambda_n)^\heartsuit\longrightarrow \mathcal{M}(\infty,\Lambda_n)^\heartsuit\hat\otimes_W W[[\Z_p^\times]]=\mathcal{M}(\infty,\Lambda_n\hat\otimes_{\Z_p}\Lambda_{\tau_0})^\heartsuit\subset \mathcal{M}(\infty,\Lambda_n\hat\otimes_{\Z_p}\Lambda_{\tau_0}).
\]
We define $M(\mu)(q)$ through the following equation:
\[
\int_{\cO^{\tau_0\times}\times\cO^{\tau_0}}f(a,b)\;dM(\mu)(q):=\beta^\ast\left(\int_{\cO^{\tau_0\times}\times\cO^{\tau_0}}f(a,b)\;d\mu(q)\right)\in \mathcal{M}(\infty,\Lambda_n\hat\otimes_{\Z_p}\Lambda_{\tau_0})^{\heartsuit},
\]
for any $f\in C^{{\bf k}_n^{\tau_0}}_n(\cO^{\tau_0},\Lambda_n)$. 
Now we have to check it satisfies the desired properties.

For every closed point $x\in \dX(\infty)(\bar\kappa)$ let us consider the restriction map
\[
{\rm res}_x:\mathcal{M}(\infty,\Lambda_n)\longrightarrow \mathcal{M}(\infty,\Lambda_n)_x\simeq H^0(\G_m,\cO_{\G_m})\hat\otimes_W\Lambda_n\simeq W[[q]]\hat\otimes_W\Lambda_n,
\]
induced by the morphism $c\mid_x$ of Proposition \ref{compXGm}. I claim that
\begin{equation}\label{claimTheta}
{\rm res}_x\int_{\cO^{\tau_0\times}\times\cO^{\tau_0}}f(a,b)\;d\Theta\mu(q)=(q+1)\frac{d}{dq}\left({\rm res}_x\int_{\cO^{\tau_0\times}\times\cO^{\tau_0}}f(a,b)\;d\mu(q)\right).
\end{equation}
Let ${\bf f_0}={\rm dlog}_0(1)$ and let $(q+1)\frac{d}{dq}$ be the natural derivation in $W[[q]]$. We write $\eta_0:=\bigtriangledown\left((q+1)\frac{d}{dq}\right){\bf f}_0$. Write $\varphi=D(V^D)\circ D(\Sigma)$ acting on $\cH^1_0$ as in the proof of Lemma \ref{Serreopext}.
By \cite[Lemma B.3.6]{liu-zhang-zhang17}, 
\begin{equation}\label{eqauxthm}
\bigtriangledown\left((q+1)\frac{d}{dq}\right)\varphi\xi=p\varphi\bigtriangledown\left((q+1)\frac{d}{dq}\right)\xi,\quad\mbox{ for all }\xi\in\cH_0^1. 
\end{equation}
Moreover, $\varphi{\bf f}_0=p{\bf f}_0$ by \cite[Lemma B.3.5]{liu-zhang-zhang17}. This implies that $\varphi\eta_0=\eta_0$ by \eqref{eqauxthm}. Applying \eqref{eqauxthm} again, we deduce 
\begin{equation}\label{eqauxthm2}
p\varphi\bigtriangledown\left((q+1)\frac{d}{dq}\right)\eta_0=\bigtriangledown\left((q+1)\frac{d}{dq}\right)\eta_0. 
\end{equation}
By \cite[Lemma B.3.5]{liu-zhang-zhang17} there exists basis $\{\alpha,\beta\}$ of $\cH_0^1$ satisfying $\varphi\alpha=p\alpha$ ($\alpha={\bf f}_0$ for example) and $\varphi\beta=\beta$. The property \eqref{eqauxthm2} implies $\bigtriangledown\left((q+1)\frac{d}{dq}\right)\eta_0=f_1\alpha+f_2\beta$ with $f_1=p^{2n}\varphi^n f_1$ and $f_2=p^n\varphi^n f_2$, for all $n\in\N$. Thus, $\bigtriangledown\left((q+1)\frac{d}{dq}\right)\eta_0=0$.
By \cite[Theorem B.2.3]{liu-zhang-zhang17} we have that ${\bf f}_0^{2}=KS(\frac{dq}{q+1})$. Hence, we have that $\bigtriangledown{\bf f}_0=\eta_0{\bf f}_0^2$ and $\bigtriangledown\eta_0=0$, implying that
\[
\Theta(f(q){\bf f}_0)=\pi_{\mathcal{L}_0}\left((q+1)\frac{d}{dq}f(q){\bf f}_0^3+f(q)\eta_0 {\bf f}_0^2\right)=(q+1)\frac{d}{dq}f(q){\bf f}_0^3.
\]
This implies that, under the isomorphism 
\[
H^0(\dX_{\rm ord}\times\dW_n, \omega^{{\bf k}_{n}})_x\stackrel{\simeq}{\longrightarrow}D^{{\bf k}^{\tau_0}_n}_n(\cO^{\tau_0},\Lambda_n)\hat\otimes_W W[[q]];\qquad \mu\longmapsto \mu(x,({\bf f}_0,({\bf f}_\tau,{\bf e}_\tau)_\tau)),
\]
the Serre operator acts on $\mu\otimes f\in D^{{\bf k}^{\tau_0}_n}_n(\cO^{\tau_0},\Lambda_n)\hat\otimes_W W[[q]]$ by 
$\Theta(\mu\otimes f)=\mu\otimes(q+1)\frac{d}{dq}f$.
Hence the claim \eqref{claimTheta} follows.

Note that $\Z_p[[q]]\simeq\Z_p[[\Z_p]]$ is topologically generated by $f_\alpha(q)=(1+q)^\alpha$, with $\alpha\in\Z_p$, and $\Z_p[[\Z_p^\times]]$ is topologically generated by $f_\alpha$ with $\alpha\in\Z_p^\times$. Hence, if we write $\Theta:\G_m=\Z_p[[q]]\rightarrow\G_m=\Z_p[[q]]$, $\Theta f=(q+1)\frac{d}{dq}f$,
\[
\Theta^k f_\alpha(q)=\alpha^k f_\alpha(q)=k^\ast\left((1+T)^\alpha\right) f_\alpha(q)=k^\ast\left(f_\alpha(\hat\G_m(T,q))\right),\qquad \alpha\in\Z_p^\times.
\]
This implies that
${\rm res}_x(\mathrm{id}\otimes k) (M(\mu)(q))={\rm res}_x(\Theta^{k}\mu)(q)$,
by Proposition \ref{compXGm}. Hence the result follows.
\end{proof}

\begin{proof}[Proof of Lemma \ref{l:ordinary projector and overconvergent projection}]
We have computed that the Serre operator $\Theta$ acts on the Serre-Tate coordinates as $\Theta f=(q+1)\frac{d}{dq}f$. By remark \ref{remactUp}, this implies that $U_{\dP_0}\Theta f=p\Theta U_{\dP_0}f$ (see \cite[Lemma 2.7]{DR14}).
The result follows from Lemma \ref{lemmaNOC}, since $\gamma_{\underline{k}, m}(\bigtriangledown_{\underline{k}}g)=\Theta g$. 
\end{proof}

\begin{rmk}\label{r:result also works for c}
Observe that the results of subsections \ref{ss:unit root splitting}, \ref{ss:connections od formal vector bundles} and \ref{ss:the overconvergent projection} are equally valid if we use $\cX^{\cC}_{\mathrm{ord}}$ instead of $\cX_{\mathrm{ord}}$.
\end{rmk}

\section{Triple product $p$-adic $L$-functions}

In this section we put together the constructions performed in the text in order to produce triple product $p$-adic $L$-functions. Firstly using the results from \S \ref{s:connections and unit root} we perform a $p$-adic interpolation of the trilinear products introduced in \ref{ss:trilinear products}, which leads to triple product $p$-adic $L$-functions. These products are related to $L$-values via proposition \ref{IntProp} which imply the interpolation property satisfied for these $p$-adic $L$-functions.

\subsection{Preliminaries}
 We define the corresponding action of the operator  $U_p$ on the spaces of distributions (and a dual version of it) where our $q$-expansions live. 

Let $R$  be a $\Lambda_n$-algebra and $k^{\tau_0}:\cO^{\tau_0}\rightarrow R^\times$ be a character. Firstly observe that $U_{\dP_0}$ acts naturally on $\mathcal{M}(\infty,R)$ through the moduli interpretation of the unitary Shimura curves. Now for $\dP\neq\dP_0$ the Hecke operator $U_{\dP}: D_{n}^{k^{\tau_0}}(\cO^{\tau_0},\mathcal{M}(\infty,R)) \rightarrow D_{n}^{k^{\tau_0}}(\cO^{\tau_0},\mathcal{M}(\infty,R))$ is given as follows: Let $\mu \in D_{n}^{k^{\tau_0}}(\cO^{\tau_0},\mathcal{M}(\infty,R))$, for each $i \in \kappa_{\dP}$ we denote $g_i\ast\mu$ the distribution given by $\int_{\cO^{\tau_0\times}\times\cO^{\tau_0}}\phi\;d(g_i\ast\mu)=\int_{\cO^{\tau_0\times}\times\cO^{\tau_0}}(\varpi_\dP g_i^{-1}\ast\phi)\;d\mu$. Then we have
\[
U_{\dP}\mu=
\sum_{i\in\kappa_{\dP}}\gamma_i(g_i\ast\mu),
\]
where $\gamma_i:\mathcal{M}(\infty,R)\rightarrow \mathcal{M}(\infty,R)$ is the morphism corresponding to the (bijective) map that sends $(A,\iota,\theta,\alpha^{\dP_0},w)$ to $(A_i,\iota_i,\varpi_\dP\theta_i,\alpha_i^{\dP_0},wg_i)$. 
Recall that 
$$U_p:=\prod_{\dP}U_{\dP}: D_{n}^{k^{\tau_0}}(\cO^{\tau_0},\mathcal{M}(\infty,R)) \rightarrow D_{n}^{k^{\tau_0}}(\cO^{\tau_0},\mathcal{M}(\infty,R)).$$

Define $\bar D_{n}^{k^{\tau_0}}(\cO^{\tau_0},R)$ as the $k^{\tau_0}$-homogeneous distributions of functions on $p\cO^{\tau_0}\times\cO^{\tau_0\times}$.
We define the operator $U_{\dP}$ acting on ${\rm Hom}_{R}\left(\bar D_{n}^{k^{\tau_0}}(\cO^{\tau_0},R),\mathcal{M}(\infty,R)\right)$ by $U_{\dP}\varphi(\mu)=\varphi(U_{\dP}^\ast\mu)$;
\[
U_{\dP}^\ast\mu=
\sum_{i\in\kappa_{\dP}}\gamma_i(\bar g_i\ast\mu);\qquad\int_{p\cO^{\tau_0}\times\cO^{\tau_0\times}}\phi\;d(\bar g_i\ast\mu)=\int_{p\cO^{\tau_0}\times\cO^{\tau_0\times}}(g_i\ast\phi)\;d\mu.
\]

Then as before we put: 
$$U_p:=\prod_{\dP}U_{\dP}: {\rm Hom}_{R}(\bar D_{n}^{k^{\tau_0}}(\cO^{\tau_0}, R), \mathcal{M}(\infty,R)) \rightarrow {\rm Hom}_{R}(\bar D_{n}^{k^{\tau_0}}(\cO^{\tau_0}, R), \mathcal{M}(\infty,R)).$$

One checks that we have a well defined morphism of $R[U_p]$-modules:
\begin{eqnarray}\label{e:distributions vs duals}
D_{n}^{k^{\tau_0}}(\cO^{\tau_0},\mathcal{M}(\infty,R))\longrightarrow{\rm Hom}_{R}\left(\bar D_{n}^{k^{\tau_0}}\left(\cO^{\tau_0},R\right),\mathcal{M}(\infty,R)\right)
\end{eqnarray}
which sends $\mu_1 \in D_{n}^{k^{\tau_0}}(\cO^{\tau_0},\mathcal{M}(\infty,R))$ to the morphism:
$$\mu_2 \mapsto \int_{\cO^{\tau_0\times}\times\cO^{\tau_0}}\int_{p\cO^{\tau_0}\times\cO^{\tau_0\times}}k^{\tau_0}(xY-Xy)d\mu_1(x,y)d\mu_2(X,Y).$$
\begin{defi}
Considering the standard ordinary operator $e_{\rm ord}:=\lim_n U_p^{n!}$ we define: 
\begin{eqnarray*}
D_{n}^{k^{\tau_0}}(\cO^{\tau_0},\mathcal{M}(\infty,R))^{\rm ord}&:=&e_{\rm ord}\left(D_{n}^{k^{\tau_0}}(\cO^{\tau_0},\mathcal{M}(\infty,R))\right),\\ 
{\rm Hom}_{R}\left(\bar D_{n}^{k^{\tau_0}}\left(\cO^{\tau_0},R\right),\mathcal{M}(\infty,R)\right)^{\rm ord}&:=&e_{\rm ord}\left({\rm Hom}_{R}\left(\bar D_{n}^{k^{\tau_0}}\left(\cO^{\tau_0},R\right),\mathcal{M}(\infty,R)\right)\right),
\end{eqnarray*}
the spaces where $U_p$ is invertible.
\end{defi}

\begin{lemma}\label{compDA}
 The morphism (\ref{e:distributions vs duals}) induces an isomorphism of $R[U_p]$-modules:
\[
D_{n}^{{\bf k}^{\tau_0}}(\cO^{\tau_0},\mathcal{M}(\infty,R))^{\rm ord}\longrightarrow{\rm Hom}_{R}\left(\bar D_{n}^{{\bf k}^{\tau_0}}\left(\cO^{\tau_0},R\right),\mathcal{M}(\infty,R)\right)^{\rm ord}.
\]

\end{lemma}
\begin{proof} Since the morphism is continuous, it is enough to prove that the specialization of any classical weight $\underline{k}\in \N[\Sigma_F]$ is an isomorphism, since such classical weights form a dense set. Write $A=\mathcal{M}(\infty,R_{\underline{k}})$. 
The same arguments in the proof of proposition \ref{propclassty} show
 that the ordinary parts of $D_{n}^{\underline{k}^{\tau_0}}(\cO^{\tau_0},A)$ and ${\rm Sym}^{\underline{k}^{\tau_0}}(A^2)^\vee$ are naturally isomorphic. Similarly, one can proof that the ordinary parts of ${\rm Hom}_{R_{\underline{k}}}\left(\bar D_{n}^{\underline{k}^{\tau_0}}\left(\cO^{\tau_0},R_{\underline{k}}\right),A\right)$ and ${\rm Sym}^{\underline{k}^{\tau_0}}(A^2)$ also agree.
 Thus taking specialization at $\underline{k}$ we obtain the natural map between ${\rm Sym}^{\underline{k}^{\tau_0}}(A^2)^\vee$ and ${\rm Sym}^{\underline{k}^{\tau_0}}(A^2)$ which is an isomorphism.
\end{proof}

\subsection{$p$-adic families of trilinear products}

We fix integers $r\geq n_3\geq n_1,n_2\geq 1$. For $i= 1, 2, 3$ we denote by $({\bf r}_{n_i},{\bf \nu}_{n_i}):\cO^\times\times\Z_p^\times\rightarrow\Lambda_{n_i}^{G}$ the universal characters of $\dW^G_{n_i}$ and $k:\dW^G_{n_i}\rightarrow\dW_{n_i}$ as introduced in (\ref{e:map between weight spaces}). Then we put ${\bf k}_{n_3}:= k({\bf r}_{n_3},\nu_{n_1}+ \nu_{n_2})$ and ${\bf k}_{n_i}:= k({\bf r}_{n_i},{\bf \nu}_{n_i})$ for $i= 1, 2$. 

 We put $\cR:=\Lambda_{n_1}^G\hat\otimes\Lambda_{n_2}^G\hat\otimes_{{\bf \nu}_{n_3}=\nu_{n_1}+{\bf \nu}_{n_2}} \Lambda_{n_3}^G\simeq \Lambda_{n_1}^G\hat\otimes\Lambda_{n_2}^G\hat\otimes \Lambda_{n_3}$ and consider the characters:
\begin{eqnarray*}
{\bf m}_1^{\tau_0}:={\bf r}_{n_1}^{\tau_0}-{\bf r}_{n_3}^{\tau_0}-{\bf r}^{\tau_0}_{n_2}+{\bf \nu}_{n_2}\circ N:\cO^{\tau_0\times}&\longrightarrow&\cR,\\
{\bf m}_2^{\tau_0}:={\bf r}^{\tau_0}_{n_2}-{\bf r}^{\tau_0}_{n_1}-{\bf r}^{\tau_0}_{n_3}+{\bf \nu}_{n_1}\circ N:\cO^{\tau_0\times}&\longrightarrow&\cR,\\
{\bf m}_3^{\tau_0}:={\bf r}^{\tau_0}_{n_3}-{\bf r}^{\tau_0}_{n_1}-{\bf r}^{\tau_0}_{n_2}:\cO^{\tau_0\times}&\longrightarrow&\cR,\\
{\bf m}_{3,\tau_0}:={\bf r}_{n_1,\tau_0}+{\bf r}_{n_2,\tau_0}-{\bf r}_{n_3,\tau_0}:\Z_p^\times&\longrightarrow&\cR,
\end{eqnarray*}
where $N:\cO^{\tau_0\times}\rightarrow\Z_p^\times$ denotes the norm map.
In the same way as in (\ref{e: important polynomial appearing a lot}) 
we denote $\underline{\Delta}^{\tau_0}\in C_{n_1}^{{\bf k}_{n_1}^{\tau_0}}(\cO^{\tau_0},\cR)\otimes_{\cR} C_{n_2}^{{\bf k}_{n_2}^{\tau_0}}(\cO^{\tau_0},\cR)\otimes_{\cR} \bar C_{n_3}^{{\bf k}_{n_3}^{\tau_0}}(\cO^{\tau_0},\cR)$
the function defined by:
\[
\underline{\Delta}^{\tau_0}((x_1,y_1),(x_2,y_2),(x_3,y_3)):={\bf m}_1^{\tau_0}(x_3y_2-x_2y_3) \cdot {\bf m}_2^{\tau_0}(x_3y_1-x_1y_3)\cdot {\bf m}_3^{\tau_0}(x_1y_2-x_2y_1),
\] 
where $\bar C_{n}^{k^{\tau_0}}(\cO^{\tau_0},\cdot)$ denote the $k^{\tau_0}$-homogeneous locally analytic functions on $p\cO^{\tau_0}\times\cO^{\tau_0\times}$, and the function is
extended by 0 where ${\bf m}_3^{\tau_0}$ is not defined.

Now we take $\mu_1\in M^{r}_{{\bf k}_{n_1}}(\Gamma_1(\cN,p),\Lambda_{n_1}^G)$ 
 and $\mu_2\in M^{r}_{{\bf k}_{n_2}}(\Gamma_1(\cN,p),\Lambda_{n_2}^G)$ 
be overconvergent modular forms. By the definition (see definition \ref{d:overconvergent families for G})  if $i=1, 2$ and $\cC\in\Pic(\cO_F)$, we have the components $\mu_{i, \cC} \in M^r_{{\bf k}_n}(\Gamma_{1,1}^{\mathfrak{c}}(\cN,p),\Lambda_{n_i}^G)^{\Delta}$ and we denote by $\mu_{i,\cC}(q)\in D_{n_i}^{{\bf k}_{n_i}^{\tau_0}}\left(\cO^{\tau_0},\mathcal{M}(\infty,\Lambda_{n_i}^G)\right)$ their $q$-expansions (see \S\ref{ss:q-expansions and serre-tate}). By the proof of proposition \ref{Up-0} (more precisely equation \eqref{q-expU}) and the fact that $U_{\dP_0}\circ V_{\dP_0}={\rm Id}$ one checks that:
\[
\mu_{1,\cC}^{[p]}(q):=(1-V_{\dP_0}U_{\dP_0})\;\mu_{1,\cC}(q)\in D_{n_1}^{{\bf k}_{n_1}^{\tau_0}}\left(\cO^{\tau_0},\mathcal{M}(\infty,\Lambda_{n_1}^G)^{\heartsuit}\right).
\]
Using theorem \ref{t: gauss-manin p-adic powers} and remark \ref{r:result also works for c} 
we obtain $M(\mu_{1,\cC}^{[p]})(q) \in D_{n_1}^{{\bf k}_{n_1}^{\tau_0}}(\cO^{\tau_0},\mathcal{M}(\infty,\Lambda_{n_1}^G\otimes\Lambda_{\tau_0})^{\heartsuit})$ 
and we put $\Theta^{{\bf m}_{3, \tau_0}}\mu_{1,\cC}^{[p]}:=({\bf m}_{3, \tau_0})^\ast M(\mu_{1,\cC}^{[p]})\in D_{n_1}^{{\bf k}_{n_1}^{\tau_0}}(\cO^{\tau_0},\mathcal{M}(\infty,\cR))$. 

We define the \emph{families of trilinear products} $t(\mu_{1,\cC},\mu_{2,\cC})\in {\rm Hom}_{\cR}(\bar D_{n_3}^{{\bf k}_{n_3}^{\tau_0}}(\cO^{\tau_0},\cR),\mathcal{M}(\infty,\cR))$ as follows:
\[
t(\mu_{1,\cC},\mu_{2,\cC})(\mu):=\int_{\cO^{\tau_0\times}\times\cO^{\tau_0}}\int_{\cO^{\tau_0\times}\times\cO^{\tau_0}}\int_{p\cO^{\tau_0}\times\cO^{\tau_0\times}}\underline{\Delta}^{\tau_0}(v_1,v_2,v_3)\;d\left(\Theta^{{\bf m}_{3, \tau_0}}\mu_{1,\cC}^{[p]}\right)(v_1)d\mu_{2,\cC}(v_2)d\mu(v_3).
\]
From lemma \ref{compDA} we obtain:
\[
(e_{\rm ord}t(\mu_{1,\cC},\mu_{2,\cC}))_{\cC}\in \bigoplus_{\cC\in\Pic(\cO_F)}\left(D_{n_3}^{{\bf k}_{n_3}^{\tau_0}}(\cO^{\tau_0},\mathcal{M}(\infty,\cR))\otimes\Q_p\right)^{\rm ord}=:M^{p-\rm{adic}}_{{\bf k}_{n_3}}(\Gamma_1(\cN,p),\cR)^{\rm ord}.
\]
Note that the space $M^{p-\rm{adic}}_{{\bf k}_n}(\Gamma_1(\cN,p),\cR)^{\rm ord}$ is endowed with the action of Hecke operators.

\subsection{Construction} Let $\mu_1\in M^{r}_{{\bf k}_{n_1}}(\Gamma_1(\cN,p),\Lambda_{n_1}^G)$ and $\mu_2\in M^{r}_{{\bf k}_{n_2}}(\Gamma_1(\cN,p),\Lambda_{n_2}^G)$ be as before 
and moreover we take $\mu_3\in M^r_{{\bf k}_{n_3}}(\Gamma_1(\cN,p),\Lambda_{n_3}^G)$ 
such that is eigenvector for the Hecke operators and  such that $U_\dP\mu_3= \alpha_3^\dP\mu_3$ for some $\alpha_3^\dP\in  (\Lambda_{n_3}^G)^\times$ and all $\dP\mid p$. Assume that there exists $\bar\mu_3\in M^r_{{\bf k}_{n_3}}(\Gamma_1(\cN_0,p),\Lambda_{n_3}^G)$ for some $\cN_0\mid\cN$ such that $\mu_3$ is an element of the space:
\[
M^{p-\rm{adic}}_{{\bf k}_{n_3}}(\Gamma_1(\cN,p), \Lambda_{n_3}^G)^{\rm ord}[\bar\mu_3]:=\{\mu\in M^{p-\rm{adic}}_{{\bf k}_{n_3}}(\Gamma_1(\cN,p),\Lambda_{n_3}^G)^{\rm ord};\quad U_\dP\mu=\alpha_3^\dP\mu;\;\; T_{\ell}\mu=a_\ell\mu,\;\ell \nmid \cN\},
\]
where $\ell$ are prime ideals of $F$, $T_\ell=T_g$ for any $g\in G(\A_f^p)$ of norm $\ell$ and $a_\ell$ is the eigenvalue of $\bar\mu_3$. Let $\cR'=\Lambda_{n_1}^G\hat\otimes\Lambda_{n_2}^G\hat\otimes\Lambda_{n_3}'$, where $\Lambda_{n_3}'$ is the fraction field of $\Lambda_{n_3}$, thus $\cR'$ can be viewed as rational functions on $\dW_{n_1}^G\times\dW_{n_2}^G\times\dW_{n_3}$ with poles at finitely many weights in $\dW_{n_3}$.

In the rest we use the following notation: If  $(x, y, z) \in\dW_{n_1}^G\times\dW_{n_2}^G\times\dW_{n_3}$ we denote by $\mu_x$, $\mu_y$, $\mu_z$ and $\bar\mu_z$ the specializations of the families $\mu_1$ at $x$, $\mu_2$ at $y$, and $\mu_3$ and $\bar\mu_3$ at $z$ respectively.  By proposition \ref{propclassty}, if $z\in \dW_{n_3}$ is a classical weight  then the specialization $\mu_{z}$ is a classical modular form in $M_{k_z}(\Gamma_1(\cN,p),\bar\Q_p)^{\rm ord}$, where $k_z$ is the specialization of ${\bf k}_{n_3}$ at $z$. 
Let us denote by $\bar\mu_z^*\in M_{k_z}(\Gamma_1(\cN_0,p),\bar\Q_p)$ the eigenvector for the adjoint of the $U_\dP$ operators associated with $\bar\mu_z$ as in equation \eqref{adjointUp}.
The following result is analogous to \cite[Lemma 2.19]{DR14}.
\begin{lemma}\label{l:construction}
There exists $\mathcal{L}_p(\mu_1,\mu_2,\mu_3)\in\cR'$ such that for each classical point $(x, y, z) \in\dW_{n_1}^G\times\dW_{n_2}^G\times\dW_{n_3}$,  we have:
\[
\mathcal{L}_p(\mu_1,\mu_2,\mu_3)(x, y, z)=\frac{\left\langle\mu_z^*,(e_{\rm ord}t(\mu_{1,\cC},\mu_{2,\cC}))_{\cC(x, y, z)}\right\rangle}{\left\langle\mu_z^*,\mu_z\right\rangle}
\] 
where $\langle\cdot,\cdot\rangle$ is the Petersson inner product defined in \S\ref{Ichino}, and $\mu_z^*\in M_{k_z}(\Gamma_1(\cN,p),\bar\Q_p)[\bar\mu_z^*]$ defines the dual basis of $\mu_z$.
\end{lemma}
\begin{proof}
Notice that $M^{p-\rm{adic}}_{{\bf k}_{n_3}}(\Gamma_1(\cN,p),\Lambda_{n_3}')^{\rm ord}[\bar\mu_3]$ is a finite dimensional $\Lambda_{n_3}'$-vector space generated by the oldforms $\bar\mu_3^d$, for any $d\mid\cD$ with $\cN=\cN_0\cD$, where 
\[
(\bar\mu_3)_\cC^d(A,\iota,\theta,\alpha^{\dP_0},w)={\bf r}_n'(\det(\gamma_d))\cdot(\bar\mu_3)_{\cC'}(A^{g_d},\iota^{g_d},\det(\gamma_d)^{-1}\theta^{g_d},k^{-1}(\alpha^{g_d})^{\dP_0},w).
\]
as in equation \eqref{eqoldformKatz} in \S\ref{oldforms} of the Appendix. The family $\bar\mu_3$ corresponds to an idempotent of the Hecke algebra, which induces a projection of $M^{p-\rm{adic}}_{{\bf k}_{n_3}}(\Gamma_1(\cN,p),\cR)^{\rm ord}$ to $M^{p-\rm{adic}}_{{\bf k}_{n_3}}(\Gamma_1(\cN,p),\cR')^{\rm ord}[\bar\mu_3]$. The projection of $e_{\rm ord}t(\mu_{1,\cC},\mu_{2,\cC})$ to the line defined by $\mu_3$ is a $\cR'$-linear combination of the forms $\bar\mu_3^d$. 
It is therefore enough to show that, for all divisors $d_1,d_2$,
\[
\left\langle\bar\mu_z^*,\bar\mu_z^{d_2}\right\rangle=\varrho(z)\cdot\left\langle\bar\mu_z^*,\bar\mu_z\right\rangle;\qquad\mbox{for some }\varrho\in\cR\otimes\Q.
\]
By lemma \ref{lemapetersson}, we have that
\[
\varrho=\varrho\left((-{\bf \nu}_{n_1}-\nu_{n_2})(q_d)_x,{\bf a}_{x}\right)_{d\mid d_1d_2}\qquad \varrho(X_d,Y_d)_{d\mid d_1d_2}\in \bar\Q[X_d,Y_d]_{d\mid d_1d_2},
\]
where $q_d:={\rm Norm}_{K/\Q}(d)$ and ${\bf a}\in\cR$ is the eigenvalue for $T_d$.
Hence the result follows since $q_d$ is prime to $p$.
\end{proof}

\begin{defi}
Let $\mu_1,\mu_2,\mu_3$, where $\mu_i\in M^r_{{\bf k}_{n_i}}(\Gamma_1(\cN,p),\Lambda_{n_i}^G)$, be test vectors for three families of eigenvectors such that $U_\dP\mu_3= \alpha_3^\dP\mu_3$ for some $\alpha_3^\dP\in  (\Lambda_{n_3}^G)^\times$ and all $\dP\mid p$. The functions $\mathcal{L}_p(\mu_1,\mu_2,\mu_3)\in\cR'$ introduced in \ref{l:construction} is called the \emph{triple product $p$-adic L-function} of $\mu_1,\mu_2,\mu_3$. 
\end{defi}

\subsection{Interpolation property}

Let $(\underline{r}_1,\nu_1)\in \dW_{n_1}^G$, $(\underline{r}_2,\nu_2)\in \dW_{n_2}^G$ and $\underline{r}\in \dW_{n_3}$ be classical weights and put $\underline{k}_1=k(\underline{r}_1, \nu_1), \underline{k}_2=k(\underline{r}_2, \nu_2)$ and $\underline{k}_3=k(\underline{r}_3, \nu_1+\nu_2)$ where $k$ is the map (\ref{e:map between weight spaces}). We suppose that $(\underline{k}_1,\underline{k}_2,\underline{k}_3)$ is unbalanced at $\tau_0$ with dominant weight $\underline{k}_3$, and $k_{\tau,i}\in\Z_{>0}$ for each $\tau\in \Sigma_F$ and $i= 1, 2, 3$. 
 
 We write $(x, y, z) \in \dW_{n_1}^G\times\dW_{n_2}^G\times\dW_{n_3}$ for the point corresponding to the triple $(\underline{k}_1,\nu_1)$, $(\underline{k}_2,\nu_2)$ and $(\underline{k}_3, \nu_1+\nu_2)$. As $\mu_3$ is ordinary then from proposition \ref{propclassty} and corollary \ref{c:automorphic forms as sections} we deduce that its specialization at $\mu_z$ correspond to an automorphic form of weight $(\underline{k}_3,\nu_1+\nu_2)$. If  $\underline{k}_1$ and $\underline{k}_2$ are big enough the the same is true for $\mu_x$ and $\mu_y$, obtaining automorphic forms of weights $(\underline{k}_1,\nu_1)$ and $(\underline{k}_2, \nu_2)$ respectively. We denote by $\pi_{x}$, $\pi_y$ and $\pi_z$ the automorphic representations of $(B\otimes\A_F)^{\times}$ generated  by these automorphic forms, and $\Pi_{x}$, $\Pi_y$ and $\Pi_z$ the corresponding cuspidal automorphic representations of $\mathrm{GL}_2(\A_F)$. 

Assume that $\mu_i$ are eigenvectors for all the $U_\dP$ operators, namely $U_\dP\mu_i=\alpha_{i}^{\dP}\cdot\mu_i$, and write $\alpha_{x}^{\dP}$, $\alpha_{y}^{\dP}$ and $\alpha_{z}^{\dP}$ for the corresponding specializations at $x$, $y$ and $z$. 
Moreover, we assume that $\bar\mu_{x}$ is the $\mathfrak{p}$-stabilization of the newform $\bar\mu_{x}^\circ$ for each $\dP\mid p$. Write $\beta_{i}^{\dP}$ for the other eigenvalue of $U_\dP$ as usual, and write $\beta_{x}^{\dP}$, $\beta_{y}^{\dP}$ and $\beta_{z}^{\dP}$ for the corresponding specializations. 

The following result justify the name given to $\mathcal{L}_p(\mu_1,\mu_2,\mu_3)$.
\begin{thm}\label{t:interpolation} With the notations above we have:
$$\mathcal{L}_p(\mu_1,\mu_2,\mu_3)(x, y, z)= K(\mu_{x}^\circ,\mu_{y}^\circ,\mu_{z}^\circ)\cdot\left(\prod_{\dP\mid p}\frac{\mathcal{E}_{\dP}(x,y,z)}{\mathcal{E}_{\dP,1}(z)}\right)\cdot\frac{L\left(\frac{1-\nu_1-\nu_2-\nu_3}{2},\Pi_{x}\otimes\Pi_{y}\otimes\Pi_{z}\right)^{\frac{1}{2}}}{\langle\bar\mu_{z}^\circ,\bar\mu_{z}^\circ\rangle}$$
here  $K(\mu_1^\circ,\mu_2^\circ,\mu_3^\circ)$ is a non-zero constant, $\mathcal{E}_\dP(x,y,z)=$
\[
\left\{\begin{array}{lc}
\mbox{\small$(1-\beta_{x}^{\dP}\beta_{y}^{\dP}\alpha_{z}^{\dP}\varpi_{\dP}^{-\underline{m}_{\dP}-\underline{2}})(1-\alpha_{x}^{\dP}\beta_{y}^{\dP}\beta_{z}^{\dP}\varpi_{\dP}^{-\underline{m}_{\dP}-\underline{2}})(1-\beta_{x}^{\dP}\alpha_{y}^{\dP}\beta_{z}^{\dP}\varpi_{\dP}^{-\underline{m}_{\dP}-\underline{2}})(1-\beta_{x}^{\dP}\beta_{y}^{\dP}\beta_{z}^{\dP}\varpi_{\dP}^{-\underline{m}_{\dP}-\underline{2}})$},&\dP\neq\dP_0\\
\mbox{\small$(1-\alpha_{x}^{\dP_0}\alpha_{y}^{\dP_0}\beta_{z}^{\dP_0}p^{1-m_{0}})(1-\alpha_{x}^{\dP_0}\beta_{y}^{\dP_0}\beta_{z}^{\dP_0}p^{1-m_{0}})(1-\beta_{x}^{\dP_0}\alpha_{y}^{\dP_0}\beta_{z}^{\dP_0}p^{1-m_{0}})(1-\beta_{x}^{\dP_0}\beta_{y}^{\dP_0}\beta_{z}^{\dP_0}p^{1-m_{0}})$},&\dP=\dP_0
\end{array}\right., 
\]
\[
\mathcal{E}_{\dP,1}(z):=\left\{\begin{array}{lc} 
(1- (\beta_z^{\dP})^2\varpi_{\dP}^{-\underline{k}_{3,\dP}-\underline{2}})\cdot (1- (\beta_z^{\dP})^2\varpi_{\dP}^{-\underline{k}_{3,\dP}-\underline{1}}),&\dP\neq \dP_0,\\
(1- (\beta_z^{\dP_0})^{2}p^{-k_{3,\tau_0}})\cdot (1- (\beta_z^{\dP_0})^{2}p^{1-k_{3,\tau_0}}),&\dP= \dP_0,
\end{array}\right. 
\]
$m_{0}=\frac{k_{1,\tau_0}+k_{2,\tau_0}+k_{3,\tau_0}}{2}\geq 0$, and $\underline{m}_{\dP}=\frac{\underline{k}_{1,\dP}+\underline{k}_{2,\dP}+\underline{k}_{3,\dP}}{2}\in\Z[\Sigma_\dP]$.
\end{thm}
\begin{proof} By construction we have:
\begin{equation}\label{e:interpolation proof 1}
\mathcal{L}_p(\mu_1,\mu_2,\mu_3)(x, y, z)=\frac{\left\langle\mu_z^*,e_{\rm ord}(\Theta^{m_{3, \tau_0}}\mu_{x, \cC}^{[p]}\cdot\mu_{y, \cC}(\underline{\Delta}_{(x,y,z)}^{\tau_0}))_{\cC}\right\rangle}{\left\langle \mu_z^*,\mu_z\right\rangle}.
\end{equation}
Observe first that $\underline{\Delta}_{(x,y,z)}^{\tau_0}$ differs from $\Delta^{\tau_0}$ of equation \eqref{e: important polynomial appearing a lot}. Indeed, $\underline{\Delta}_{(x,y,z)}^{\tau_0}((x_1,y_1),(x_2,y_2),(x_3,y_3))$ is extended by zero whether $(x_1y_2-x_2y_1)\not\in\cO^{\tau_0\times}$. Thus if we denote $\mu_{x,y}=\Theta^{m_{3, \tau_0}}\mu_{x, \cC}^{[p]}\cdot\mu_{y, \cC}$ then we obtain:
\[
\varepsilon:=
\int\underline{\Delta}_{(x,y,z)}^{\tau_0}d\mu_{x,y}d\mu_z-\int\Delta^{\tau_0}d\mu_{x,y}d\mu_z
=\int_D\int_{p\cO^{\tau_0}\times\cO^{\tau_0\times}}\Delta^{\tau_0}(v_1,v_2,v_3)d\mu_{x,y}(v_1,v_2)d\mu_z(v_3),
\]
where  $D=\prod_{\dP\neq\dP_0}D_\dP$ and
\[
D_\dP=\{((x_1,y_1),(x_2,y_2))\in(\cO_\dP^{\times}\times\cO_\dP)^2:\;(x_1y_2-x_2y_1)\not\in\cO_\dP^{\times}\}=\bigcup_{i\in\kappa_\dP}(D_i\times D_i),
\]
with $D_i=\bigcup_{a\in\kappa_\dP^\times}(a+ p\cO_\dP)\times(ai+ p\cO_\dP)$. Since $\mu_i$ are $U_\dP$-eigenvectors, a calculation similar to that of proof of Proposition \ref{Upcomp} shows that the corresponding $\dP$-component
$\varepsilon_\dP:=\int_{D_\dP}\int_{p\cO_\dP\times\cO_\dP^{\times}}\Delta^{\tau_0}d\mu_{x,y}d\mu_z$ satisfies
\begin{eqnarray*}
\varepsilon_\dP&=&\frac{1}{\alpha_{x}^{\dP}\alpha_{y}^{\dP}}\sum_{i\in\kappa_\dP}\int_{D_i}\int_{D_i}\int_{p\cO_\dP\times\cO_\dP^{\times}}\Delta^{\tau_0}dU_\dP\mu_{x,y}d\mu_z\\
&=&\frac{1}{\alpha_{x}^{\dP}\alpha_{y}^{\dP} }\sum_{i\in\kappa_\dP}\gamma_i\int_{\cO_\dP^\times\times\cO_\dP}\int_{\cO_\dP^\times\times\cO_\dP}\int_{p\cO_\dP\times\cO_\dP^{\times}}\Delta^{\tau_0}(v_1\varpi_\dP g_i^{-1},v_2\varpi_\dP g_i^{-1},v_3)d\mu_{x,y}(v_1,v_2)d\mu_z(v_3)\\
&=&\frac{\varpi_\dP^{\underline{m}_{3,\dP}}}{\alpha_{x}^{\dP}\alpha_{y}^{\dP} }\sum_{i\in\kappa_\dP}\gamma_i\int_{\cO_\dP^\times\times\cO_\dP}\int_{\cO_\dP^\times\times\cO_\dP}\int_{p\cO_\dP\times\cO_\dP^{\times}}\Delta^{\tau_0}(v_1,v_2,v_3 g_i)d\mu_{x,y}(v_1,v_2)d\mu_z(v_3)\\
&=&\frac{\varpi_\dP^{\underline{m}_{3,\dP}}}{\alpha_{x}^{\dP}\alpha_{y}^{\dP} }\int_{\cO_\dP^\times\times\cO_\dP}\int_{\cO_\dP^\times\times\cO_\dP}\int_{p\cO_\dP\times\cO_\dP^{\times}}\Delta^{\tau_0}(v_1,v_2,v_3)d\mu_{x,y}(v_1,v_2)U_\dP^{\ast} d\mu_z(v_3)
\end{eqnarray*}
since $\varpi_\dP g_j^{-1}\ast1_{D_i}=0$ if $i\neq j$ and $\varpi_\dP g_i^{-1}\ast1_{D_i}=1_{\cO_\dP^\times\times\cO_\dP}$. By remark \ref{rmkonU}, $\alpha_x^{\dP}\beta_x^{\dP}=\varpi_\dP^{\underline{k}_{1,\dP}+\underline{1}}$ if $\dP\neq\dP_0$ and $\alpha_x^{\dP_0}\beta_x^{\dP_0}=p^{k_{1,\tau_0}-1}$. This implies that
\[
\mathcal{L}_p(\mu_1,\mu_2,\mu_3)(x, y, z)=\left(\prod_{\dP\neq \dP_0}\left(1-\beta_{x}^{\dP}\beta_{y}^{\dP}\alpha_{z}^{\dP}\varpi_{\dP}^{-\underline{m}_{\dP}-\underline{2}}\right)\right)\cdot\frac{\left\langle\mu_z^*,e_{\rm ord}(\Theta^{m_{3, \tau_0}}\mu_{x, \cC}^{[p]}\cdot\mu_{y, \cC}(\Delta^{\tau_0}))_{\cC}\right\rangle}{\left\langle\mu_z^*,\mu_z\right\rangle}.
\]

As $\bigtriangledown_{\underline{k}_1}^{m_{3, \tau_0}}\mu_{x, \cC}^{[p]}\mu_{y, \cC}(\Delta^{\tau_0})\in  H^0(\cX^{\cC}_r,\cH^{\underline{k}_3}_{m_{3, \tau_0}})$ and $0 \leq 2m_{3, \tau_0}<k_{\tau_0,3}$ from lemmas \ref{lemmaNOC}, \ref{l:ordinary projector and overconvergent projection} and remark \ref{r:result also works for c} we obtain $e_{\rm ord}(\Theta^{m_{3, \tau_0}}\mu_{x,\cC}^{[p]}\mu_{y,\cC}(\Delta^{\tau_0}))= $
\begin{equation}\label{e:interpolation proof 2}
e_{\rm ord}(\gamma_{\underline{k}_3, m_{3, \tau_0}}\bigtriangledown^{m_{3, \tau_0}}\mu_{x,\cC}^{[p]}\mu_{y,\cC}(\Delta^{\tau_0}))=e_{\rm ord}(\cH^r(\bigtriangledown^{m_{3, \tau_0}}\mu_{x,\cC}^{[p]}\mu_{y,\cC}(\Delta^{\tau_0}))).
\end{equation}
A laborious but straightforward computation shows that
\begin{eqnarray*}
&&\bigtriangledown^{m_{3,\tau_0}}\mu_{x,\cC}^{[p]}\mu_{y,\cC}(\Delta^{\tau_0}))=\\
&&=(-1)^{m_{3,\tau_0}}\binom{k_{3,\tau_0}-2}{m_{3,\tau_0}+k_{2,\tau_0}-1}^{-1} t(\mu_{x,\cC}^{[p]},\mu_{y,\cC})+\bigtriangledown\left(\sum_{i=0}^{m_{3,\tau_0}-1}a_i\bigtriangledown^{i}\mu_{x,\cC}^{[p]}\bigtriangledown^{m_{3,\tau_0}-1-i}\mu_{y,\cC}(\Delta^{\tau_0}))\right),
\end{eqnarray*}
where $t(\mu_{x,\cC}^{[p]},\mu_{y,\cC})$ is defined as in Theorem \ref{tripleonsheaves} and
\[
a_i=(-1)^{i+m_{3,\tau_0}+1}\binom{k_{3,\tau_0}-2}{m_{3,\tau_0}+k_{2,\tau_0}-1}^{-1}\left(\sum_{j=0}^i\binom{m_{3,\tau_0}}{j}\binom{m_{3,\tau_0}+k_{1,\tau_0}+k_{2,\tau_0}-2}{k_{1,\tau_0}+j-1}\right).
\]
This  relation above implies:
\begin{equation}\label{e:interpolation proof 3}
\cH^r(\bigtriangledown^{m_{3, \tau_0}}\mu_{x,\cC}^{[p]}\mu_{y,\cC}(\Delta^{\tau_0}))=(-1)^{m_{3, \tau_0}}\binom{k_{3,\tau_0}-2}{m_{3, \tau_0}+k_{2,\tau_0}-1}^{-1}\cdot t(\mu_{x,\cC}^{[p]},\mu_{y,\cC}).
\end{equation}

%

Since $\mu_z$ is ordinary, we obtain from \eqref{e:interpolation proof 1}, \eqref{e:interpolation proof 2} and \eqref{e:interpolation proof 3}:
\[
\mathcal{L}_p(\mu_1,\mu_2,\mu_3)(x, y, z)=(-1)^{m_{3, \tau_0}}\binom{k_{3,\tau_0}-2}{m_{3, \tau_0}+k_{2,\tau_0}-1}^{-1}\cdot\frac{\left\langle\mu_z^*, t(\mu_x^{[p]},\mu_y)\right\rangle}{\left\langle \mu_z^*,\mu_z\right\rangle}\cdot\prod_{\dP\neq \dP_0}\left(1-\beta_{x}^{\dP}\beta_{y}^{\dP}\alpha_{z}^{\dP}\varpi_{\dP}^{-\underline{m}_{\dP}-\underline{2}}\right).
\]
Thus, the result follows by theorem \ref{tripleonsheaves}, proposition \ref{IntProp}, remark \ref{rmkonU} and proposition \ref{eulerfactors}.
Notice that $K(\mu_x^\circ,\mu_y^\circ, \mu_z^\circ)= (C\cdot C(\mu_{x}^\circ,\mu_{y}^\circ,\mu_{z}^\circ))^{1/2}\cdot 2^{-2 +m_{3, \tau_0}}$, where the constants $C$ and $C(\mu_{x}^\circ,\mu_{y}^\circ,\mu_{z}^\circ)$ are as given in proposition \ref{IntProp}.
\end{proof}

\part{Appendix}\label{Delta and Hecke}

\section{$\Delta$-action and Hecke operators}\label{Heckops}
In this section we verify that the two actions of  $\Delta$ introduced in \S \ref{GvsG'} are compatible with the morphism introduced in lemma \ref{l:seccions as complex forms}. Moreover, we describe the action of the Hecke operators on the space of quaternionic automorphic forms in terms of the associated moduli description of the unitary Shimura curves.

\subsection{Compatibility of the $\Delta$-action}
In \S \ref{GvsG'} we defined $M_{\underline{k}}(\Gamma_{1,1}^{\cC}(\cN),\C)$ to be the $\C$-vector space of holomorphic functions $f:\dH\rightarrow\bigotimes_{\tau \neq \tau_0}\cP_\tau(k_\tau)^\vee$ such that $f(\gamma z)=(cz+d)^{k_{\tau_0}}\gamma f(z)$ for all $\gamma\in \Gamma_{1,1}^{\cC}(\cN)$. This space is endowed with an action of $\Delta=(\cO_F)_+^\times /U_\cN^2$ given by 
\[
s\ast_{\nu} f(z):=s^{\frac{-\underline{k}+ 2k_{\tau_0}\tau_0+ \nu\underline{1}}{2}}(cz+d)^{-k_{\tau_0}}\gamma_s^{-1} f(\gamma_sz),\qquad s\in \Delta,
\]
for any $\gamma_s\in\Gamma_1^{\cC}(\cN)$ with $\det\gamma_s=s$.
By Lemma \ref{l:seccions as complex forms} and Remark \ref{rmkonGammac}, we can interpret elements in $M_{\underline{k}}(\Gamma_{1,1}^{\cC}(\cN),\C)$ as sections of $\omega^{\underline{k}}$ at a connected component of certain Shimura curve $X^{\cC}$. Given $f\in M_{\underline{k}}(\Gamma_{1,1}^{\cC}(\cN),\C)$ the corresponding section is 
\[
f(z)\left(\prod_{\tau\neq\tau_0}\left|\begin{array}{cc}dx_\tau&dy_\tau\\x_\tau&Y_\tau\end{array}\right|^{k_\tau}\right)dx_{\tau_0}^{k_{\tau_0}}\in H^0(X^{\cC},\omega^{\underline{k}}), 
\]
where $w=(dx_{\tau_0},(dx_\tau,dy_\tau)_\tau)$ is a basis of $\left(\Omega^1_{{\bf A}/X^{\cC}}\right)^{+,2}$. Seen as a Katz modular form $f(z)=f(A_z,\iota_z,\theta_z,\alpha_z,w)$, where $A_z:=V/\Lambda_z$, where $V:=\left((\C^2\otimes_{F,\tau_0}E)\times\M_2(\C)^{\Sigma_F\setminus\{\tau_0\}}\right)$, $\Lambda_z=\{v_z(\beta);\;\beta\in J_{\cC}:=\hat\cO_\cM b_{\cC}^{-1}\cap D\}$ and
\[
v_z(m\otimes e):=\left(\tau_0(m)\left(\begin{array}{c}z\\1\end{array}\right)\otimes e,(\tilde\tau(m\otimes e))_{\tau\neq\tau_0}\right)\in V,\qquad m\in B,\quad e\in E.
\]
The polarization $\theta_z$ is given by the paring
\[
\Theta_z:\Lambda_z\times\Lambda_z\longrightarrow J_{\cC}\times J_\cC\stackrel{\Theta}{\longrightarrow}\Z,\qquad
(v_z(\beta_1),v_z(\beta_1))\longmapsto \Theta(\beta_1,\beta_2),
\]
and $\alpha_z$ is the isomorphism (see Remark \ref{rmkonGammac}) $\alpha_z:(\Lambda_z\otimes\hat\Z\simeq\hat\cO_{\cM}b_{\cC}^{-1},\Theta_z)\rightarrow (\hat\cO_{\cM},\cC^{-1}\Theta)$ given by $\alpha_z(\beta)=\beta b_\cC$.
\begin{prop}\label{proponDelta}
For any $s\in \Delta$ and any $f\in M_{\underline{k}}(\Gamma_{1,1}^{\cC}(\cN),\C)$,
\[
(s\ast f)(A_z,\iota_z,\theta_z,\alpha_z,w)=s^{\frac{-\underline{k}+ 2k_{\tau_0}\tau_0+ \nu\underline{1}}{2}}\cdot f(A_z,\iota_z,s^{-1}\theta_z,k_s\alpha_z,w),
\]
where $k_s=b_{\cC}^{-1}\gamma_s^{-1}b_\cC$.
\end{prop}
\begin{rmk}
Notice that $\gamma_s^{-1}\alpha_z$ provides an isomorphism between $(\hat T A_z,\theta_z)$ and $(\hat\cO_{\cM},s\cC^{-1}\Theta)$, thus an isomorphism between $(\hat T A_z,s^{-1}\theta_z)$ and $(\hat\cO_{\cM},\cC^{-1}\Theta)$.
\end{rmk}
\begin{proof}
We have an isomorphism, $\tau_0(\gamma_s)=\left(\begin{array}{cc}a&b\\c&d\end{array}\right)$,
\[
\varphi_s: (A_{\gamma_s z},\iota_{\gamma_sz})\longrightarrow (A_z,\iota_z): \qquad V\ni(v,(M_\tau)_{\tau\neq\tau_0})\longmapsto ((cz+d)\cdot v,(M_\tau\tilde\tau(\gamma_s))_{\tau\neq\tau_0}).
\]
Indeed, given $v_{\gamma_s z}(m\otimes s)\in\Lambda_{\gamma_s z}$,
\[
\varphi_s(v_{\gamma_s z}(m\otimes s))=((cz+d)\cdot \tau_0(m)\left(\begin{array}{c}\gamma_s z\\1\end{array}\right)\otimes s,(\tilde\tau(m\otimes s)\tilde\tau(\gamma_s))_{\tau\neq\tau_0})=v_z(m\gamma_s\otimes s)\in\Lambda_z.
\]
The isomorphism $\varphi_s$ sends $\theta_{\gamma_s z}$ to $s^{-1}\theta_z$, since the corresponding pairing is given by
\[
\begin{array}{ccccccc}
\Lambda_z\times\Lambda_z&\longrightarrow& \Lambda_{\gamma_sz}\times\Lambda_{\gamma_s z}&\longrightarrow& J_{\cC}\times J_{\cC}&\stackrel{\Theta}{\longrightarrow}&\Z\\
(v_z(\beta),v_z(\beta'))&\longmapsto&(v_{\gamma_s z}(\beta\gamma_s^{-1}),v_{\gamma_s z}(\beta'\gamma_s^{-1}))&\longmapsto& (\beta\gamma_s^{-1},\beta'\gamma_s^{-1})&&
\end{array}
\]
Moreover, $\varphi_s$ sends $\alpha_{\gamma_s z}$ to $(b_{\cC}^{-1}\gamma_s^{-1}b_\cC)\alpha_z$, and $w$ to $((cz+d)^{-1}dx_{\tau_0},((dx_\tau,dy_\tau)\tilde\tau(\gamma_s)^{-1})_{\tau\neq\tau_0}):=(cz+d)^{-1}w\gamma_s^{-1}$. We conclude that
\begin{eqnarray*}
(s\ast f)(A_z,\iota_z,\theta_z,\alpha,w)&=&(s\ast f)(z)=s^{\frac{-\underline{k}+ 2k_{\tau_0}\tau_0+ \nu\underline{1}}{2}}(cz+d)^{-k_{\tau_0}}\gamma_s^{-1} f(\gamma_sz)\\
&=&s^{\frac{-\underline{k}+ 2k_{\tau_0}\tau_0+ \nu\underline{1}}{2}}\cdot(cz+d)^{-k_{\tau_0}}\gamma_s^{-1} f(A_{\gamma_sz},\iota_{\gamma_sz},\theta_{\gamma_sz},\alpha_z,w)\\
&\stackrel{(B1)}{=}&s^{\frac{-\underline{k}+ 2k_{\tau_0}\tau_0+ \nu\underline{1}}{2}}\cdot(cz+d)^{-k_{\tau_0}}\gamma_s^{-1} f(A_{z},\iota_{z},s^{-1}\theta_{z},k_s\alpha_z,(cz+d)^{-1}w\gamma_s^{-1})\\
&\stackrel{(B4)}{=}&s^{\frac{-\underline{k}+ 2k_{\tau_0}\tau_0+ \nu\underline{1}}{2}}\cdot f(A_z,\iota_z,s^{-1}\theta_z,k_s\alpha_z,w),
\end{eqnarray*}
and the result follows.
\end{proof}

\subsection{Moduli description of Hecke operators}
Let $(A,\iota,\theta,\alpha)\in X^{\cC}(R)$ and fix a representative $\bar\alpha:(\hat T(A),\theta)\rightarrow(\hat\cO_\cM,\cC^{-1}\Theta)$ of $\alpha$. For any $g\in G(\A_f)$, we consider the double coset $K_1^B(\cN)gK_1^B(\cN)=\bigsqcup_i g_iK_1^B(\cN)$. The preimage $\bar\alpha^{-1}(n^{-1}\hat\cO_\cM g_i^{-1})$, where $n\in\Z$ is big enough so that $ng_i\in\hat\cO_\cM$, provides a $\cO_\cM$-submodule $C_i\subset A_{tor}$ isomorphic to $n^{-1}\hat\cO_\cM g_i^{-1}/\hat\cO_\cM$. We define $A^{g_i}:=A/C_i$ together with an isogeny $\psi_{g_i}:(A^{g_i},\iota^{g_i})\rightarrow(A,\iota)$. We have a polarization given by the composition
\[
\theta^{g_i}:A^{g_i}\stackrel{\psi_{g_i}}{\longrightarrow}A\stackrel{\theta}{\longrightarrow}A^\vee\stackrel{\psi_{g_i}^\vee}{\longrightarrow}(A^{g_i})^\vee.
\]
Moreover, if $\det(g_i)=\cC_i$, the abelian variety $A^{g_i}$ comes equipped with an isomorphism
\[
\bar\alpha^{g_i}:(\hat T(A^{g_i}),\theta^{g_i})\stackrel{\bar\alpha}{\longrightarrow}(n^{-1}\hat\cO_\cM g_i^{-1},\cC^{-1}\Theta)\stackrel{\cdot g_i}{\longrightarrow}(\hat\cO_\cM,\cC_i^{-1}\cC^{-1}\Theta),
\]
providing a point $(A^{g_i},\iota^{g_i},\theta^{g_i},\alpha^{g_i})\in X^{\cC\cC_i}(R)$, where $\alpha_{g_i}$ denotes the class of $\bar\alpha^{g_i}$ modulo $K_{1,1}^B(\cN)$. By abuse of notation, write also $w=(dx_{\tau_0},(dx_\tau,dy_\tau)_\tau)\subset\Omega_{A^{g_i}/R}^1$ for the pull-back of the basis $w\subset\Omega_{A/R}^1$ by means of $\psi_{g_i}^\ast$. 

By Proposition \ref{compautshe}, any $f\in H^0(G(\Q),\cA(\underline{k},\nu))^{K_1^B(\cN)}$ can be seen as a vector
	\[
	(f_\cC)_{\cC}\in  \bigoplus_{\cC\in\Pic(\cO_K)}M_{\underline{k}}(\Gamma_{1,1}^\cC(\cN),\C)^\Delta=\bigoplus_{\cC\in\Pic(\cO_K)} H^0(X^{\cC,0},\omega^{\underline{k}})^\Delta.
	\]
Thus, this last space inherits the action of the Hecke operator $T_g$ attached to the double coset $K_1^B(\cN)gK_1^B(\cN)$. Recall that the isomorphism of Proposition \ref{compautshe} depends on the choice of elements $b_\cC\in G(\A_f)$ whose norm is a representative of the class $\cC$.	
\begin{prop}\label{HeckopKatz}
Assume that $b_{\cC}g_i=\gamma_i^{-1}b_{\cC'}k_i$, for some $\gamma_i\in G(\Q)_+$ and $k_i\in K_1^B(\cN)$. Then,
\begin{equation}\label{eqHeckops}
(T_g f)_\cC(A,\iota,\theta,\alpha,w)=\sum_i\det(\gamma_i)^{\frac{-\underline{k}+ 2k_{\tau_0}\tau_0+ \nu\underline{1}}{2}}f_{\cC'}(A^{g_i},\iota^{g_i},\det(\gamma_i)^{-1}\theta^{g_i},k_i^{-1}\alpha^{g_i},w),
\end{equation}
where $\cC'$ is the class $\cC\det(g)$.
\end{prop}
\begin{rmk}\label{rmkondefHecke}
We have the following remarks:
\begin{itemize}
\item Note that $\alpha^{g_i}$ provides an isomorphism between $(\hat T A^{g_i},\theta^{g_i})$ and $(\hat \cO_\cM,\cC^{-1}\cC_i^{-1}\Theta)$, hence $k_i^{-1}\alpha^{g_i}$ identifies $(\hat T A^{g_i},\theta^{g_i})$ and $(\hat \cO_\cM,\det(k_i)\cC^{-1}\cC_i^{-1}\Theta)=(\hat \cO_\cM,\det(\gamma_i)(\cC')^{-1}\Theta)$. We conclude that $k_i^{-1}\alpha^{g_i}$ provides an isomorphism between $(\hat T A^{g_i},\det(\gamma_i)^{-1}\theta^{g_i})$ and $(\hat \cO_\cM,(\cC')^{-1}\Theta)$.

\item The construction of $A^{g_i}$ depends on the representative $\bar\alpha$, but the expression in \eqref{eqHeckops} does not. Indeed, for any $\gamma\in K_{1,1}^B(\cN)$, we have $A_{\gamma}^{g_i}=A^{\gamma g_i}$, where $A_{\gamma}^{g_i}$ is the corresponding abelian variety constructed by means of $\gamma\bar\alpha$. But $\gamma g_i=g_j k_0$ for some $j$ and $k_0\in K^B_1(\cN)$. Hence, 
$(A_\gamma^{g_i},\iota_\gamma^{g_i},\det(\gamma_i)^{-1}\theta_\gamma^{g_i},k_i^{-1}(\gamma\bar\alpha)^{g_i})=(A^{g_j},\iota^{g_j},\det(\gamma_i)^{-1}\theta^{g_j},k_i^{-1}k_0\bar\alpha^{g_j})$,
since $(A^{g_jk_0},\iota^{g_jk_0},\theta^{g_jk_0})=(A^{g_j},\iota^{g_j},\theta^{g_j})$ and $(\gamma\bar\alpha)^{g_i}=k_0\bar\alpha^{g_j}$. Notice that we have
\[
\gamma_j^{-1}b_{\cC'}k_jk_0=b_{\cC}g_jk_0=b_{\cC}\gamma g_i=(b_{\cC}\gamma b_{\cC}^{-1}) b_\cC g_i=(b_{\cC}\gamma b_{\cC}^{-1}) \gamma_i^{-1}b_{\cC'}k_i,
\]
hence we can choose $\gamma_i$ and $\gamma_j$ such that $\det(\gamma_i)=\det(\gamma_j)$, and $k_i=k_jk_0$. Thus, we have $(A_\gamma^{g_i},\iota_\gamma^{g_i},\det(\gamma_i)^{-1}\theta_\gamma^{g_i},k_i^{-1}(\gamma\bar\alpha)^{g_i})=(A^{g_j},\iota^{g_j},\det(\gamma_j)^{-1}\theta^{g_j},k_j^{-1}\bar\alpha^{g_j})$ and the claim follows.

\item The expression \eqref{eqHeckops} does not depend on the choice of $\gamma_i$: If we have another decomposition $b_{\cC}g_i=\bar\gamma_i^{-1}b_{\cC'}\bar k_i$, then $\beta=\bar\gamma_i\gamma_i^{-1}\in G(\Q)_+\cap b_{\cC'}K_1(\cN)b_{\cC'}^{-1}=\Gamma_1^{\cC'}(\cN)$, $\bar k_i=b_{\cC'}^{-1}\beta b_{\cC'}k_i$, and, if we write $s=\det(\beta)$ and $k_s=b_{\cC'}^{-1}\beta^{-1} b_{\cC'}$, we obtain
\begin{eqnarray*}
&&\det(\bar\gamma_i)^{\frac{-\underline{k}+ 2k_{\tau_0}\tau_0+ \nu\underline{1}}{2}}f_{\cC'}(A^{g_i},\iota^{g_i},\det(\bar\gamma_i)^{-1}\theta^{g_i},\bar k_i^{-1}\alpha^{g_i},w)\\
&=&(s\det(\gamma_i))^{\frac{-\underline{k}+ 2k_{\tau_0}\tau_0+ \nu\underline{1}}{2}}f_{\cC'}(A^{g_i},\iota^{g_i},\det(\gamma_i)^{-1}s^{-1}\theta^{g_i}, k_i^{-1}k_s\alpha^{g_i},w)\\
&=&\det(\gamma_i)^{\frac{-\underline{k}+ 2k_{\tau_0}\tau_0+ \nu\underline{1}}{2}}f_{\cC'}(A^{g_i},\iota^{g_i},\det(\gamma_i)^{-1}\theta^{g_i}, k_i^{-1}\alpha^{g_i},w),
\end{eqnarray*}
by $\Delta$-invariance of $f_{\cC'}$.

\item The tuple $(A^{g_i},\iota^{g_i},\det(\gamma_i)^{-1}\theta^{g_i}, k_i^{-1}\alpha^{g_i})$ depend only on the class $g_i K^B_1(\cN)$. Indeed, if we replace $g_i$ by $g_ik$, for some $k\in K^B_1(\cN)$, we have trivially $(A^{g_ik},\iota^{g_ik},\theta^{g_ik})=(A^{g_i},\iota^{g_i},\theta^{g_i})$. Moreover, $\alpha^{g_ik}=k\alpha^{g_i}$, hence the claim easily follows.

\item If we assume the claim of the proposition, then $(T_g f)_{\cC}$ is $\Delta$-invariant. Indeed, given $s\in \Delta$ and $k_s\in K_1^B(\cN)$ with $\det(k_s)=s^{-1}$, we have
\begin{eqnarray*}
s\ast (T_gf)_\cC(A,\iota,\theta,\alpha,w)&=&s^{\frac{-\underline{k}+ 2k_{\tau_0}\tau_0+ \nu\underline{1}}{2}}\cdot  (T_gf)_\cC(A,\iota,s^{-1}\theta,k_s\alpha,w)\\
&=&\sum_i(s\det(\gamma_i))^{\frac{-\underline{k}+ 2k_{\tau_0}\tau_0+ \nu\underline{1}}{2}}f_{\cC'}(A^{g_i},\iota^{g_i},s^{-1}\det(\gamma_i)^{-1}\theta^{g_i},k_sk_i^{-1}\alpha^{g_i},w)\\
&=&(T_gf)_\cC(A,\iota,\theta,\alpha,w),
\end{eqnarray*}
by $\Delta$-invariance of $f_{\cC'}$.
\end{itemize}
\end{rmk}

\begin{proof}[Proof of Proposition \ref{HeckopKatz}]
As above $J_{\cC}=\hat\cO_\cM b_{\cC}^{-1}\cap D$, hence we have that 
\begin{eqnarray*}
(T_g f)_{\cC}(A_z,\iota_z,\theta_z,\alpha_z,w)&=&(T_gf)(z,b_\cC)=\sum_i f(z,b_\cC g_i)=\sum_i f(z,\gamma_i^{-1} b_{\cC'})\\
&=&\sum_i \det(\gamma_i)^{\frac{-\underline{k}+ 2k_{\tau_0}\tau_0+ \nu\underline{1}}{2}}(cz+d)^{-k_{\tau_0}}\gamma_i^{-1}f(\gamma_i z,b_{\cC'})\\
&=&\sum_i \det(\gamma_i)^{\frac{-\underline{k}+ 2k_{\tau_0}\tau_0+ \nu\underline{1}}{2}}(cz+d)^{-k_{\tau_0}}\gamma_i^{-1}f_{\cC'}(A_{\gamma_i z},\iota_{\gamma_i z},\theta_{\gamma_i z},\alpha_{\gamma_i z},w).
\end{eqnarray*}
where $A_z=V/\Lambda_{z}^{\cC}$, with $\Lambda_{z}^{\cC}:=\{v_z(\beta);\;\beta\in J_{\cC}\}$. Notice that we have an isomorphism
\[
\psi_i: A_{\gamma_iz}=V/\Lambda_{\gamma_i z}^{\cC'}\longrightarrow A_z^{g_i}=V/\Lambda_z^{\cC,g_i};\qquad \psi_i(v,(M_\tau)_{\tau\neq\tau_0})=\left((cz+d)v,(M_\tau\tilde\tau(\gamma_i))_{\tau\neq\tau_0}\right),
\]
where $\Lambda_z^{\cC,g_i}:=\{v_z(\lambda);\;\lambda\in J_{\cC}^{g_i}:=\hat\cO_\cM g_i^{-1}b_{\cC}^{-1}\cap D\}$. Notice that $\psi_i$ is well defined since $\psi_i(v_{\gamma_i z}(\beta))=v_z(\beta\gamma_i)$, for all $\beta\in J_{\cC'}$, and $\beta\gamma_i\in \hat\cO_\cM b_{\cC'}^{-1}\gamma_i\cap D=\hat\cO_\cM k_ig_i^{-1} b_{\cC}^{-1}\cap D=J_{\cC}^{g_i}$. Similarly as in the proof of Proposition \ref{proponDelta}, we have that $\psi_i$ sends $\theta_{\gamma_i z}$ to $\det(\gamma_i)^{-1}\theta_z^{g_i}$, it sends $\alpha_{\gamma_i z}$ to
\[
\xymatrix{
k_i^{-1}\alpha_z^{g_i}:\hat T(A_z^{g_i})\simeq \hat\cO_\cM g_i^{-1}b_{\cC}^{-1}\ar[rd]^{\qquad b_{\cC}}\ar[r]^{\quad\qquad\qquad\gamma_i^{-1}}&\hat\cO_\cM b_{\cC'}^{-1}\ar[r]^{b_{\cC'}}&\hat\cO_\cM,\\
&\hat\cO_\cM g_i^{-1}\ar[r]^{g_i}&\hat\cO_\cM\ar[u]^{k_i^{-1}}
}
\]
and it sends $w$ to $(cz+d)^{-1}w\gamma_i^{-1}:=((cz+d)^{-1}dx_{\tau_0},(dx_\tau,dy_\tau)\tilde\tau(\gamma_i)^{-1})$. We obtain $(T_g f)_{\cC}(A_z,\iota_z,\theta_z,\alpha_z,w)= $
\begin{eqnarray*}
&\stackrel{(B1)}{=}&\sum_i \det(\gamma_i)^{\frac{-\underline{k}+ 2k_{\tau_0}\tau_0+ \nu\underline{1}}{2}}(cz+d)^{-k_{\tau_0}}\gamma_i^{-1}f_{\cC'}(A_{z}^{g_i},\iota_{z}^{g_i},\det(\gamma_i)^{-1}\theta_{z}^{g_i},k_i^{-1}\alpha_z^{g_i},(cz+d)^{-1}w\gamma_i^{-1})\\
&\stackrel{(B4)}{=}&\sum_i \det(\gamma_i)^{\frac{-\underline{k}+ 2k_{\tau_0}\tau_0+ \nu\underline{1}}{2}}f_{\cC'}(A_{z}^{g_i},\iota_{z}^{g_i},\det(\gamma_i)^{-1}\theta_{z}^{g_i},k_i^{-1}\alpha_z^{g_i},w),
\end{eqnarray*}
hence the result follows.
\end{proof}

\subsection{The $U_\dP$-operator}\label{Upop}
Assume that $f$ has Iwahori level at $\dP\mid p$, and let $g=\left(\begin{array}{cc}\varpi&\\&1\end{array}\right)$, where $\varpi$ is a uniformizer at $\dP$. We denote by $U_\dP$ the Hecke operator $T_g$ as defined above. In this case the corresponding coset has a distinguished decomposition
\[
K_0^B(\dP)gK_0^B(\dP)=\bigsqcup_{i\in\cO_F/\dP}\left(\begin{array}{cc}\varpi&i\\&1\end{array}\right)K_0^B(\dP)\subset G(F_\dP).
\]
For any $\cC\in \Pic(\cO_F)$, assume that $\cC\dP$ lies in the class $\cC'$, this implies that $\cC\dP=\cC'(\gamma_\dP)$, for some $\gamma_\dP\in\cO_F$. Assuming that $\cC$ and $\cC'$ are coprime to $\dP$, $\gamma_\dP$ generates $\dP$ in $\cO_\dP$. Assuming that $b_{\cC}$ and $b_{\cC'}$ have identity components in $G(F_\dP)$, we have that
\[
b_{\cC}\left(\begin{array}{cc}\varpi&i\\&1\end{array}\right)=\left(\begin{array}{cc}\gamma_\dP&i\\&1\end{array}\right)b_{\cC'}k_i, \qquad k_i\in K^B_1(\cN).
\]
By the previous computations,
\[
(U_\dP f)_{\cC}(A,\iota,\theta,\alpha,w)=\gamma_\dP^{\frac{-\underline{k}+ 2k_{\tau_0}\tau_0+ \nu\underline{1}}{2}}\sum_{i\in \cO_F/\dP}\left(\begin{array}{cc}\gamma_\dP&i\\&1\end{array}\right)\ast f_{\cC'}\left(A_i,\iota_i,\gamma_\dP\theta_i, k_i^{-1}\alpha_i,w\left(\begin{array}{cc}\gamma_\dP&i\\&1\end{array}\right)\right),
\]
where $A_i=A/C_i$ (and $\iota_i,\theta_i, \alpha_i$ are also obtained from $C_i$) with 
\[
C_i\stackrel{\alpha}{\simeq}\hat\cO_\cM\left(\begin{array}{cc}\varpi&i\\&1\end{array}\right)^{-1}/\hat\cO_\cM=\left(\M_2(\cO_\dP)\left(\begin{array}{cc}\varpi&i\\&1\end{array}\right)^{-1}/\M_2(\cO_\dP)\right)\times \left(\M_2(\cO_\dP)\left(\begin{array}{cc}\varpi&i\\&1\end{array}\right)^{-1}/\M_2(\cO_\dP)\right).
\]
Thus $C_i$ is characterized by its subgroup $C_i^{-,1}:=C_i\cap A[\dP]^{-,1}$. Moreover, the subgroups $C_i^{-,1}$ correspond to the cyclic subgroups of $A[\dP]^{-,1}$ not intersecting with the canonical subgroup $C\subset A[\dP]^{-,1}$ that characterizes $\alpha_\dP$ after Remark \ref{rmkonpts}.

\begin{rmk}\label{rmkbasisUp}
Notice that, given a basis $B$ of $T_\dP A$, the isogeny $A\rightarrow A_i$ identifies $B\left(\begin{array}{cc}\varpi&i\\&1\end{array}\right)$ with a basis of $T_\dP(A_i)$. Moreover, if ${\rm dlog} (B)=w$, we have ${\rm dlog}\left(B\left(\begin{array}{cc}\varpi&i\\&1\end{array}\right)\right)=w\left(\begin{array}{cc}\varpi&i\\&1\end{array}\right)$.
\end{rmk}

\subsection{Oldforms}\label{oldforms}

Given a newform $\phi\in H^0(G(\Q),\cA(\underline{k},\nu))^{K_1(\cN_0)}$,  we can construct the oldform 
\[
\phi^d\in H^0(G(F),\cA(\underline{k},\nu))^{K_1(\cN)};\qquad \phi^d(f)(g):=\phi(f)(g g_d);\qquad g_d:=\left(\begin{array}{cc}1&\\ &\varpi_d\end{array}\right)\in G(\Q_d),
\]
where $\cN=\cN_0\cD$, $d\mid\cD$,
$\Q_d=\prod_{v\mid {\rm Norm}_{K/\Q}(d)}\Q_v$ and $\varpi_d$ is the product of uniformizers of $v\mid d$. 

We now give a new description of the oldform $\phi^d$ in terms of the Katz modular form interpretation. Recall that $\phi$ can be seen as a vector (Proposition \ref{compautshe}) 
\[
f=(f_\cC)_{\cC}\in  \bigoplus_{\cC\in\Pic(\cO_K)}M_{\underline{k}}(\Gamma_{1,1}^\cC(\cN_0),\C)^{\Delta_{\cN_0}}=\bigoplus_{\cC\in\Pic(\cO_K)} H^0(X_{\cN_0}^{\cC,0},\omega^{\underline{k}})^{\Delta_{\cN_0}},\qquad \Delta_{\cN_0}=(\cO_F)_+^\times/U^2_{\cN_0}.
\]
Hence, if $b_\cC g_d=\gamma_d^{-1}b_{\cC'}k$, for some $\gamma_d\in G(\Q)_+$ and $k\in K^B_1(\cN_0)$,
\begin{eqnarray*}
f_\cC^d(A_z,\iota_z,\theta_z,\alpha_z,w)&=&f^d(z,b_\cC)=f(z,b_\cC g_d)=\det(\gamma_d)^{\frac{-\underline{k}+ 2k_{\tau_0}\tau_0+ \nu\underline{1}}{2}}(cz+d)^{-k_{\tau_0}}\gamma_d^{-1}f(\gamma_d z,b_{\cC'})\\
&=&\det(\gamma_d)^{\frac{-\underline{k}+ 2k_{\tau_0}\tau_0+ \nu\underline{1}}{2}}(cz+d)^{-k_{\tau_0}}\gamma_d^{-1}f_{\cC'}(A_{\gamma_dz},\iota_{\gamma_dz},\theta_{\gamma_dz},\alpha_{\gamma_dz},w).
\end{eqnarray*}
We conclude by an analogous computation as above that 
\begin{equation}\label{eqoldformKatz}
f_\cC^d(A,\iota,\theta,\alpha,w)=\det(\gamma_d)^{\frac{-\underline{k}+ 2k_{\tau_0}\tau_0+ \nu\underline{1}}{2}}f_{\cC'}(A^{g_d},\iota^{g_d},\det(\gamma_d)^{-1}\theta^{g_d},k^{-1}\alpha^{g_d},w).
\end{equation}

	
\Addresses	
	
\bibliographystyle{hsiam}
\bibliography{biblio}	

\begin{thebibliography}{10}

\bibitem{AI17}
{\sc F.~Andreatta and A.~Iovita}, {\em Triple product $p$-adic $l$-functions
  associated to finite slope $p$-adic families of modular forms},  (2017).
\newblock Preprint.

\bibitem{AIP-siegel}
{\sc F.~Andreatta, A.~Iovita, and V.~Pilloni}, {\em {$p$}-adic families of
  {S}iegel modular cuspforms}, Ann. of Math.,  (2015).

\bibitem{AIP-halo}
\leavevmode\vrule height 2pt depth -1.6pt width 23pt, {\em Spectral halo},
  Annales scientifiques de l'{\'{E}}cole Normale Sup\'{e}rieure,  (2017).

\bibitem{BDR15}
{\sc M.~Bertolini, H.~Darmon, and V.~Rotger}, {\em Beilinson-{F}lach elements
  and {E}uler systems {II}: the {B}irch and {S}winnerton-{D}yer conjecture for
  {H}asse-{W}eil {$L$}-series}, Journal of Algebraic Geometry, 24 (2015),
  pp.~569--604.

\bibitem{BSV2}
{\sc M.~Bertolini, M.~Seveso, and R.~Venerucci}, {\em Diagonal classes and the
  {B}loch-{K}ato conjecture}, Munster Journal of Mathematics, Special volume in
  honor of {C}hristophe.

\bibitem{BPS}
{\sc S.~Bijakowski, V.~Pilloni, and B.~Stroh}, {\em Classicite de formes
  modulaires surconvergentes}, Annals of Mathematics,  (2016).

\bibitem{brasca13}
{\sc R.~Brasca}, {\em {$p$}-adic modular forms of non integral weight over
  {S}himura curves}, Compos. Math., 149(1) (2013), pp.~32--62.

\bibitem{Bump}
{\sc D.~Bump}, {\em Automorphic {F}orms and {R}epresentations}, Cambridge
  University Press, 1984.

\bibitem{carayol86}
{\sc H.~Carayol}, {\em {S}ur la mauvaise r\'{e}duction des courbes de
  {S}himura}, Compos. Math., 59 (1986), pp.~151--230.

\bibitem{Col97}
{\sc R.~F. Coleman}, {\em {$p$}-adic {B}anach spaces and families of modular
  forms}, Invent. Math., 127 (1997), pp.~417 -- 479.

\bibitem{DR14}
{\sc H.~Darmon and V.~Rotger}, {\em Diagonal cycles and euler systems {I}: A
  $p$-adic {G}ross-{Z}agier formula}, Annales scientifiques de l'{\'{E}}cole
  Normale Sup\'{e}rieure, 47(4) (2014), pp.~779--832.

\bibitem{DR17}
\leavevmode\vrule height 2pt depth -1.6pt width 23pt, {\em Diagonal cycles and
  euler systems {II}: the {B}irch and {S}winnerton-{D}yer conjecture for
  {H}asse-{W}eil-{A}rtin {$L$}-series}, Journal of the AMS, 30(3) (2017),
  pp.~601--672.

\bibitem{ding17}
{\sc Y.~Ding}, {\em Formes modulaires {$p$}-adiques sur les courbes de
  {S}himura unitaires et compatibilit\'{e} local-global}, M\'{e}moires de la
  Soc. Math. France, 155 (2017).

\bibitem{greenberg-seveso}
{\sc M.~Greenberg and M.~A. Seveso}, {\em Triple product {$p$}-adic
  {$L$}-functions for balanced weights}.
\newblock Preprint.

\bibitem{harris-kudla1}
{\sc M.~Harris and S.~Kudla}, {\em The central critical value of a triple
  product {$L$}-function}, Annals of Math.,  (1991).

\bibitem{harris-kudla2}
\leavevmode\vrule height 2pt depth -1.6pt width 23pt, {\em On a conjecture of
  {J}acquet}, 2004.

\bibitem{harris-tilouine01}
{\sc M.~Harris and J.~Tilouine}, {\em {$p$}-adic measures and square roots of
  special values of triple product {$L$}-functions}, Math. Annalen, 320 (2001),
  pp.~127--147.

\bibitem{hsieh18}
{\sc M.-L. Hsieh}, {\em Hida families and {$p$}-adic triple product
  {$L$}-functions}, American Journal of Mathematics,  (2018).

\bibitem{ichino08}
{\sc A.~Ichino}, {\em Trilinear forms anf the central values of triple product
  $l$-functions}, Duke Math., 145(2) (2008), pp.~281--307.

\bibitem{JSW15}
{\sc D.~Jetchev, C.~Skinner, and X.~Wan}, {\em The {B}irch and
  {S}winnerton-{D}yer formula for elliptic curves of analytic rank one},
  (2015).
\newblock Preprint.

\bibitem{Jon}
{\sc O.~Jones}, {\em An analogue of the bgg resolution for locally analytic
  principal series}, J. Number Theory, 131(9) (2011), pp.~1616--1640.

\bibitem{kassaei04}
{\sc P.~Kassaei}, {\em {$\mathcal{P}$}-adic modular forms over {S}himura curves
  over totally real fields}, Compos. Math., 140 (2004), pp.~359--395.

\bibitem{Kassaei06}
\leavevmode\vrule height 2pt depth -1.6pt width 23pt, {\em A gluing lemma and
  overconvergent modular forms}, Duke Math., 132(3) (2006), pp.~509--529.

\bibitem{Kassaei09}
\leavevmode\vrule height 2pt depth -1.6pt width 23pt, {\em Overconvergence and
  classicality: the case of curves}, J. Reine Angew. Math., 631 (2009),
  pp.~109--139.

\bibitem{KLZ17}
{\sc G.~Kings, D.~Loeffler, and S.~Zerbes}, {\em Rankin--{E}isenstein classes
  and explicit reciprocity laws}, Cambridge J. Math., 5 (2017).

\bibitem{LLZ16}
{\sc A.~Lei, D.~Loeffler, and S.~Zerbes}, {\em Euler systems for {H}ilbert
  modular surfaces},  (2016).
\newblock Preprint.

\bibitem{liu-zhang-zhang17}
{\sc Y.~Liu, S.~Zhang, and W.~Zhang}, {\em A {$p$}-adic walspurger formual},
  Duke Math., 167 (2018), pp.~743--833.

\bibitem{LZ}
{\sc D.~Loeffler and S.~Zerbes}, {\em Rankin--{E}isenstein classes in {C}oleman
  families}, Res. Math. Sci., 29 (Robert Coleman memorial volume) (2016).

\bibitem{Loke}
{\sc H.~Loke}, {\em Trilinear forms of {$\mathrm{gl}_2$} and local
  {$\epsilon$}-factors}, Pacific Journal of Mathematics,  (2001).

\bibitem{Pra90}
{\sc D.~Prasad}, {\em Trilinear forms for representations of {$\mathrm{GL}(2)$}
  and local {$\epsilon$}-factors}, Compositio Math.,  (1990).

\bibitem{SU14}
{\sc C.~Skinner and E.~Urban}, {\em The {I}wasawa {M}ain {C}onjectures for
  {GL}$(2)$}, Invent. Math., 195 (1) (2014), pp.~1--277.

\bibitem{Urb11}
{\sc E.~Urban}, {\em Eigenvarieties for reductive groups}, Ann. of Math., 174
  (2011), pp.~1695 -- 1784.

\bibitem{Urb2}
\leavevmode\vrule height 2pt depth -1.6pt width 23pt, {\em Nearly
  overconvergent modular forms}, Iwasawa theory 2012, Contrib. Math. Comput., 7
  (2014), pp.~401--441.

\bibitem{Vig}
{\sc M.-F. Vign{\'e}ras}, {\em Arithm\'etique des alg\`ebres de quaternions},
  vol.~800 of Lecture Notes in Mathematics, Springer, Berlin, 1980.

\end{thebibliography}
	
\end{document}